\documentclass[reqno,10pt,a4paper]{amsart}

\usepackage[margin=0.8in]{geometry}
\usepackage{parskip}
\usepackage{amsmath}
\usepackage{amsfonts}
\usepackage{amssymb}
\usepackage{amsthm}
\usepackage{xspace}
\usepackage{array}
\usepackage{url}
\usepackage[bookmarks=false]{hyperref}
\usepackage[dvipsnames]{xcolor}
\usepackage{mathrsfs}
\usepackage{mathtools}
\usepackage{scalerel}
\usepackage{enumitem}
\usepackage{titlesec}

\hypersetup
    {
        colorlinks    = true,
        linkcolor     = blue,     
        urlcolor      = SeaGreen,
    	citecolor     = JungleGreen,
    	bookmarksopen = false,
    }
\theoremstyle{plain}
\newtheorem{theorem}{Theorem}
\newtheorem{lemma}[theorem]{Lemma}
\newtheorem{proposition}[theorem]{Proposition}
\newtheorem{corollary}[theorem]{Corollary}
\newtheorem{example}[theorem]{Example}
\newtheorem{motivatingExample}[theorem]{Motivating Example}
\newtheorem{remark}[theorem]{Remark}

\theoremstyle{remark}

\numberwithin{equation}{section}
\numberwithin{theorem}{section}

\newcommand{\ie}{i.e.\@\xspace}

\newcommand{\R}{\mathbb{R}}
\newcommand{\N}{\mathbb{N}}

\newcommand{\tr}{\mathrm{tr}}
\renewcommand{\S}{\mathbb{S}}
\renewcommand{\P}{\mathbb{P}}

\DeclareMathOperator{\vol}{\mathrm{vol}}
\renewcommand{\div}{\operatorname{div}}

\newcommand{\two}{\mathrm{I\!I}}

\DeclareMathOperator{\hess}{Hess}
\DeclareMathOperator{\tracelessHess}{\mathring{Hess}}

\DeclareMathOperator{\scal}{scal}
\DeclareMathOperator{\ric}{Ric}

\DeclareMathOperator{\sn}{sn}
\newcommand{\snk}{\sn_k}

\newcommand*{\dd}{\mathop{}\!\mathrm{d}}

\newcommand{\dist}{\mathrm{dist}}

\titleformat{\section}{\normalfont\fontsize{12}{15}\bfseries}{\thesection.}{1em}{}

\titleformat{\subsection}[runin]{\bfseries}{\thesubsection.}{0.5em}{}
\titlespacing*{\subsection}{0pt}{2ex}{1em}

\begin{document}

\title[Sharp monotonicity for closed three-manifolds with positive scalar curvature]{Sharp monotonicity for closed three-manifolds with positive scalar curvature}

\author[Cosmin Manea]{Cosmin Manea}
\address{Department of Mathematics, Massachusetts Institute of Technology, 77 Massachusetts Avenue, Cambridge, MA, USA}
\email{cosmanea@mit.edu}

\begin{abstract}
    We prove two sharp monotonicity formulae for the natural Green's function of a closed three-manifold with positive scalar curvature, analogous to the Munteanu-Wang theorem for non-negative scalar curvature.
    We also show that both results may be interpreted as the monotonicity of a suitably defined ``enclosed mass'' of the sublevel sets of Green's function in the corresponding ``conformal'' blow-up of the manifold.
\end{abstract}

\maketitle

\section{Introduction} \label{sec: intro}

We develop the ``level set approach'' in the setting of closed three-manifolds with positive scalar curvature and extend the monotonicity formula involving the level sets of Green's function for the Laplacian on non-parabolic manifolds of Munteanu and Wang from {\cite{MunteanuWangComparisonPaper}}.
Instead of the Laplacian, we use a natural Schr\"{o}dinger operator with a constant, curvature-dependent potential, which acts as a comparison operator for the conformal Laplacian.
We will present two sharp monotonic quantities, for which the vanishing of each of their derivatives characterizes the round three-sphere.
Moreover, both quantities recover the Munteanu--Wang result in the limit of the curvature lower bound going to zero.
We will also discuss a way in which both formulae may be interpreted as the monotonicity of the ``enclosed mass'' of the sublevel sets of Green's function with respect to the ``conformal'' blow-up of the manifold.

\subsection{The first monotonicity formula.}
The monotonicity of the first quantity requires only{\footnote{Formally, the theorem also needs an assumption akin to the local integrability of the reciprocal of the gradient of Green's function, so that all integral terms are finite; we will discuss this slightly later.}} a positive lower bound on the scalar curvature, and takes the following form.

\begin{theorem} \label{intro-th: sharp monotonic quantity}
Let $(M, g)$ be a closed, connected Riemannian three-manifold with $\scal \geq 6k$ for some $k > 0$.
Fix a point $p \in M$ and let $G$ be Green's function{\footnote{Our convention is that, if $L$ is a positive-definite, second order differential operator on a closed three-manifold, then its Green's function $G$ with singularity at a point $p$ is normalized to solve the equation $LG = 4\pi \delta_p$, where $\delta_p$ is the Dirac distribution at $p$.}} for the operator $- \Delta + 3k/4$, with singularity at $p$.
Assume, in addition, that all regular level sets of $G$ are path-connected.

Then, for $s := G^{-1}$, we have{\footnote{In this paper, we will integrate both over level and sublevel/superlevel sets of continuous functions on a three-manifold.
The former integral is taken with respect to the two-dimensional Hausdorff measure, whereas the latter with respect to the three-dimensional Hausdorff measure.
We will suppress these measures from the notation.}},
\begin{equation} \label{intro-eq: sharp monotonic quantity}
    \frac{\dd}{\dd r} \left( \frac{1}{r^3} \int_{s = r} \|\nabla s\|^2 + \frac{3k}{2r^2} \int_{s \leq r} \|\nabla s\| + \frac{3k^2}{32 r^2} \int_{s \leq r} \frac{r^2 - s^2}{\|\nabla s\|} - \frac{4 \pi}{r} \right) \geq 0
\end{equation}
for almost every $r > 0$ in the range of $s$.
Moreover, if equality holds at some $r_0 > 0$, then $r_0 \leq 2/\sqrt{k}$ and $(M, g)$ is isometric to the three-sphere of constant curvature $k$.
\end{theorem}

To explain Theorem \ref{intro-th: sharp monotonic quantity}, let us start by considering the case of the three-sphere of curvature $k$.
If $\Delta$ is the non-positive Laplacian, the operator $-\Delta + 3k/4$ is the conformal Laplacian, and its Green's function $G$ may be explicitly computed as $G = (2 \snk (\dist(p, \cdot)/2))^{-1}$, where $\snk$ is the curvature-adapted sine function (cf.~{\cite[Lemma 2.1]{ManeaSGE}}); in particular, $s = 2 \snk(\dist(p, \cdot)/2)$.
In this case, one may then explicitly compute that
\[
\frac{1}{r^3} \int_{s = r} \|\nabla s\|^2 = \frac{4\pi}{r} \left(1 - \frac{k}{4} r^2 \right)^2, \qquad \frac{3k}{2r^2} \int_{s \leq r} \|\nabla s\| = 2 \pi k r - \frac{3 \pi k^2 r^3}{10}, \qquad \frac{3k^2}{32r^2} \int_{s \leq r} \frac{r^2 - s^2}{\|\nabla s\|} = \frac{\pi k^2 r^3}{20},
\]
and thus, for all $r \leq 2/\sqrt{k}$ (as the range of $s$ is $[0, 2/\sqrt{k}]$),
\begin{equation} \label{intro-eq: sharp monotonic quantity equality for the sphere}
    \frac{1}{r^3} \int_{s = r} \|\nabla s\|^2 + \frac{3k}{2r^2} \int_{s \leq r} \|\nabla s\| + \frac{3k^2}{32 r^2} \int_{s \leq r} \frac{r^2 - s^2}{\|\nabla s\|} = \frac{4 \pi}{r}.
\end{equation}

For a general closed three-manifold $(M, g)$ with $\scal \geq 6k$, the operator $- \Delta + 3k/4$ acts as a comparison operator for the conformal Laplacian, and it is equal to it when the scalar curvature is constant; this operator is also more suitable for computations due to the absence of pointwise curvature quantities in its potential term.
If, in addition, $\ric \geq 2kg$, the author previously proved (cf.~{\cite{ManeaSGE}}) a sharp gradient estimate for its Green's function, namely the inequality
\begin{equation} \label{intro-eq: sharp gradient estimate in positive ricci}
    \|\nabla s\|^2 + \frac{k}{4} s^2 \leq 1,
\end{equation}
with equality holding at \emph{some} point only when $(M, g)$ is isometric to the three-sphere of curvature $k$.
In this case, various sharp monotonic quantities involving the level and sublevel sets of $s \equiv G^{-1}$, analogous to those from Theorem \ref{intro-th: sharp monotonic quantity}, have also been investigated in the case of positive Ricci curvature, extending previous results of Colding (cf.~{\cite{ColdingNewMonotonicityFormulas}}) and Colding--Minicozzi (cf.~{\cite{ColdingMinicozziMonotonicityFormulas}}) in non-negative curvature.
From this point of view, Theorem \ref{intro-th: sharp monotonic quantity} is a positive-curvature adaptation of similar monotonicity formulae of Munteanu and Wang in non-negative scalar curvature (cf.~{\cite{MunteanuWangComparisonPaper}}); we will explain this relationship in the coming paragraphs.

To understand \eqref{intro-eq: sharp monotonic quantity}, the coarea formula allows us to rewrite
\[
\int_{s \leq r} \|\nabla s\| = \int_0^r \mathrm{area}(s = t) \dd t, \qquad \int_{s \leq r} \frac{r^2 - s^2}{\|\nabla s\|} = \int_0^r (r^2 - t^2) \left( \int_{s = t} \frac{1}{\|\nabla s\|^2} \right) \dd t,
\]
where $\mathrm{area}(s = t)$ denotes the two-dimensional Hausdorff measure of the level set $\{s = t\}$.
Thus, the monotonicity statement of Theorem \ref{intro-th: sharp monotonic quantity} may be written as
\[
\frac{\dd}{\dd r} \left( \frac{1}{r^3} \int_{s = r} \|\nabla s\|^2 + \frac{3k}{2r^2} \int_0^r \mathrm{area}(s = t) \dd t + \frac{3k^2}{32 r^2} \int_0^r (r^2 - t^2) \left( \int_{s = t} \frac{1}{\|\nabla s\|^2} \right) \dd t - \frac{4 \pi}{r} \right) \geq 0.
\]
Let us now discuss each integral term individually.
Observe first that each of them is non-negative.
The first integral term is just the energy of $s$ on its level sets, whilst the second term may be interpreted as the ``total area'' of the level sets of $s$.
The third term may be understood as a ``weighted total energy'' of the foliation of $M$ by the level sets of $s$.
Indeed, if $\{s = t\}$ is a regular level set of $s$, then there is some $\varepsilon > 0$ so that the flow of the vector field $V := \|\nabla s\|^{-2} \nabla s$ provides a diffeomorphism between $\{s = t\} \times (-\varepsilon, \varepsilon)$ and $\{t - \varepsilon < s < t + \varepsilon\} \subset M$.
In this neighbourhood, the metric $g$ takes the form $g = \|\nabla s\|^{-2} ds \otimes ds + g_s$, where $g_s$ denotes $g$ restricted to the level sets of $s$.
It follows that the term $\int_{s = t} \|\nabla s\|^{-2} = \int_{s = t} \|V\|^2$ is the ``energy'' of the level set $\{s = t\}$.
Finally, combined with \eqref{intro-eq: sharp monotonic quantity equality for the sphere}, the sum of these three integral terms represents an ``excess'' with respect to the corresponding ``comparison quantity'' on the sphere of curvature $k$.

\subsection{Proof and context discussion for Theorem \ref{intro-th: sharp monotonic quantity}.}
Together with the above, the idea behind Theorem \ref{intro-th: sharp monotonic quantity} comes from the quantitative expression of the corresponding derivative; we will see in Section \ref{sec: sharp monotonic quantity proof of monotonicity} that
\begin{equation} \label{intro-eq: quantitative derivative of sharp monotonic quantity}
    \begin{split}
        &\frac{\dd}{\dd r} \left( \frac{1}{r^3} \int_{s = r} \|\nabla s\|^2 + \frac{3k}{2r^2} \int_{s \leq r} \|\nabla s\| + \frac{3k^2}{32 r^2} \int_{s \leq r} \frac{r^2 - s^2}{\|\nabla s\|} - \frac{4 \pi}{r} \right) \\
        &= \frac{1}{r^3} \int_0^r \left[ 2\pi (2 - \chi(s = t)) + \int_{s = t} \left( \frac{1}{2} (\scal - 6k) + \frac{1}{2\|\nabla s^2\|^2} \|\!\tracelessHess{s^2} \|^2 \right) \right] \dd t,
    \end{split}
\end{equation}
where $\tracelessHess{s^2}$ denotes the traceless Hessian of $s^2$.
The assumption that the regular level sets of $G$ (and hence of $s$) are path-connected implies that $\chi(s = t) \leq 2$ for almost every $t$ in the range of $s$; moreover, $\scal \geq 6k$, and thus the right-hand side of \eqref{intro-eq: quantitative derivative of sharp monotonic quantity} is non-negative.
The rigidity statement is also immediate from the above computation: if the right-hand side above is zero for some $r > 0$, then $\tracelessHess{s^2} = 0$ on $\{s < r\}$; the rigidity of warped products combined with a Positive Mass Theorem argument for the ``conformal`` blow-up manifold $(M \setminus \{ p \}, G^4 g)$ then show that $(M, g)$ must be isometric to the three-sphere of constant curvature $k$.
The right-hand side of \eqref{intro-eq: quantitative derivative of sharp monotonic quantity} thus measures explicitly the ``geometric excess'' of $M$ with respect to the sphere of curvature $k$ from the point of view of the scalar curvature of $M$ and the geometry of the foliation by the level sets of $s$.

The proof of \eqref{intro-eq: quantitative derivative of sharp monotonic quantity} employs a rewriting of the Bochner formula for $\|\nabla s\|^2$ along the level sets of $s$ using the traced Gauss equation; this will essentially be an extension of the identity from {\cite[Proposition 3.6]{ManeaSGE}} that was used to prove the sharp gradient estimate for $s$ in positive Ricci curvature (cf.~\eqref{intro-eq: sharp gradient estimate in positive ricci}), and it is analogous to {\cite[Proposition 1.17]{CMGradientEsimatesScalarCurvature}}.
We will see in Section \ref{sec: sharp monotonic quantity proof of monotonicity} (cf.~Proposition \ref{prop: divergence on regular level sets for sharp monotonic quantity}) that if $\Sigma$ is a regular level set of $s$, then
\begin{equation} \label{intro-eq: bochner formula for sharp monotonic quantity}
    \begin{split}
        \frac{1}{2\|\nabla s\|} \div \left( \frac{1}{\|\nabla s\|} \nabla \left( \|\nabla s\|^2 + \frac{k}{4} s^2 \right) \right) &= - \frac{1}{2} \scal_\Sigma + \frac{1}{2} (\scal - 6k) + \frac{1}{2\|\nabla s^2\|^2} \|\!\tracelessHess{s^2}\|^2 \\
        &\quad \, + \frac{1}{s^2 \|\nabla s\|^2} \left( \|\nabla s\|^2 + \frac{k}{4} s^2 \right)^2 \\
        &\quad \, + \left( \frac{1}{2s^2} + \frac{k}{4\|\nabla s\|^2} \right) \tracelessHess{s^2} \left( \frac{\nabla s}{\|\nabla s\|}, \frac{\nabla s}{\|\nabla s\|} \right),
    \end{split}
\end{equation}
where $\scal_\Sigma$ denotes the scalar curvature of $\Sigma$ in the induced metric; in the limit as $k \downarrow 0$, we recover precisely the result from {\cite[Proposition 1.17]{CMGradientEsimatesScalarCurvature}}.
Theorem \ref{intro-th: sharp monotonic quantity} is now implied by computing the integral of $\div ( \|\nabla s\|^{-1} \nabla ( \|\nabla s\|^2 + ks^2/4 ))$ over the subset $\{s \leq r\}$ in two different ways: one by using the Divergence Theorem, and one by applying the coarea formula and \eqref{intro-eq: bochner formula for sharp monotonic quantity}.
The last two terms on the right-hand side of \eqref{intro-eq: bochner formula for sharp monotonic quantity} account for the integral terms over the sublevel sets of $s$ from the statement of Theorem \ref{intro-th: sharp monotonic quantity}.

There are two other topics that we must discuss related to the hypotheses of Theorem \ref{intro-th: sharp monotonic quantity}: the assumption of path-connectedness of the level sets of $G$, and (as mentioned earlier) the finiteness of each integral term from \eqref{intro-eq: sharp monotonic quantity}.
Before we proceed with this, let us first contextualize Theorem \ref{intro-th: sharp monotonic quantity}.
Monotonicity formulae play an important role in geometry and analysis because of their wide range of geometric implications.
Classical examples include the Bishop-Gromov relative volume comparison theorem for the monotonicity of the ratio of volumes of balls and Perelman's monotonicity formulae for the Ricci flow.

In the setting of monotonicity for Green's functions, Colding and Minicozzi (cf.~{\cite{ColdingNewMonotonicityFormulas, ColdingMinicozziMonotonicityFormulas}}) obtained monotonic quantities along the level sets of the minimal, positive Green's function for the Laplacian on non-parabolic manifolds with non-negative Ricci curvature, which were used, for instance, in the study of the uniqueness of tangent cones at infinity for Einstein manifolds (cf.~{\cite{CMUniquenessOfTangentCones}}).
Similar formulae were also obtained in the realm of potential theory (cf.~{\cite{MonotonicityInPotentialTheory, MonotonicityAndGeometricInequalitiesForHypersurfacesInNonNegativeRicci, MinkowskiIneqViaPotentialTheory}}) and for the $p$-Laplace operator (cf.~{\cite{MonotinicityPGreensFunction}}), and were used to prove various geometric inequalities for manifolds with non-negative Ricci curvature.

In recent years, monotonic quantities involving the level sets of functions satisfying special (elliptic) equations have been studied extensively to understand the geometry of three-manifolds with non-negative scalar curvature.
A key development in this direction was Stern's Bochner formula (cf.~{\cite{SternLevelSetScalarCurvature}}), which was used to obtain an integral inequality involving the scalar curvature of closed three-manifolds and harmonic maps to the circle.
Since then, other monotonic formulae along the level sets of harmonic functions have led to new results for non-compact three-manifolds with non-negative scalar curvature.
Of particular importance is the work of Munteanu and Wang (cf.~{\cite{MunteanuWangComparisonPaper}}), where they showed that if $\widetilde{G}$ is the minimal, positive Green's function for the Laplacian of a non-parabolic three-manifold with non-negative scalar curvature, then, for $\widetilde{s} := \widetilde{G}^{-1}$,
\begin{equation} \label{intro-eq: MW sharp monotonicity for greens function}
    \frac{\dd}{\dd r} \left( \frac{1}{r^3} \int_{\widetilde{s} = r} \|\nabla \widetilde{s}\|^2 - \frac{4\pi}{r} \right) \geq 0
\end{equation}
if $\widetilde{G} \to 0$ at infinity and the regular level sets of $\widetilde{G}$ are path-connected; moreover, equality holds at some $r > 0$ only if the subset $\{\widetilde{s} < r\}$ is isometric to a Euclidean ball of radius $r$.
It is now apparent that the result of Theorem \ref{intro-th: sharp monotonic quantity} specializes to the one from \eqref{intro-eq: MW sharp monotonicity for greens function} in the limit $k \downarrow 0$.
For this reason, we chose to express Theorem \ref{intro-th: sharp monotonic quantity} in this form, instead of, for instance, normalizing to $k = 1$.

Munteanu and Wang used \eqref{intro-eq: MW sharp monotonicity for greens function} to obtain various comparison and geometric inequalities on manifolds with a (possibly negative) scalar curvature lower bound (cf.~{\cite{MunteanuWangGeometryOfPSC, MunteanuWangBottomSpectrumComparison, MunteanuWangIntegralScalarCurvatureBound, MunteanuWangParabolicityInPSC}}), such as comparisons for the bottom of the spectrum, sharp integral bounds for the scalar curvature, and non-parabolicity of open manifolds with uniformly positive scalar curvature.
Colding and Minicozzi (cf.~{\cite{CMGradientEsimatesScalarCurvature}}) also recently obtained the ``averaged'' sharp gradient bound
\begin{equation} \label{intro-eq: averaged gradient estimate in scalar curvature}
    \int_{\widetilde{s} = r} \|\nabla \widetilde{s} \|^2 \leq 4 \pi r^2,
\end{equation}
generalizing the pointwise estimate $\|\nabla \widetilde{s}\| \leq 1$ previously obtained in {\cite{ColdingNewMonotonicityFormulas}} in non-negative curvature.

The ``level set approach'' from above has also been used in the realm of positive mass theorems.
For example, harmonic functions have been used to prove lower bounds for the ADM mass of (the exterior regions of) asymptotically flat three-manifolds in {\cite{BrayKazarasHarmonicFunctionsAndMass}}, and a modified notion of ``spacetime harmonic'' functions has also given a new proof of the Spacetime Positive Mass Theorem (cf.~{\cite{BrayKazarasSpacetimeHarmonicFunctionsAndMass, BrayKazarasKhuriPerspectives}}).
Similar monotonic quantities have also been employed to provide other new proofs of the Positive Mass Theorem (cf.~{\cite{PMTViaGreenFunction}}) and the Riemannian Penrose Inequality (cf.~{\cite{PenroseIneqViaGreenFunction}}), as well as show a Positive Mass Theorem for asymptotically hyperbolic manifolds (cf.~{\cite{PMTAsymptoticallyHypManifolds}}), a generalized ``$X$-Positive Mass Theorem'' (cf.~{\cite{XADMMassTheorem}}), and prove relationships between the ADM mass of asymptotically flat three-manifolds with boundary and the area and capacity of their boundary (cf.~{\cite{OronzioADMMassAreaCapacity, AreaCapacityIneqViaPotentialTheory}}).
Finally, more recently, for compact manifolds with boundary and positive scalar curvature, a one-parameter family of Positive Mass Theorems has been shown using $p$-Green's functions for the Laplace operator for $p \in (1, 3)$ (cf.~{\cite{PMTCompactPSCManifolds}}).
In this sense, our result from Theorem \ref{intro-th: sharp monotonic quantity} is a continuation of the ``level set approach'' programme, focusing on closed manifolds with positive scalar curvature.

\subsection{The path-connectedness of the level sets and finiteness of the quantity from \eqref{intro-eq: sharp monotonic quantity}.}
Let us now return to discussing the assumptions in Theorem \ref{intro-th: sharp monotonic quantity}.
All the previously described results in non-negative scalar curvature were obtained under the assumption of path-connectedness of the (regular) level sets of the functions in question.
For non-compact manifolds, this has been shown to follow from additional topological assumptions on the manifolds.
For instance, Munteanu and Wang proved in {\cite{MunteanuWangComparisonPaper}} that, in the setting of \eqref{intro-eq: MW sharp monotonicity for greens function}, if the manifold has only one end and vanishing first Betti number, then all level sets of $\widetilde{G}$ are path-connected.
For asymptotically flat manifolds with compact and minimal boundary (such as in the setting of {\cite{BrayKazarasHarmonicFunctionsAndMass}}), the connectedness of the level sets is implied essentially by the vanishing of the second relative homology group of the manifold.

In the setting of closed manifolds with positive scalar curvature from Theorem \ref{intro-th: sharp monotonic quantity}, the path-connectedness of the (regular) level sets of Green's function must be assumed separately, as no topological condition may ensure this property.
We will show in Appendix \ref{appendix-sec: connectedness of the level sets of greens functions} that every closed three-manifold which admits Riemannian metrics of positive scalar curvature also admits a sequence $(g_n)_{n \in \N}$ of metrics with scalar curvature bounded below by $6k$ for which Green's function for the operator $-\Delta_{g_n} + 3k/4$ has at least one level set with at least $n$ path-connected components.
This will be a consequence of the Gromov-Lawson-Schoen-Yau results regarding the existence of metrics of positive scalar curvature on connected sums of manifolds (cf.~{\cite{SchoenYauPSC, GromovLawsonPSC}}).
Concretely, we will show that, for two manifolds $X$ and $Y$ with positive scalar curvature, since the metric on the connected sum $X \# Y$ on the gluing region may be chosen to be isometric to a cylinder of the form $[0, \ell] \times \S^2_K$ where $\S^2_K$ denotes the two-sphere of curvature $K \geq 3k$ and $\ell > 0$, Green's function for the operator $- \Delta + 3k/4$ on $X \# Y$ with singularity at a point in $X$ decays exponentially in $\ell$ on the interior region of $Y$.

The conclusion of Theorem \ref{appendix-th: arbitrarily many connected components on every psc topology} then follows by applying this result to the connected sum of a manifold with $n$ three-spheres for each $n \geq 1$, as such surgery operations do not change the topology of the manifold.
As previously stated, this means that no topological condition may imply the connectedness of the (regular) level sets of such Green's functions.
Moreover, a geometric condition alone will also not guarantee this property, as even Green's functions of spherical manifolds may have disconnected level sets (such as $\S^3/Q_8$, cf.~Example \ref{appendix-example: non-connectedness on level sets on constant curvature spaces}).
Hence, one would hope that, to ensure the path-connectedness of the level sets of Green's function, we would need both a topological and a geometric assumption on a closed three-manifold with positive scalar curvature.
However, it is unclear what such assumptions one could take; for this reason, we decided to add the connectedness condition as an assumption in Theorem \ref{intro-th: sharp monotonic quantity}.

The last point of discussion about Theorem \ref{intro-th: sharp monotonic quantity} is the well-definedness of the monotonic quantity from its statement.
More precisely, the first two integral terms from \eqref{intro-eq: sharp monotonic quantity} are always finite and locally absolutely continuous for $r$ in the range of $s$, but the well-definedness of the third one requires some local integrability for $\|\nabla s\|^{-1}$.
The asymptotics of Green's function near its singularity imply that $\|\nabla s\|^{-1}$ is integrable near $p$, and thus Theorem \ref{intro-th: sharp monotonic quantity} is well-posed for $r > 0$ small enough.
However, a priori, $\|\nabla s\|^{-1}$ may not be integrable on all of $M$, and the corresponding term from \eqref{intro-eq: sharp monotonic quantity} may be infinite for some $r > 0$.

We do not have a general criterion for the finiteness of the third integral term.
The ``simple'' condition of imposing $\|\nabla s\|^{-1} \in L^1(M)$ as a hypothesis of Theorem \ref{intro-th: sharp monotonic quantity} is rather unnatural and too strong, as one may check explicitly that $\|\nabla s\|^{-1}$ is not integrable on $M$ when $k = 1$ and $(M, g) \simeq \R \P^3$ with the unit, round metric from the three-sphere.
However, for generic metrics with positive scalar curvature, $s$ is a Morse function on $M \setminus \{ p \}$, and hence $\|\nabla s\|^{-1}$ is integrable on $M$.
We will show this in Appendix \ref{appendix-sec: morse theory for greens function and integrability of third integral term}, where we will also present some divergence conditions and show a finiteness criterion for when the Riemannian manifold is real-analytic.
We conjecture that this integral term is always finite and locally absolutely continuous for $r$ in the range of $s$, but we do not have a proof of this in full generality.

\subsection{The second monotonicity formula.}
The above concludes the discussion about Theorem \ref{intro-th: sharp monotonic quantity}.
Let us now present the second monotonic quantity, which is expressed as follows.

\begin{theorem} \label{intro-th: alternative sharp monotonic quantity}
Let $(M, g)$ be a closed, connected Riemannian three-manifold with $\scal \geq 6k$ for some $k > 0$.
Fix a point $p \in M$ and let $G$ be Green's function for the operator $- \Delta + 3k/4$, with singularity at $p$.
Assume, in addition, that all regular level sets of $G$ are path-connected and that $G \geq \sqrt{k}/2$ on $M \setminus \{ p \}$.

Then, for $s := G^{-1}$, we have
\allowdisplaybreaks
\begin{align*}
    \frac{\dd}{\dd r} \bigg( \frac{1}{r^3} &\int_{s = r} \|\nabla s\|^2 + \int_{s \leq r} \frac{3k(2 + ks^2 - 6 r^{-2} s^2)}{s^2(4 - ks^2)^2} \|\nabla s\|^3 \\
    &+ \int_{s \leq r} \left( \frac{3k}{4} - \frac{3k(r^2 - s^2)(4 + 5k s^2)}{8r^2(4 - ks^2)} \right) \frac{\|\nabla s\|}{s^2} - \frac{4 \pi}{r} \bigg) \geq 0,
\end{align*}
for almost every $r \in (0, 2/\sqrt{k}]$.
Moreover, if equality holds at some $r < 2/\sqrt{k}$, then $(M, g)$ is isometric to the three-sphere of constant curvature $k$.
\end{theorem}

As for Theorem \ref{intro-th: sharp monotonic quantity}, we will show a quantitative version of Theorem \ref{intro-th: alternative sharp monotonic quantity} (similar to \eqref{intro-eq: quantitative derivative of sharp monotonic quantity}), which will prove that the derivative is zero if $(M, g)$ is isometric to the three-sphere of curvature $k$.
However, it seems simpler to show this latter property directly from this quantitative computation of the derivative rather than by a concrete computation of each integral term (as in \eqref{intro-eq: sharp monotonic quantity equality for the sphere}); additionally, Theorem \ref{intro-th: alternative sharp monotonic quantity} again recovers the Munteanu--Wang theorem (cf.~{\cite{MunteanuWangComparisonPaper}} and \eqref{intro-eq: MW sharp monotonicity for greens function}) in the limit as $k \downarrow 0$.

The idea of the proof of Theorem \ref{intro-th: alternative sharp monotonic quantity} is analogous to the argument presented in \eqref{intro-eq: bochner formula for sharp monotonic quantity}: rewriting the Bochner formula for $\|\nabla s\|^2$ using the traced Gauss equation along the regular level sets of $s$.
However, for this theorem, we will work entirely in terms of $G$ instead of $s$; this way, it will be more apparent how the additional assumption of the curvature-dependent lower bound for $G$ is used (cf.~Proposition \ref{prop: sharp bochner identity via G for alternative monotonic quantity}).
The computations will be more involved, but the argument will be more in the spirit of that of Munteanu and Wang from {\cite{MunteanuWangComparisonPaper}}.

The motivation behind Theorem \ref{intro-th: alternative sharp monotonic quantity} is the following.
We described earlier that, due to the flexibility of positive scalar curvature given by the Gromov-Lawson-Schoen-Yau surgery results, the connected sum of two closed manifolds with positive scalar curvature also admits positive scalar curvature, with the metric on the gluing region being isometric to a two-sphere cylinder of arbitrarily long length; this is one of the reasons why it is difficult to control the geometry of positive scalar curvature.
From the point of view of Green's functions, as mentioned before, Appendix \ref{appendix-sec: connectedness of the level sets of greens functions} shows that Green's function for our comparison operator converges to zero on such cylinders.
Thus, imposing a lower bound on Green's function is one way to prevent such long cylinders from existing in the manifold.
This also induces further geometric restrictions on $(M, g)$; for instance, testing the Green's function equation with a non-zero constant function gives the volume upper bound $3 k^{3/2} \vol(M, g) \leq 32 \pi$.

The curvature-dependent lower bound $G \geq \sqrt{k}/2$ from the statement of Theorem \ref{intro-th: alternative sharp monotonic quantity} is always satisfied if we instead assume the stronger curvature lower bound $\ric \geq 2k g$ (cf.~{\cite[Corollary 2.3]{ManeaSGE}}).
Thus, if such a lower bound holds, Theorem \ref{intro-th: alternative sharp monotonic quantity} provides an additional monotonic quantity to that from Theorem \ref{intro-th: sharp monotonic quantity}.
Compared to the quantity from Theorem \ref{intro-th: sharp monotonic quantity}, all integral terms from the statement of Theorem \ref{intro-th: alternative sharp monotonic quantity} are finite and locally absolutely continuous for $r$ in the range of $s$ (this is proved in the second part of Appendix \ref{appendix-sec: integration over the (super)level sets of greens function}); at the same time, it is not clear if the latter two are always non-negative.
However, the hypothesis that the (regular) level sets of $G$ are path-connected must still be imposed, as even in this case, spherical manifolds may still have Green's functions with disconnected level sets (cf.~Example \ref{appendix-example: non-connectedness on level sets on constant curvature spaces}).

\subsection{The enclosed mass point of view.}
Lastly, let us end this section by describing how both Theorems \ref{intro-th: sharp monotonic quantity} and \ref{intro-th: alternative sharp monotonic quantity} may be viewed as the monotonicity of the ``enclosed mass'' of the superlevel sets of $s$ in the ``conformal'' blow-up $(M \setminus \{p\}, G^4 g)$.
Let us give the explanation for this for Theorem \ref{intro-th: sharp monotonic quantity}, as the discussion for Theorem \ref{intro-th: alternative sharp monotonic quantity} is analogous; this will be the content of Section \ref{sec: sharp monotonic quantity applications}.

By the formula for the curvature under a conformal change of the metric, the scalar curvature of the manifold $(M \setminus \{p\}, G^4 g)$ is equal to $G^{-4}(\scal - 6k)$; this is non-negative by assumption, and equal to zero if and only if $\scal \equiv 6k$.
Since the dimension of $M$ is equal to three, the arguments from {\cite[Lemmas 6.4 and Theorem 6.5]{LeeParkerYamabeProblem}} may readily be adapted to our Green's function and show that $G$ has the asymptotic expansion $G = \dist(p, \cdot)^{-1} + A + O(\dist(p, \cdot))$ for a constant $A \in \R$ and that the manifold $(M \setminus \{p\}, G^4 g)$ is complete and asymptotically flat.
{\cite[Lemma 9.7]{LeeParkerYamabeProblem}} then shows that $A$ is equal to half the ADM mass of $(M \setminus \{p\}, G^4 g)$, which we denote by $m(M, p)$.
The Positive Mass Theorem tells us that $m(M, p) \geq 0$, and $m(M, p) = 0$ if and only if $(M \setminus \{p\}, G^4 g)$ is isometric to $\R^3$, which in turn holds if and only if $(M, g)$ is isometric to the three-sphere of curvature $k$.

The aforementioned asymptotics of $G$ near $p$ then imply that, as $r \downarrow 0$,
\[
\frac{1}{r^3} \int_{s = r} \|\nabla s\|^2 = \frac{4 \pi}{r} - 4 \pi m(M, p) + O(r)
\]
and that
\[
\lim_{r \downarrow 0} \left( \frac{3k}{2r^2} \int_{s \leq r} \|\nabla s\| \right) = \lim_{r \downarrow 0} \left( \frac{3k^2}{32r^2} \int_{s \leq r} \frac{r^2 - s^2}{\|\nabla s\|} \right) = 0.
\]
It then follows that
\begin{equation} \label{intro-eq: asymptotics of sharp monotonic quantity near the pole gives the mass}
    - \frac{1}{4\pi} \lim_{r \downarrow 0} \left( \frac{1}{r^3} \int_{s = r} \|\nabla s\|^2 + \frac{3k}{2r^2} \int_{s \leq r} \|\nabla s\| + \frac{3k^2}{32 r^2} \int_{s \leq r} \frac{r^2 - s^2}{\|\nabla s\|} - \frac{4 \pi}{r} \right) = m(M, p).
\end{equation}
Thus, up to a constant, the quantity from \eqref{intro-eq: sharp monotonic quantity} may be interpreted as the ``enclosed mass'' (or ``localized mass'') of the ``interior region'' $\{s \geq r\}$, and Theorem \ref{intro-th: sharp monotonic quantity} as the monotonicity of this enclosed mass.
Furthermore, we can integrate \eqref{intro-eq: quantitative derivative of sharp monotonic quantity} on an interval containing zero to obtain an integral formula for the difference between $m(M, p)$ and this enclosed mass (cf.~\eqref{align: mass and quasilocal mass difference via sharp monotonic quantity}).
A similar point of view may also be employed for Theorem \ref{intro-th: alternative sharp monotonic quantity}.

\subsection{Organization of the paper.} In Sections \ref{sec: sharp monotonic quantity proof of monotonicity} and \ref{sec: alternative monotonic quantity proof of monotonicity}, we provide the proofs of Theorems \ref{intro-th: sharp monotonic quantity} and \ref{intro-th: alternative sharp monotonic quantity}, respectively.
Section \ref{sec: sharp monotonic quantity applications} is dedicated to explaining the enclosed mass point of view for both theorems.
In Appendix \ref{appendix-sec: integration over the (super)level sets of greens function}, we discuss the asymptotics of Green's function near its singularity and prove that the integral terms from \eqref{intro-eq: sharp monotonic quantity} in Theorem \ref{intro-th: sharp monotonic quantity} are locally absolutely continuous when finite, and that those in Theorem \ref{intro-th: alternative sharp monotonic quantity} are always finite and locally absolutely continuous.
Appendix \ref{appendix-sec: connectedness of the level sets of greens functions} contains the discussion about the decay of Green's function on long two-sphere cylinders and the non-path-connectedness of its level sets.
Lastly, in Appendix \ref{appendix-sec: morse theory for greens function and integrability of third integral term}, we discuss the finiteness of the third integral term from Theorem \ref{intro-th: sharp monotonic quantity}, presenting a few divergence and convergence criteria and also examining the case of real-analytic metrics.

\textbf{Acknowledgements}. I would like to thank my advisor, Tobias Colding, for suggesting that I pursue this project and for many valuable discussions.
I thank Ovidiu Munteanu for an insightful conversation about {\cite{MunteanuWangComparisonPaper}}.
I am also grateful to Dongyeong Ko and Tristan Ozuch for numerous valuable discussions about the geometry of positive scalar curvature and the results of this manuscript, and to Matthew Gursky for suggesting that one may stratify the critical points by the rank of the Hessian.
Lastly, I thank Tobias Colding, Dongeyong Ko, and Tristan Ozuch for reading a preliminary version of this manuscript and for their helpful comments and suggestions.
During this project, I was also partially supported by NSF DMS Grant 2405393.
\section{The proof of Theorem \ref{intro-th: sharp monotonic quantity}} \label{sec: sharp monotonic quantity proof of monotonicity}

In this section, we prove Theorem \ref{intro-th: sharp monotonic quantity}.
We will assume that all integral terms in \eqref{intro-eq: sharp monotonic quantity} are finite; the proof that the quantity is locally absolutely continuous is provided in Appendix \ref{appendix-sec: integration over the (super)level sets of greens function}.
Throughout the rest of this section, let $(M^3, g)$ be a closed, connected Riemannian three-manifold with $\scal \geq 6k$ for some $k > 0$, fix a point $p \in M$, and let $G$ be Green's function for the operator $- \Delta + 3k/4$, with singularity at $p$.
By the strong maximum principle, $G$ is positive on $M \setminus \{ p \}$.
Assume furthermore that all level sets of $G$ are path-connected, and put $s := G^{-1}$.

As mentioned in Section \ref{sec: intro}, the proof of Theorem \ref{intro-th: sharp monotonic quantity} relies on rewriting the Bochner formula for $\|\nabla s\|^2$ using the traced Gauss equation along the regular level sets of $s$ to obtain \eqref{intro-eq: bochner formula for sharp monotonic quantity}.
The main ingredient is the following identity, which was used to prove a sharp gradient estimate for $s$ in {\cite[Theorem 1.1]{ManeaSGE}}.
This identity has already been proved in {\cite[Proposition 3.6]{ManeaSGE}}, but, for convenience, we also describe the argument below.

\begin{proposition} \label{prop: drift bochner formula for sharp gradient quantity}
The following identity holds on $M \setminus \{p\}$:
\begin{equation} \label{eq: drift bochner formula for sharp gradient quantity}
   \Delta \left( \|\nabla s\|^2 + \frac{k}{4} s^2 \right) - \frac{1}{2s^4} \tracelessHess{s^2}(\nabla s^2, \nabla s^2) = \frac{1}{2s^2} \left( \|\!\tracelessHess{s^2}\|^2 + \ric(\nabla s^2, \nabla s^2) - 2k \|\nabla s^2\|^2 \right),
\end{equation}
where $\tracelessHess{s^2}$ denotes the traceless Hessian of $s^2$.
\end{proposition}

\begin{proof}
One could directly apply the Bochner formula for $\|\nabla s\|^2$ and rewrite every term using $s^2$; this was the approach taken in {\cite[Proposition 3.6]{ManeaSGE}}.
We will provide a slightly different proof; we will essentially proceed with the same computations, but we will write them in a manner that will also be useful for the proof of Theorem \ref{intro-th: sharp monotonic quantity}.

The first observation is that
\begin{equation} \label{eq: Delta s^2}
    \Delta s^2 = 6 \|\nabla s\|^2 - \frac{3k}{2} s^2.
\end{equation}
This follows by a direct computation using the chain rule and the identity $s \equiv G^{-1}$.
We now claim that
\begin{equation} \label{eq: gradient of sharp gradient quantity is traceless hessian}
    d \left( \|\nabla s\|^2 + \frac{k}{4} s^2 \right) = \frac{1}{2s^2} \tracelessHess{s^2} (\nabla s^2, \cdot).
\end{equation}
To see this, we may just compute, using the product rule, that $d \|\nabla s\|^2 = 2 \hess{s}(\nabla s, \cdot)$ and $\hess{s^2} = 2s \hess{s} + 2 ds \otimes ds$, so
\allowdisplaybreaks
\begin{align*}
    d \left( \|\nabla s\|^2 + \frac{k}{4} s^2 \right) &= 2 \hess{s}(\nabla s, \cdot) + \frac{ks}{2} ds \\
    &= 2 \left( \frac{1}{2s} \hess{s^2} (\nabla s, \cdot) - \frac{1}{s} \|\nabla s\|^2 ds \right) + \frac{ks}{2} ds \\
    &= \frac{1}{s} \hess{s^2}(\nabla s, \cdot) - \frac{2}{s} \left( \|\nabla s\|^2 - \frac{k}{4} s^2 \right) ds.
\end{align*}
The proof of the claim is now finished by combining the above equation with \eqref{eq: Delta s^2} and the identity $\hess{s}^2 = \tracelessHess{s^2} + (\Delta s^2/3)g$.

With the above, we are now ready to prove \eqref{eq: drift bochner formula for sharp gradient quantity}.
On the one hand, taking divergence in \eqref{eq: gradient of sharp gradient quantity is traceless hessian} and using that $\div \tracelessHess{s^2} = (2/3) d \Delta s^2 + \ric(\nabla s^2, \cdot)$, the product rule gives us the Bochner formula
\begin{equation} \label{eq: bochner formula with ricci term}
    \Delta \left( \|\nabla s\|^2 + \frac{k}{4} s^2 \right) = \frac{1}{2s^2} \left( \|\!\tracelessHess{s^2}\|^2 + \frac{2}{3} \langle \nabla \Delta s^2, \nabla s^2 \rangle + \ric(\nabla s^2, \nabla s^2) \right) - \frac{1}{2s^4} \tracelessHess{s^2}(\nabla s^2, \nabla s^2).
\end{equation}
On the other hand, differentiating \eqref{eq: Delta s^2} and recalling that $d \|\nabla s\|^2 = 2 \hess{s}(\nabla s, \cdot)$, we get that
\pagebreak
\allowdisplaybreaks
\begin{align*}
    \frac{2}{3} \langle \nabla \Delta s^2, \nabla s^2 \rangle &= 4 \langle \nabla \|\nabla s\|^2, \nabla s^2 \rangle - k \|\nabla s^2\|^2 \\
    &= 8 \hess{s}(\nabla s, \nabla s^2) - k \|\nabla s^2\|^2 \\
    &= \frac{4}{s} \hess{s^2}(\nabla s, \nabla s^2) - 16 \|\nabla s\|^4 - k \|\nabla s^2\|^2 \\
    &= \frac{2}{s^2} \tracelessHess{s^2}(\nabla s^2, \nabla s^2) - 2k \|\nabla s^2\|^2,
\end{align*}
where the last identity follows from \eqref{eq: Delta s^2} and the product rule.
Combining this last equation with \eqref{eq: bochner formula with ricci term}, we may conclude the desired result.
\end{proof}

The next step in proving \eqref{intro-eq: bochner formula for sharp monotonic quantity} is to use the traced Gauss equation on the regular level sets of $s$ to rewrite the Ricci curvature term from \eqref{eq: drift bochner formula for sharp gradient quantity} using only the scalar curvature and $\tracelessHess{s^2}$.
Since the second fundamental form of the level sets of $s$ may be written in terms of the restricted Hessian of $s^2$, we first record the following simple lemma.

\begin{lemma} \label{lemma: norm of hessian of s squared along the regular level sets of s}
If $\Sigma$ is a regular level set of $s$, then
\begin{equation} \label{eq: norm of hessian of s squared along the regular level sets of s}
    \|\!\tracelessHess{s^2}\|^2 = \frac{1}{2} (\tracelessHess{s^2}(\nu, \nu))^2 + \|\!\tracelessHess{s^2}(\nu)^\top\|^2 + \|\!\tracelessHess{s^2}(\nu, \cdot)\|^2 + \|(\tracelessHess{s^2}|_\Sigma)^\circ\|^2,
\end{equation}
where $\nu := \|\nabla s^2\|^{-1} \nabla s^2$, $\tracelessHess{s^2}(\nu)$ is the vector field associated with the $1$-form $\tracelessHess{s^2}(\nu, \cdot)$, $\top$ denotes the tangential component with respect to $\Sigma$, and $(\tracelessHess{s^2}|_\Sigma)^\circ$ is the traceless part of $\tracelessHess{s^2}|_\Sigma$.
\end{lemma}

\begin{proof}
By definition of the norm,
\[
\|\!\tracelessHess{s^2}\|^2 = (\tracelessHess{s^2}(\nu, \nu))^2 + 2 \|\!\tracelessHess{s^2}(\nu)^\top\|^2 + \|\!\tracelessHess{s^2}|_\Sigma\|^2.
\]
At the same time, $\tracelessHess{s^2}|_\Sigma = (\tracelessHess{s^2}|_\Sigma)^\circ + (\tr_\Sigma(\tracelessHess{s^2}|_\Sigma)/2)g|_\Sigma$, and we may directly compute
\begin{equation} \label{eq: trace of traceless hess s squared along regular level sets of s}
    \tr_\Sigma ( \tracelessHess{s^2}|_\Sigma) = \tr_\Sigma \left( \hess{s^2}|_\Sigma - \frac{\Delta s^2}{3} g|_\Sigma \right) = \frac{1}{3} \Delta s^2 - \hess{s^2}(\nu, \nu) = - \tracelessHess{s^2}(\nu, \nu).
\end{equation}
Thus, $\|\!\tracelessHess{s^2}|_\Sigma\|^2 = \|(\tracelessHess{s^2}|_\Sigma)^\circ\|^2 + \tracelessHess{s^2}(\nu, \nu)^2/2$,
and the desired conclusion now follows because 
\[
\|\!\tracelessHess{s^2}(\nu)^\top\|^2 = \|\!\tracelessHess{s^2}(\nu, \cdot)\|^2 - (\tracelessHess{s^2}(\nu, \nu))^2.
\]
\end{proof}

With the above, we are ready to rewrite the Ricci curvature along the regular level sets of $s$.

\begin{proposition} \label{prop: ricci curvature along regular level sets of s}
If $\Sigma$ is a regular level set of $s$, then
\begin{equation} \label{eq: ricci curvature along regular level sets of s}
    \begin{split}
        \ric(\nabla s^2, \nabla s^2) + \|\!\tracelessHess{s^2}\|^2 &= \frac{\|\nabla s^2\|^2}{2} (\scal - \scal_\Sigma) + \frac{1}{2} \|\!\tracelessHess{s^2}\|^2 \\ 
        &\quad \, + \frac{1}{\|\nabla s^2\|^2} \|\!\tracelessHess{s^2}(\nabla s^2, \cdot)\|^2 + \frac{(\Delta s^2)^2 }{9} - \frac{\Delta s^2}{3 \|\nabla s^2\|^2} \tracelessHess{s^2} (\nabla s^2, \nabla s^2),
    \end{split}
\end{equation}
where $\scal_\Sigma$ denotes the scalar curvature of $\Sigma$ in the induced metric.
\end{proposition}

\begin{proof}
As $s > 0$ on $M \setminus \{p\}$, $\Sigma$ is also a regular level set of $s^2$.
For the rest of this proof, all computations are understood along the level set $\Sigma$.

As in the proof of Lemma \ref{lemma: norm of hessian of s squared along the regular level sets of s}, put $\nu := \|\nabla s^2\|^{-1} \nabla s^2$.
Let $\two$ and $H$ be the second fundamental form and mean curvature of $\Sigma$, respectively, with respect to $\nu$.
It follows that
\[
\two = \frac{1}{\|\nabla s^2\|} \hess{s^2}|_\Sigma, \qquad H = \div \left( \frac{\nabla s^2}{\|\nabla s^2\|} \right).
\]
Let us first rewrite both $\two$ and $H$ in terms of $\tracelessHess{s^2}$.
On the one hand, since $d \|\nabla s^2\| = \hess{s^2}(\nu, \cdot)$, we have
\begin{equation} \label{eq: mean curvature of level sets of s in terms of Hessian of s squared}
    H \equiv \div \left( \frac{\nabla s^2}{\|\nabla s^2\|} \right) = \frac{1}{\|\nabla s^2\|} \left( \Delta s^2 - \hess{s^2}(\nu, \nu) \right) = \frac{1}{\|\nabla s^2\|} \left( \frac{2}{3} \Delta s^2 - \tracelessHess{s^2}(\nu, \nu) \right).
\end{equation}
On the other hand, as $\two \equiv \|\nabla s^2\|^{-1} \hess{s^2}|_\Sigma$, it follows that
\begin{equation} \label{eq: second fundamental form of level sets of s in terms of hessian of s squared part 1}
      \two = \frac{1}{\|\nabla s^2\|} \left( \tracelessHess{s^2}|_\Sigma + \frac{\Delta s^2}{3} g|_\Sigma \right) = \frac{1}{\|\nabla s^2\|} (\tracelessHess{s^2}|_\Sigma)^\circ + \frac{1}{\|\nabla s^2\|} \left( \frac{\tr_\Sigma(\tracelessHess{s^2}|_\Sigma)}{2} + \frac{\Delta s^2}{3} \right) g|_\Sigma,
\end{equation}
where $(\tracelessHess{s^2}|_\Sigma)^\circ$ denotes the traceless part of $\tracelessHess{s^2}|_\Sigma$.
Combining \eqref{eq: second fundamental form of level sets of s in terms of hessian of s squared part 1} with the identity \eqref{eq: trace of traceless hess s squared along regular level sets of s} from the proof of Lemma \ref{lemma: norm of hessian of s squared along the regular level sets of s}, we conclude that
\begin{equation} \label{eq: second fundamental form of level sets of s in terms of hessian of s squared part 2}
    \two = \frac{1}{\|\nabla s^2\|} (\tracelessHess{s^2}|_\Sigma)^\circ + \frac{1}{2\|\nabla s^2\|} \left( \frac{2}{3} \Delta s^2 - \tracelessHess{s^2}(\nu, \nu) \right) g|_\Sigma.
\end{equation}

Now, the traced Gauss equation and the rearrangement identity of Schoen-Yau tell us that
\begin{equation} \label{eq: gauss codazzi and rearrangement on regular level sets of s part 1}
    2\ric(\nu, \nu) = \scal - \scal_\Sigma + H^2 - \|\two\|^2.
\end{equation}
Inserting \eqref{eq: mean curvature of level sets of s in terms of Hessian of s squared} and \eqref{eq: second fundamental form of level sets of s in terms of hessian of s squared part 2} into \eqref{eq: gauss codazzi and rearrangement on regular level sets of s part 1}, we obtain the identity
\begin{equation} \label{eq: gauss codazzi and rearrangement on regular level sets of s part 2}
    \ric(\nu, \nu) = \frac{1}{2} (\scal - \scal_\Sigma) + \frac{1}{4\|\nabla s^2\|^2} \left( \frac{2}{3} \Delta s^2 - \tracelessHess{s^2}(\nu, \nu) \right)^2 - \frac{1}{2\|\nabla s^2\|^2} \| (\tracelessHess{s^2}|_\Sigma)^\circ \|^2.
\end{equation}
Combining this with \eqref{eq: norm of hessian of s squared along the regular level sets of s} from Lemma \ref{lemma: norm of hessian of s squared along the regular level sets of s}, we conclude that
\begin{equation} \label{eq: gauss codazzi and rearrangement on regular level sets of s part 3}
    \begin{split}
        \ric(\nu, \nu) + \frac{1}{\|\nabla s^2\|^2} \|\!\tracelessHess{s^2}\|^2 &= \frac{1}{2} (\scal - \scal_\Sigma) + \frac{1}{4\|\nabla s^2\|^2} \left( \frac{2}{3} \Delta s^2 - \tracelessHess{s^2}(\nu, \nu) \right)^2 \\
        &\quad \, + \frac{1}{2\|\nabla s^2\|^2} \|(\tracelessHess{s^2}|_\Sigma)^\circ\|^2 + \frac{1}{2\|\nabla s^2\|^2} (\tracelessHess{s^2}(\nu, \nu))^2 \\
        &\quad \, + \frac{1}{\|\nabla s^2\|^2} \| \tracelessHess{s^2}(\nu)^\top \|^2 + \frac{1}{\|\nabla s^2\|^2} \|\!\tracelessHess{s^2}(\nu, \cdot)\|^2.
    \end{split}
\end{equation}
Now, expanding
\[
\frac{1}{4\|\nabla s^2\|^2} \left( \frac{2}{3} \Delta s^2 - \tracelessHess{s^2}(\nu, \nu) \right)^2 = \frac{(\Delta s^2)^2}{9\|\nabla s^2\|^2} - \frac{\Delta s^2}{3\|\nabla s^2\|^2} \tracelessHess{s^2}(\nu, \nu) + \frac{1}{4\|\nabla s^2\|^2} (\tracelessHess{s^2}(\nu, \nu))^2,
\]
substituting the above into \eqref{eq: gauss codazzi and rearrangement on regular level sets of s part 3}, and grouping the terms, we obtain the identity
\begin{equation} \label{eq: gauss codazzi and rearrangement on regular level sets of s part 4}
    \begin{split}
        \ric(\nu, \nu) + \frac{1}{\|\nabla s^2\|^2} \|\!\tracelessHess{s^2}\|^2 &= \frac{1}{2} (\scal - \scal_\Sigma) + \frac{(\Delta s^2)^2}{9\|\nabla s^2\|^2} - \frac{\Delta s^2}{3\|\nabla s^2\|^2} \tracelessHess{s^2}(\nu, \nu) \\
        &\quad \, + \frac{1}{\|\nabla s^2\|^2} \left[ \frac{3}{4} (\tracelessHess{s^2}(\nu, \nu))^2 + \frac{1}{2} \|(\tracelessHess{s^2}|_\Sigma)^\circ\|^2 + \|\!\tracelessHess{s^2}(\nu)^\top\|^2 \right] \\
        &\quad \, + \frac{1}{\|\nabla s^2\|^2} \|\!\tracelessHess{s^2}(\nu, \cdot)\|^2.
    \end{split}
\end{equation}
Finally, since $\|\!\tracelessHess{s^2}(\nu, \cdot)\|^2 = (\tracelessHess{s^2}(\nu, \nu))^2 + \|\!\tracelessHess{s^2}(\nu)^\top\|^2$, \eqref{eq: norm of hessian of s squared along the regular level sets of s} from Lemma \ref{lemma: norm of hessian of s squared along the regular level sets of s} implies that
\allowdisplaybreaks
\begin{align*}
    \|\!\tracelessHess{s^2}\|^2 &= \frac{1}{2} (\tracelessHess{s^2}(\nu, \nu))^2 + \|\!\tracelessHess{s^2}(\nu)^\top\|^2 + \|\!\tracelessHess{s^2}(\nu, \cdot)\|^2 + \|(\tracelessHess{s^2}|_\Sigma)^\circ\|^2 \\
    &= \frac{3}{2} (\tracelessHess{s^2}(\nu, \nu))^2 + \|(\tracelessHess{s^2}|_\Sigma)^\circ\|^2 + 2 \|\!\tracelessHess{s^2}(\nu)^\top\|^2.
\end{align*}
Inserting this into \eqref{eq: gauss codazzi and rearrangement on regular level sets of s part 4}, we obtain the equation
\begin{equation*}
    \begin{split}
        \ric(\nu, \nu) + \frac{1}{\|\nabla s^2\|^2} \|\!\tracelessHess{s^2}\|^2 &= \frac{1}{2} (\scal - \scal_\Sigma) + \frac{(\Delta s^2)^2}{9\|\nabla s^2\|^2} - \frac{\Delta s^2}{3\|\nabla s^2\|^2} \tracelessHess{s^2}(\nu, \nu) \\
        &\quad \, + \frac{1}{2 \|\nabla s^2\|^2} \|\!\tracelessHess{s^2}\|^2 + \frac{1}{\|\nabla s^2\|^2} \|\!\tracelessHess{s^2}(\nu, \cdot)\|^2.
    \end{split}
\end{equation*}
The desired conclusion now follows by multiplying this preceding identity by $\|\nabla s^2\|^2$.
\end{proof}

Combining Propositions \ref{prop: drift bochner formula for sharp gradient quantity} and \ref{prop: ricci curvature along regular level sets of s}, we obtain the following result.

\begin{corollary} \label{cllr: almost divergence computation for sharp monotonic quantity}
If $\Sigma$ is a regular level set of $s$, then
\begin{equation}
    \begin{split}
        \Delta \left( \|\nabla s\|^2 + \frac{k}{4} s^2 \right) &= - \|\nabla s\|^2 \scal_\Sigma + \|\nabla s\|^2 (\scal - 6k) + \frac{1}{4s^2} \|\!\tracelessHess{s^2}\|^2 \\ 
        &\quad \, + \frac{2}{s^2} \left(\|\nabla s\|^2 + \frac{k}{4} s^2 \right)^2 +  \frac{1}{2s^2 \|\nabla s^2\|^2}\|\!\tracelessHess{s^2}(\nabla s^2, \cdot)\|^2 \\
        &\quad \, + \left( \frac{1}{4s^4} + \frac{k}{4 \|\nabla s^2\|^2} \right) \tracelessHess{s^2}(\nabla s^2, \nabla s^2),
    \end{split}
\end{equation}
where $\scal_\Sigma$ denotes the scalar curvature of $\Sigma$ in the induced metric.
\end{corollary}

\begin{proof}
Using that $\|\nabla s^2\|^2 = 4 s^2 \|\nabla s\|^2$, combining Propositions \ref{eq: drift bochner formula for sharp gradient quantity} and \ref{prop: ricci curvature along regular level sets of s} and then grouping terms, we have
\begin{equation} \label{eq: drift bochner formula on regular level sets of s}
    \begin{split}
        \Delta \left( \|\nabla s\|^2 + \frac{k}{4} s^2 \right) &= - \|\nabla s\|^2 \scal_\Sigma + \|\nabla s\|^2 (\scal - 6k) + \frac{1}{4s^2} \|\!\tracelessHess{s^2}\|^2 \\
        &\quad \, + 2k \|\nabla s\|^2 + \frac{1}{2s^2 \|\nabla s^2\|^2} \|\tracelessHess{s^2}(\nabla s^2, \cdot)\|^2 + \frac{(\Delta s^2)^2}{18 s^2} \\
        &\quad \, + \left( \frac{1}{2s^4} - \frac{\Delta s^2}{6 s^2 \|\nabla s^2\|^2} \right) \tracelessHess{s^2}(\nabla s^2, \nabla s^2).
    \end{split}
\end{equation}
Now, by \eqref{eq: Delta s^2} from the proof of Proposition \ref{eq: drift bochner formula for sharp gradient quantity}, we may also compute
\allowdisplaybreaks
\begin{align} \label{eq: Delta s^2 squared for ricci on regular level sets of s}
    \frac{(\Delta s^2)^2}{18s^2} &= \frac{1}{18s^2} \left( 6 \|\nabla s\|^2 - \frac{3k}{2} s^2 \right)^2 \nonumber \\
    &= \frac{2}{s^2} \left( \|\nabla s\|^2 + \frac{k}{4} s^2 \right)^2 - 2k \|\nabla s\|^2
\end{align}
and
\begin{equation} \label{eq: Delta s^2 divided by s squared grad s squared}
    - \frac{\Delta s^2}{6 s^2 \|\nabla s^2\|^2} = - \frac{1}{4s^4} + \frac{k}{4\|\nabla s^2\|^2}
\end{equation}
Substituting \eqref{eq: Delta s^2 squared for ricci on regular level sets of s} and \eqref{eq: Delta s^2 divided by s squared grad s squared} into \eqref{eq: drift bochner formula on regular level sets of s} gives the desired conclusion.
\end{proof}

Using the above identities, we may proceed with the proof of \eqref{intro-eq: bochner formula for sharp monotonic quantity}, which we state as a proposition below.
As previously remarked, this is a positive-curvature extension of {\cite[Proposition 1.17]{CMGradientEsimatesScalarCurvature}}.

\begin{proposition} \label{prop: divergence on regular level sets for sharp monotonic quantity}
If $\Sigma$ is a regular level set of $s$, then
\begin{equation} \label{eq: divergence form bochner formula along level sets for sharp monotonic quantity}
    \begin{split}
        \frac{1}{2\|\nabla s\|} \div \left( \frac{1}{\|\nabla s\|} \nabla \left( \|\nabla s\|^2 + \frac{k}{4} s^2 \right) \right) &= - \frac{1}{2} \scal_\Sigma + \frac{1}{2} (\scal - 6k) + \frac{1}{2\|\nabla s^2\|^2} \|\!\tracelessHess{s^2}\|^2 \\
        &\quad \, + \frac{1}{s^2 \|\nabla s\|^2} \left( \|\nabla s\|^2 + \frac{k}{4} s^2 \right)^2 \\
        &\quad \, + \left( \frac{1}{2s^2} + \frac{k}{4\|\nabla s\|^2} \right) \tracelessHess{s^2} \left( \frac{\nabla s}{\|\nabla s\|}, \frac{\nabla s}{\|\nabla s\|} \right),
    \end{split}
\end{equation}
where $\scal_\Sigma$ denotes the scalar curvature of $\Sigma$ in the induced metric.
\end{proposition}

\begin{proof}
By the product rule,
\allowdisplaybreaks
\begin{align} \label{align: product rule for divergence term for sharp monotonic quantity}
    \frac{1}{2\|\nabla s\|} \div \left( \frac{1}{\|\nabla s\|} \nabla \left( \|\nabla s\|^2 + \frac{k}{4} s^2 \right) \right) &= \frac{1}{2 \|\nabla s\|^2} \Delta \left( \|\nabla s\|^2 + \frac{k}{4} s^2 \right) \nonumber \\
    &\quad \, - \frac{1}{2 \|\nabla s\|^3} \left\langle \nabla \left( \|\nabla s\|^2 + \frac{k}{4} s^2 \right), \nabla \|\nabla s\| \right\rangle.
\end{align}
On the one hand, by Corollary \ref{cllr: almost divergence computation for sharp monotonic quantity}, we have
\begin{equation} \label{eq: almost divergence for sharp monotonic quantity part 1}
    \begin{split}
        \frac{1}{2 \|\nabla s\|^2}\Delta \left( \|\nabla s\|^2 + \frac{k}{4} s^2 \right) &= - \frac{\scal_\Sigma}{2} + \frac{\scal - 6k}{2} + \frac{1}{8s^2 \|\nabla s\|^2} \|\!\tracelessHess{s^2}\|^2 \\ 
        &\quad \, + \frac{1}{s^2 \|\nabla s\|^2} \left(\|\nabla s\|^2 + \frac{k}{4} s^2 \right)^2 + \frac{1}{16 s^4 \|\nabla s\|^4} \|\!\tracelessHess{s^2}(\nabla s^2, \cdot)\|^2 \\
        &\quad \, + \frac{1}{2 \|\nabla s\|^2} \left( \frac{1}{4s^4} + \frac{k}{4 \|\nabla s^2\|^2} \right) \tracelessHess{s^2}(\nabla s^2, \nabla s^2).
    \end{split}
\end{equation}
On the other hand, we may compute, with the chain rule and the identity $d \|\nabla s\| = \hess{s} (\|\nabla s\|^{-1} \nabla s, \cdot)$,
\allowdisplaybreaks
\begin{align} \label{align: second term in divergence of sharp monotonic quantity part 1}
    \frac{k}{4} \langle \nabla s^2, \nabla \|\nabla s\| \rangle &= \frac{k}{4\|\nabla s^2\|} \hess{s}(\nabla s^2, \nabla s^2) \nonumber \\
    &= \frac{k}{4\|\nabla s^2\|} \left( \frac{1}{2s} \hess{s^2}(\nabla s^2, \nabla s^2) - 4s \|\nabla s\|^4 \right) \nonumber \\
    &= \frac{k}{4 \|\nabla s^2\|} \left( \frac{1}{2s} \tracelessHess{s^2}(\nabla s^2, \nabla s^2) - k s^3 \|\nabla s\|^2 \right) \nonumber \\
    &= \frac{k}{16 s^2 \|\nabla s\|} \tracelessHess{s^2}(\nabla s^2, \nabla s^2) - \frac{k^2}{8} s^2 \|\nabla s\|,
\end{align}
where the second-to-last identity follows because $\Delta s^2 = 6 \|\nabla s\|^2 - 3k s^2/2$ (cf.~\eqref{eq: Delta s^2}).

Similarly, $\langle \nabla \|\nabla s\|^2, \nabla \|\nabla s\| \rangle = 2 \|\nabla s\| \|\nabla \|\nabla s\|\|^2$, and by \eqref{eq: gradient of sharp gradient quantity is traceless hessian} from the proof of Proposition \ref{prop: drift bochner formula for sharp gradient quantity}, it follows that
\allowdisplaybreaks
\begin{align} \label{align: second term in divergence of sharp monotonic quantity part 2}
    2\|\nabla s\| \|\nabla \|\nabla s\|\|^2 &= \frac{1}{2\|\nabla s\|} \left\|\frac{1}{2s^2} \tracelessHess{s^2}(\nabla s^2, \cdot) - \frac{ks}{2} ds \right\|^2 \nonumber \\
    &= \frac{1}{8s^4 \|\nabla s\|} \|\!\tracelessHess{s^2}(\nabla s^2, \cdot)\|^2 + \frac{k^2}{8} s^2 \|\nabla s\| - \frac{k}{8s^2 \|\nabla s\|} \tracelessHess{s^2}(\nabla s^2, \nabla s^2).
\end{align}
Thus, by \eqref{align: second term in divergence of sharp monotonic quantity part 1} and \eqref{align: second term in divergence of sharp monotonic quantity part 2}, we conclude that
\begin{equation} \label{eq: almost divergence for sharp monotonic quantity part 2}
    - \frac{1}{2\|\nabla s\|^3}\left\langle \nabla \left( \|\nabla s\|^2 + \frac{k}{4} s^2 \right), \nabla \|\nabla s\| \right\rangle = -\frac{1}{16 s^4 \|\nabla s\|^4} \|\!\tracelessHess{s^2}(\nabla s^2, \cdot)\|^2 + \frac{k}{32 s^2 \|\nabla s\|^4} \tracelessHess{s^2}(\nabla s^2, \nabla s^2).
\end{equation}
The desired conclusion now follows by combining \eqref{align: product rule for divergence term for sharp monotonic quantity} with \eqref{eq: almost divergence for sharp monotonic quantity part 1} and \eqref{eq: almost divergence for sharp monotonic quantity part 2}.
\end{proof}

We are now ready to prove Theorem \ref{intro-th: sharp monotonic quantity}.

\begin{proof}[Proof of Theorem \ref{intro-th: sharp monotonic quantity}]
For the rest of this proof, all derivatives are computed at regular values of $s$, which form a dense and full measure subset of its range by Lemma \ref{appendix-lemma: regular values of G}.
As mentioned in Section \ref{sec: intro}, the main idea of this proof is to compute the integral of $\div(\|\nabla s\|^{-1} \nabla (\|\nabla s\|^2 + ks^2/4))$ on the sublevel set $\{s \leq r\}$ in two different ways: one via the Divergence Theorem, and one by using the coarea formula and Proposition \ref{prop: divergence on regular level sets for sharp monotonic quantity}.
Thus, let us proceed with this argument; we remark that equation \eqref{align: product rule for divergence term for sharp monotonic quantity} and the asymptotics of $s$ near its singularity (cf.~Lemma \ref{appendix-lemma: asymptotics of greens function near singularity}) show that the function $\div(\|\nabla s\|^{-1} \nabla (\|\nabla s\|^2 + ks^2/4))$ is integrable near $p$, and  Remark \ref{appendix-remark: integrability of the singular vector field for sharp monotonic quantity} shows that it is thus also locally integrable on $M$.
Denote by
\begin{equation} \label{eq: definition of energy and area terms for s}
    E(r) := \int_{s = r} \|\nabla s\|^2, \qquad A(r) := \mathrm{area}(s = r), \qquad B(r) := \int_{s = r} \frac{1}{\|\nabla s\|^2},
\end{equation}
where $\mathrm{area}(s = r)$ denotes the two-dimensional Hausdorff measure of the level set $\{s = r\}$.
By the coarea formula, the statement of Theorem \ref{intro-th: sharp monotonic quantity} may be written as
\begin{equation} \label{eq: what we need to prove the sharp monotonic quantity via s}
    \frac{\dd}{\dd r} \left( \frac{1}{r^3} E(r) + \frac{3k}{2r^2} \int_0^r A(t) \dd t + \frac{3k^2}{32r^2} \int_0^r (r^2 - t^2) B(t) \dd t - \frac{4 \pi}{r} \right) \geq 0,
\end{equation}
with equality at some $r > 0$ only if $r \leq 2/\sqrt{k}$ and $(M, g)$ is isometric to the three-sphere of curvature $k$.

Put $\nu := \|\nabla s\|^{-1} \nabla s$.
Let us first compute the derivatives of $E$ and $A$; these will be used in the subsequent computations.
For $E$, by Lemma \ref{appendix-lemma: local AC of energy term} and recalling that $d \|\nabla s\| = \hess{s} (\nu, \cdot)$, we have $E'(r) = \int_{s = r} ( \Delta s + \hess{s}(\nu, \nu) )$.
By \eqref{eq: Delta s^2} from the proof of Proposition \ref{prop: drift bochner formula for sharp gradient quantity} and the product rule, we have $\Delta s = 2 \|\nabla s\|^2/s - 3ks/4$; substituting this into the preceding derivative computation, we obtain the identity
\begin{equation} \label{eq: derivative of energy term with respect to s}
    E'(r) - \frac{2}{r} E(r) + \frac{3k}{4} r A(r) = \int_{s = r} \hess{s}(\nu, \nu).
\end{equation}
Similarly, for $A$, by the first variation of area, we have
\begin{equation} \label{align: derivative of area of level sets of s}
    A'(r) = \int_{s = r} \frac{1}{\|\nabla s\|} \div \left( \frac{\nabla s}{\|\nabla s\|} \right) = \int_{s = r} \frac{\Delta s - \hess{s}(\nu, \nu)}{\|\nabla s\|^2} = \frac{2}{r} A(r) - \frac{3k}{4} r B(r) - \int_{s = r} \frac{\hess{s}(\nu, \nu)}{\|\nabla s\|^2}.
\end{equation}

Let us now proceed with the computation of the integral of $\div(\|\nabla s\|^{-1} \nabla (\|\nabla s\|^2 + ks^2/4))$ on the sublevel set $\{s \leq r\}$.
On the one hand, for regular values $0 < u < r$ of $s$, Remark \ref{appendix-remark: integrability of the singular vector field for sharp monotonic quantity} from Appendix \ref{appendix-sec: integration over the (super)level sets of greens function} tells us that, with a Dominated Convergence Theorem argument, we may apply the Divergence Theorem to obtain that
\allowdisplaybreaks
\begin{align*}
    \int_{u \leq s \leq r} \div \left( \frac{1}{\|\nabla s\|} \nabla \left( \|\nabla s\|^2 + \frac{k}{4} s^2 \right) \right) &= \int_{s = r} \left\langle \nabla \left( \|\nabla s\|^2 + \frac{k}{4} s^2 \right), \frac{\nabla s}{\|\nabla s\|^2} \right\rangle \\
    &\quad \, - \int_{s = u} \left\langle \nabla \left( \|\nabla s\|^2 + \frac{k}{4} s^2 \right), \frac{\nabla s}{\|\nabla s\|^2} \right\rangle.
\end{align*}
By the asymptotics of $s$ near $p$ (cf.~Lemma \ref{appendix-lemma: asymptotics of greens function near singularity}), the second boundary term converges to zero as $u \downarrow 0$; thus, the Dominated Convergence Theorem implies that
\[
\int_{s \leq r} \div \left( \frac{1}{\|\nabla s\|} \nabla \left( \|\nabla s\|^2 + \frac{k}{4} s^2 \right) \right) = \int_{s = r} \left\langle \nabla \left( \|\nabla s\|^2 + \frac{k}{4} s^2 \right), \frac{\nabla s}{\|\nabla s\|^2} \right\rangle.
\]
Using the product rule and the identity $d \|\nabla s\| = \hess{s}(\nu, \cdot)$ for the above right-hand side, we obtain that
\[
\int_{s \leq r} \div \left( \frac{1}{\|\nabla s\|} \nabla \left( \|\nabla s\|^2 + \frac{k}{4} s^2 \right) \right) = 2 \int_{s = r} \hess{s}(\nu, \nu) + \frac{k}{2} r \, \mathrm{area}(s = r) \equiv 2 \int_{s = r} \hess{s}(\nu, \nu) + \frac{k}{2} r A(r).
\]
Inserting this equation into \eqref{eq: derivative of energy term with respect to s}, it follows that
\begin{equation} \label{eq: proof of sharp monotonic quantity with s part 1}
    \frac{1}{2} \int_{s \leq r} \div \left( \frac{1}{\|\nabla s\|} \nabla \left( \|\nabla s\|^2 + \frac{k}{4} s^2 \right) \right) = E'(r) - \frac{2}{r} E(r) + k r A(r).
\end{equation}

On the other hand, we may apply the coarea formula to compute
\[
\frac{1}{2} \int_{s \leq r} \div \left( \frac{1}{\|\nabla s\|} \nabla \left( \|\nabla s\|^2 + \frac{k}{4} s^2 \right) \right) = \int_0^r \left( \int_{s = t} \frac{1}{2 \|\nabla s\|} \div \left( \frac{1}{\|\nabla s\|} \nabla \left( \|\nabla s\|^2 + \frac{k}{4} s^2 \right) \right) \right) \dd t,
\]
and we may then use \eqref{eq: divergence form bochner formula along level sets for sharp monotonic quantity} from Proposition \ref{prop: divergence on regular level sets for sharp monotonic quantity} to conclude that
\allowdisplaybreaks
\begin{align*}
    \frac{1}{2} \int_{s \leq r} \div \left( \frac{1}{\|\nabla s\|} \nabla \left( \|\nabla s\|^2 + \frac{k}{4} s^2 \right) \right) &= \int_0^r \left( \int_{s = t} \left( - \frac{\scal_{\{s = t\}}}{2} + \frac{\scal - 6k}{2} + \frac{1}{2\|\nabla s^2\|^2} \|\!\tracelessHess{s^2}\|^2 \right) \right) \dd t \\
    &\quad \, + \int_0^r \left( \int_{s = t} \frac{1}{s^2 \|\nabla s\|^2} \left( \|\nabla s\|^2 + \frac{k}{4} s^2 \right)^2 \right) \dd t \\
    &\quad \, + \int_0^r \left( \int_{s = t} \left( \frac{1}{2s^2} + \frac{k}{4\|\nabla s\|^2} \right) \tracelessHess{s^2} \left( \nu, \nu \right) \right) \dd t,
\end{align*}
where $\scal_{\{s = t\}}$ is the scalar curvature of the level set $\{s = t\}$ in the induced metric.
Using the Gauss-Bonnet Theorem, the above equation becomes
\begin{equation} \label{eq: proof of sharp monotonic quantity with s part 2}
    \begin{split}
        \frac{1}{2} \int_{s \leq r} \div \left( \frac{1}{\|\nabla s\|} \nabla \left( \|\nabla s\|^2 + \frac{k}{4} s^2 \right) \right) &= \int_0^r \left( - 2 \pi \chi(s = t) + \int_{s = t} \left( \frac{\scal - 6k}{2} + \frac{1}{2\|\nabla s^2\|^2} \|\!\tracelessHess{s^2}\|^2 \right) \right) \dd t \\
        &\quad \, + \int_0^r \left( \int_{s = t} \frac{1}{s^2 \|\nabla s\|^2} \left( \|\nabla s\|^2 + \frac{k}{4} s^2 \right)^2 \right) \dd t \\
        &\quad \, + \int_0^r \left( \int_{s = t} \left( \frac{1}{2s^2} + \frac{k}{4\|\nabla s\|^2} \right) \tracelessHess{s^2} \left( \nu, \nu \right) \right) \dd t.
    \end{split}
\end{equation}
We now proceed with rewriting the last two integral terms in \eqref{eq: proof of sharp monotonic quantity with s part 2}.
For the first one, expanding the square, we obtain that
\allowdisplaybreaks
\begin{align} \label{eq: proof of sharp monotonic quantity with s part 3}
    \int_{s = t} \frac{1}{s^2 \|\nabla s\|^2} \left( \|\nabla s\|^2 + \frac{k}{4} s^2 \right)^2 &= \int_{s = t} \frac{1}{s^2 \|\nabla s\|^2} \left( \|\nabla s\|^4 + \frac{k}{2} s^2 \|\nabla s\|^2 + \frac{k^2}{16} s^4 \right) \nonumber \\
    &=  \frac{1}{t^2} E(t) + \frac{k}{2} A(t) + \frac{k^2}{16} t^2 B(t),
\end{align}
where $E, A$, and $B$ were defined in \eqref{eq: definition of energy and area terms for s}.

For the second term, we rephrase it using only $\hess{s}$ as follows: by definition of the traceless Hessian, by the chain rule, and by \eqref{eq: Delta s^2} from the proof of Proposition \ref{prop: drift bochner formula for sharp gradient quantity}, we have
\[
\tracelessHess{s^2} = \hess{s^2} - \frac{\Delta s^2}{3} g = 2s \hess{s} + 2 ds \otimes ds - \left( 2 \|\nabla s\|^2 - \frac{k}{2} s^2 \right)g,
\]
so
\begin{equation} \label{eq: proof of sharp monotonic quantity with s part 4}
    \left( \frac{1}{2s^2} + \frac{k}{4\|\nabla s\|^2} \right) \tracelessHess{s^2} \left( \nu, \nu \right) = \left( \frac{1}{s} + \frac{ks}{2\|\nabla s\|^2} \right) \left( \hess{s}(\nu, \nu) + \frac{k}{4} s \right).
\end{equation}
Furthermore,
\allowdisplaybreaks
\begin{align*}
    \int_{s = t} \left( \frac{1}{s} + \frac{ks}{2\|\nabla s\|^2} \right) \left( \hess{s}(\nu, \nu) + \frac{k}{4} s \right) &= \frac{1}{t} \int_{s = t} \left( \hess{s}(\nu, \nu) + \frac{kt}{4} \right) + \frac{kt}{2} \int_{s = t} \frac{\hess{s}(\nu, \nu)}{\|\nabla s\|^2} \\
    &\quad \, + \frac{k^2 t^2}{8} \int_{s = t} \frac{1}{\|\nabla s\|^2} \\
    &\equiv \frac{k}{4} A(t) + \int_{s = t} \hess{s}(\nu, \nu) + \frac{kt}{2} \int_{s = t} \frac{\hess{s}(\nu, \nu)}{\|\nabla s\|^2} + \frac{k^2 t^2}{8} B(t).
\end{align*}
Inserting \eqref{eq: derivative of energy term with respect to s} and \eqref{align: derivative of area of level sets of s} into the preceding equation and combining the resulting one with \eqref{eq: proof of sharp monotonic quantity with s part 4}, it follows that
\begin{equation} \label{eq: proof of sharp monotonic quantity with s part 5}
    \int_0^r \left( \int_{s = t}  \left( \frac{1}{2s^2} + \frac{k}{4\|\nabla s\|^2} \right) \tracelessHess{s^2}(\nu, \nu) \right) \dd t = \int_0^r \left( \frac{1}{t} E'(t) - \frac{2}{t^2} E(t) + 2k A(t) - \frac{k}{2} t A'(t) - \frac{k^2}{4} t^2 B(t) \right) \dd t.
\end{equation}
Let us now integrate by parts the terms involving $E'$ and $A'$.
Lemma \ref{appendix-lemma: local AC of energy term} tells us that $E$ is locally absolutely continuous; whilst $A$ may, in general, not satisfy this property, it is equal almost everywhere with a locally absolutely continuous function, so we may still perform the integration by parts for the $A'$ term above (a more detailed explanation of this argument is provided in Lemma \ref{appendix-lemma: local AC of area of the level sets of s} and \eqref{appendix-eq: how to integrate by parts the derivative of the area of the level sets of s} in Appendix \ref{appendix-sec: integration over the (super)level sets of greens function}).
The asymptotics of $s$ near its singularity (cf.~Lemma \ref{appendix-lemma: asymptotics of greens function near singularity}) also show that $E(r), A(r) \sim 4 \pi r^2$ as $r \downarrow 0$, and we thus obtain that
\begin{equation} \label{eq: proof of sharp monotonic quantity with s part 6}
    \int_0^r \frac{1}{t} E'(t) \dd t = \frac{1}{r} E(r) + \int_0^r \frac{1}{t^2} E(t) \dd t, \qquad \int_0^r t A'(t) \dd t = r A(r) - \int_0^r A(t) \dd t.
\end{equation}
Substituting \eqref{eq: proof of sharp monotonic quantity with s part 6} into \eqref{eq: proof of sharp monotonic quantity with s part 5}, combining the resulting equation with \eqref{eq: proof of sharp monotonic quantity with s part 3} and \eqref{eq: proof of sharp monotonic quantity with s part 2}, and then simplifying and grouping the common terms, we finally obtain that
\begin{equation} \label{eq: proof of sharp monotonic quantity with s part 7}
    \begin{split}
        \frac{1}{2} \int_{s \leq r} \div \left( \frac{1}{\|\nabla s\|} \nabla \left( \|\nabla s\|^2 + \frac{k}{4} s^2 \right) \right) &= \int_0^r \left( - 2 \pi \chi(s = t) + \int_{s = t} \left( \frac{\scal - 6k}{2} + \frac{1}{2\|\nabla s^2\|^2} \|\!\tracelessHess{s^2}\|^2 \right) \right) \dd t \\
        &\quad \, + \frac{1}{r} E(r) + 3k \int_0^r A(t) \dd t - \frac{k}{2} r A(r) - \frac{3k^2}{16} \int_0^r t^2 B(t) \dd t.
    \end{split}
\end{equation}
Using \eqref{eq: proof of sharp monotonic quantity with s part 1} and \eqref{eq: proof of sharp monotonic quantity with s part 7} and simplifying terms gives us the differential identity
\begin{equation} \label{eq: proof of sharp monotonic quantity with s part 8}
    \begin{split}
        E'(r) - \frac{3}{r} E(r) + \frac{3k}{2} r A(r) &= 3k \int_0^r A(t) \dd t - \frac{3k^2}{16} \int_0^r t^2 B(t) \dd t \\
        &\quad \, - 2 \pi \int_0^r \chi(s = t) + \int_0^r \left( \int_{s = t} \left( \frac{\scal - 6k}{2} + \frac{1}{2\|\nabla s^2\|^2} \|\!\tracelessHess{s^2}\|^2 \right) \right) \dd t
    \end{split}
\end{equation}
With this equation, we may now proceed with proving \eqref{intro-eq: quantitative derivative of sharp monotonic quantity}, the quantitative version of Theorem \ref{intro-th: sharp monotonic quantity}.
Define
\begin{equation} \label{eq: definition of w for sharp monotonic quantity via s}
    w(r) := \frac{1}{r^2} E(r) + \frac{3k}{2r} \int_0^r A(t) \dd t + \frac{3k^2}{32r} \int_0^r (r^2 - t^2) B(t) \dd t.
\end{equation}
To prove Theorem \ref{intro-th: sharp monotonic quantity}, we need to show that the derivative of $(w(r) - 4\pi)/r$ is non-negative (cf.~\eqref{eq: what we need to prove the sharp monotonic quantity via s}).
As previously mentioned, $E$ is locally absolutely continuous (cf.~Lemma \ref{appendix-lemma: local AC of energy term}); the second term in the definition of $w$ also clearly satisfies the same property.
Lemma \ref{appendix-lemma: local AC in sharp monotonic quantity} shows that the last integral term from \eqref{eq: definition of w for sharp monotonic quantity via s} is also locally absolutely continuous, and \eqref{appendix-eq: derivative of third term in sharp monotonic quantity} tells us how to compute its derivative.
We then see that
\[
w'(r) = \frac{1}{r^2} E'(r) - \frac{2}{r^3} E(r) - \frac{3k}{2r^2} \int_0^r A(t) \dd t + \frac{3k}{2r} A(r) + \frac{3k^2}{32} \int_0^r B(t) \dd t + \frac{3k^2}{32r^2} \int_0^r t^2 B(t) \dd t,
\]
and thus
\begin{equation} \label{eq: proof of sharp monotonic quantity with s part 9}
    \begin{split}
        \frac{\dd}{\dd r} \left( \frac{w(r) - 4\pi}{r} \right) &\equiv \frac{w'(r)}{r} - \frac{w(r)}{r^2} + \frac{4\pi}{r^2} \\
        &= \frac{1}{r^3} E'(r) - \frac{3}{r^4} E(r) + \frac{3k}{2r^2} A(r) - \frac{3k}{r^3} \int_0^r A(t) \dd t + \frac{3k^2}{16 r^3} \int_0^r t^2 B(t) \dd t + \frac{4\pi}{r^2}.
    \end{split}
\end{equation}
Combining \eqref{eq: proof of sharp monotonic quantity with s part 8} with \eqref{eq: proof of sharp monotonic quantity with s part 9}, it follows that
\begin{equation}  \label{eq: proof of sharp monotonic quantity with s - done}
    \frac{\dd}{\dd r} \left( \frac{w(r) - 4\pi}{r} \right) = \frac{1}{r^3} \int_0^r \left[ 2\pi (2 - \chi(s = t)) + \int_{s = t} \left( \frac{1}{2} (\scal - 6k) + \frac{1}{2\|\nabla s^2\|^2} \|\!\tracelessHess{s^2} \|^2 \right) \right] \dd t,
\end{equation}
which concludes the proof of \eqref{intro-eq: quantitative derivative of sharp monotonic quantity}.
The monotonicity statement of Theorem \ref{intro-th: sharp monotonic quantity} is now immediate, as each term from the right-hand side of \eqref{eq: proof of sharp monotonic quantity with s - done} is non-negative; indeed, by assumption, all the regular level sets of $G$ (and hence of $s$) are path-connected, so $\chi(s = t) \leq 2$ for almost every $t > 0$, and $\scal - 6k$ and $\|\!\tracelessHess{s^2}\|^2$ are also non-negative.

For the rigidity statement, we proceed as follows.
Suppose that there is some regular value $r_0 > 0$ of $s$ so that
\[
\left. \frac{\dd }{\dd r} \right|_{r = r_0} \left( \frac{w(r) - 4\pi }{r} \right) = 0.
\]
This, together with \eqref{eq: proof of sharp monotonic quantity with s - done} and the remark that each term from the integrand on the right-hand side of \eqref{eq: proof of sharp monotonic quantity with s - done} is non-negative, implies (by continuity) that $\tracelessHess{s}^2 = 0$ on $\{s \leq r_0\}$.
By \eqref{eq: gradient of sharp gradient quantity is traceless hessian} from the proof of Proposition \ref{prop: drift bochner formula for sharp gradient quantity}, it follows that $\nabla (\|\nabla s\|^2 + ks^2/4) = 0$ on $\{s \leq r_0\}$.
We now claim that the subset $\{s < r_0\}$ is connected; let us prove this claim.
Since $\{s < r_0\}$ is an open subset of a three-dimensional manifold, it suffices to prove that $\{s < r_0\} \cup \{p\}$ is connected.
If, for the sake of contradiction, $\{s < r_0\} \cup \{p\} = \{G > 1/r_0\} \cup \{p\}$ were disconnected, then there would exist a connected component $U \subset (\{G > 1/r_0\} \cup \{p\})$ not containing the singularity $p$.
However, as $G$ is smooth and positive on $U$, it is bounded on $U$, and the maximum principle implies that $G$ attains is maximum on $\partial U \subset \partial \{G > 1/r_0\} = \{G = 1/r_0\}$, which is a contradiction as $G$ is cannot be constant on $\{G > 1/r_0\}$ (as a Green's function).
Thus, $\{G > 1/r_0\} \cup \{p\} \equiv \{s < r_0\} \cup \{p\}$ is connected, which concludes the proof of this claim.

By the previous paragraph, it follows that $\|\nabla s\|^2 + ks^2/4$ is constant on $\{s \leq r_0\}$; the asymptotics of $s$ near $p$ (cf.~Lemma \ref{appendix-lemma: asymptotics of greens function near singularity}) necessarily imply that $\|\nabla s\|^2 + ks^2/4 \to 1$ as $s \downarrow 0$, and thus $\|\nabla s\|^2 + ks^2/4 \equiv 1$ on $\{s \leq r_0\}$.
Since we may also compute (using \eqref{eq: Delta s^2} from the proof of Proposition \ref{prop: drift bochner formula for sharp gradient quantity})
\[
\hess{s^2} - (2 - ks^2)g = \tracelessHess{s^2} + \left( \frac{\Delta s^2}{3} - 2 + ks^2 \right) g = \tracelessHess{s^2} + 2 \left( \|\nabla s\|^2 + \frac{k}{4} s^2 - 1 \right) g,
\]
we obtain that $\|\!\hess{s^2} - (2 - ks^2)g\|^2 = \|\!\tracelessHess{s^2}\|^2 + 12 (\|\nabla s\|^2 + k s^2/4 - 1)^2$.
The previous computations tells us that both terms from this right-hand side are zero, and we thus have
\begin{equation} \label{eq: sharp monotonic quantity rigidity equation for gradient}
    \hess{s^2} = (2 - ks^2)g \text{ and } \|\nabla s\|^2 + k s^2/4 \equiv 1
\end{equation}
on $\{s \leq r_0\}$.
The second equation from \eqref{eq: sharp monotonic quantity rigidity equation for gradient} necessarily implies that $r_0 \leq 2/\sqrt{k}$, whilst the first, combined with the rigidity of warped products (cf.~{\cite[Section 1]{CheegerColdingWarpedProducts}} or {\cite[Corollary 4.3.4]{PetersenRG}}), tells us that, on $\{s \leq r_0\}$, the metric $g$ has the explicit form $g = db^2 + \snk^2(b) g_{\S^2}$, where $g_{\S^2}$ is the unit, round metric on the two-sphere, $b = (2/\sqrt{k}) \arcsin(s \sqrt{k}/2)$, and $\sn_k$ denotes the curvature-adapted sine function.
This means that $s = 2 \snk(\dist(p, \cdot)/2)$ and $G = (2 \snk (\dist(p, \cdot)/2))^{-1}$.

To finish the proof of rigidity, we employ the Positive Mass Theorem for the ``conformal'' blow-up of $(M, g)$.
Consider the metric $G^4 g$ on $M \setminus \{ p \}$; this is a complete Riemannian metric with scalar curvature equal to $G^{-4} (\scal - 6k)$, which is non-negative by assumption.
Lemma \ref{appendix-lemma: asymptotics of greens function near singularity} tells us that, near $p$, $G$ has an asymptotic expansion of the form $G = \dist(p, \cdot)^{-1} + A + O(\dist(p, \cdot))$ for a constant $A$ and that the metric $G^4 g$ is an asymptotically flat metric on $M \setminus \{ p \}$.
Although we are not working with the conformal Laplacian, the arguments from {\cite[Lemma 9.7]{LeeParkerYamabeProblem}} apply verbatim to show that the ADM mass of $(M \setminus \{ p \}, G^4 g)$ is equal to $2A$.
From the previous paragraph, we know that $G = (2 \snk(\dist(p, \cdot)/2))^{-1}$ near $p$, so we necessarily have that $A = 0$.
By the Positive Mass Theorem, it follows that $(M \setminus \{ p \}, G^4 g)$ is isometric to $\R^3$.
However, this implies that $\scal \equiv 6k$ on $M$, so $G$ is Green's function for the conformal Laplacian of $(M, g)$; thus, $(M, g)$ is also conformally diffeomorphic to the three-sphere.
These conditions necessarily imply that $(M, g)$ must be isometric to the three-sphere of curvature $k$.
This concludes the proof of Theorem \ref{intro-th: sharp monotonic quantity}.
\end{proof}
\section{The proof of Theorem \ref{intro-th: alternative sharp monotonic quantity}} \label{sec: alternative monotonic quantity proof of monotonicity}

In this section, we prove Theorem \ref{intro-th: alternative sharp monotonic quantity}.
As before, let $(M^3, g)$ be a closed, connected Riemannian three-manifold with $\scal \geq 6k$ for some $k > 0$, fix a point $p \in M$, and let $G$ be Green's function for the operator $- \Delta + 3k/4$, with singularity at $p$.
By the strong maximum principle, $G$ is positive on $M \setminus \{ p \}$.
Assume furthermore that all level sets of $G$ are path-connected and that $G \geq \sqrt{k}/2$ on $M \setminus \{p\}$, and put $s := G^{-1}$.
For ease of notation, in the rest of this section, we denote by 
\begin{equation} \label{eq: notation for G11}
    G_{11} := \left\langle \nabla \|\nabla G\|, \frac{\nabla G}{\|\nabla G\|} \right\rangle = \hess{G} \left( \frac{\nabla G}{\|\nabla G\|}, \frac{\nabla G}{\|\nabla G\|} \right),
\end{equation}
where the second equality follows because $d\|\nabla G\|^2 = 2\hess{G}(\nabla G, \cdot)$.

The proof of Theorem \ref{intro-th: alternative sharp monotonic quantity} uses similar ingredients to those from Section \ref{sec: sharp monotonic quantity proof of monotonicity}, such as rewriting the Ricci curvature along the level sets of $s$ using the traced Gauss equation.
However, we will proceed slightly differently from Section \ref{sec: sharp monotonic quantity proof of monotonicity}.
We will rewrite the statement of Theorem \ref{intro-th: alternative sharp monotonic quantity} purely in terms of $G$ and proceed with the proof using only the derivatives of $G$.
Even though one may employ a similar strategy as in Section \ref{sec: sharp monotonic quantity proof of monotonicity} -- using Proposition \ref{prop: divergence on regular level sets for sharp monotonic quantity} -- we will present a different proof, more in the spirit of the argument of Munteanu and Wang from {\cite{MunteanuWangComparisonPaper}}.
This makes it more apparent how the monotonic quantity arises and why the hypothesis that $G \geq \sqrt{k}/2$ is natural.

In terms of $G$, via the chain rule, the statement of Theorem \ref{intro-th: alternative sharp monotonic quantity} may be rewritten as the inequality
\begin{equation} \label{eq: alternative sharp monotonic quantity expression via G}
    \begin{split}
        \frac{\dd}{\dd r} \bigg( \frac{1}{r} &\int_{G = r} \|\nabla G\|^2 + \int_{G \geq r} \frac{3k(2G^2 + k - 6r^2)}{G^2(4G^2 - k)^2} \|\nabla G\|^3 \\
        &+ \int_{G \geq r} \left( \frac{3k}{4} - \frac{3k(G^2 - r^2)(4G^2 + 5k)}{8G^2(4G^2 - k)} \right) \|\nabla G\| - 4 \pi r \bigg) \leq 0,
    \end{split}
\end{equation}
with equality holding at some $r > \sqrt{k}/2$ only if $(M, g)$ is isometric to the three-sphere of curvature $k$.
As mentioned in Section \ref{sec: intro}, each integral term above is well-defined (\ie finite) and locally absolutely continuous for $r > \sqrt{k}/2$ in the range of $G$; these results are proved in Appendix \ref{appendix-sec: integration over the (super)level sets of greens function} (cf.~Lemmas \ref{appendix-lemma: local AC of energy term}, \ref{appendix-lemma: welldefinedness of alternative monotonic quantity}, and \ref{appendix-lemma: local AC of alternative monotonic quantity}).

To prove \eqref{eq: alternative sharp monotonic quantity expression via G}, we proceed with a similar argument as in {\cite[Section 2]{MunteanuWangComparisonPaper}}: we will rewrite the Bochner formula for $\|\nabla G\|^2$ along the regular level sets of $G$ using the traced Gauss equation and a sharp ``Kato identity'' for $\hess{G}$.
We will show (cf.~Proposition \ref{prop: sharp bochner identity via G for alternative monotonic quantity}) that if $\Sigma$ is a regular level set of $G$, then
\begin{equation} \label{eq: laplace of G along its level sets for alternative monotonic quantity}
    \begin{split}
        \Delta \|\nabla G\| &= - \frac{\|\nabla G\|}{2} \scal_\Sigma + \frac{15k}{4} \|\nabla G\| + \frac{3}{2} \frac{k - 2 G^2}{G(G^2 - k/4)} \|\nabla G\| (\Delta G - G_{11}) - \frac{3}{4} \frac{(k - 2 G^2)^2}{G^2(G^2 - k/4)^2} \|\nabla G\|^3 \\
        &\quad \, + \frac{\| \nabla^\Sigma \|\nabla G\| \|^2}{\|\nabla G\|} + \frac{\|(\hess{G}|_\Sigma)^\circ\|^2}{2 \|\nabla G\|} + \frac{\|\nabla G\|}{2} (\scal - 6k) \\
        &\quad \, + \frac{3}{4\|\nabla G\|} \left( \Delta G - G_{11} - \frac{k - 2 G^2}{G(G^2 - k/4)} \|\nabla G\|^2 \right)^2,
    \end{split}
\end{equation}
where $\scal_\Sigma$ denotes the scalar curvature of $\Sigma$ in the induced metric, $\nabla^\Sigma$ the tangential part along $\Sigma$ of the gradient, and $(\hess{G}|_\Sigma)^\circ$ is the traceless part of $\hess{G}|_\Sigma$.
It is now apparent that the condition $G \geq \sqrt{k}/2$ is required for the right-hand side of \eqref{eq: laplace of G along its level sets for alternative monotonic quantity} to be well-defined: this lower bound implies that $G > \sqrt{k}/2$ on every regular level set, hence $G^2 > k/4$ and all terms from the right-hand side of \eqref{eq: laplace of G along its level sets for alternative monotonic quantity} are well-defined.
Moreover, on the three-sphere of curvature $k$, one can check that $\Delta G - G_{11} = G^{-1} (G^2 - k/4)^{-1} (k - 2G^2) \|\nabla G\|^2$.
We will then obtain Theorem \ref{intro-th: alternative sharp monotonic quantity} the same way as we proved Theorem \ref{intro-th: sharp monotonic quantity}: by integrating $\Delta \|\nabla G\|$ over the superlevel set $\{G \geq r\}$ and computing this integral in two ways: one using the Divergence Theorem, and the other by applying the coarea formula and \eqref{eq: laplace of G along its level sets for alternative monotonic quantity}.

We begin with the traced Gauss equation for the regular level sets of $G$.
This is essentially Proposition \ref{prop: ricci curvature along regular level sets of s} written in terms of $G$ instead of $s$.

\begin{lemma} \label{lemma: exact codazzi equation for G}
If $\Sigma$ is a regular level set of $G$, then
\[
\ric \left( \nabla G, \nabla G \right) - \|\nabla \|\nabla G\|\|^2 = \frac{\|\nabla G\|^2}{2} \left( \scal - \scal_\Sigma \right) - \frac{1}{2} \|\! \hess{G}\|^2 + \frac{9k^2}{32} G^2 - \frac{3k}{4} G_{11} G,
\]
where $\scal_\Sigma$ denotes the scalar curvature of $\Sigma$.
\end{lemma}

\begin{proof}
Put $\nu := \|\nabla G\|^{-1} \nabla G$.
As in the proof of Proposition \ref{prop: ricci curvature along regular level sets of s}, the traced Gauss equation and the Schoen-Yau rearrangement identity tell us that, on $\Sigma$, we have
\begin{equation} \label{eq: traced codazzi and schoen yau for level set of G}
    2\ric \left( \nu, \nu \right) = \scal - \scal_\Sigma + H^2 - \|\two\|^2,
\end{equation}
where $\two$ and $H$ are the second fundamental form and the mean curvature, respectively, of $\Sigma$, with respect to the unit normal vector field $\nu$.
As $\Sigma$ is a level set of $G$, we also know that
\begin{equation} \label{eq: second fundamental form and mean curvature of level set of G}
    \two = \frac{\hess{G}|_\Sigma}{\|\nabla G\|}, \quad H = \div (\nu).
\end{equation}
Since $\Delta G = 3kG/4$, we may explicitly compute
\begin{equation} \label{eq: mean curvature of the level set of G}
    H = \frac{\Delta G}{\|\nabla G\|} - \frac{\langle \nabla \|\nabla G\|, \nabla G \rangle }{\|\nabla G\|^2} = \frac{1}{\|\nabla G\|} \left( \frac{3k}{4} G - G_{11} \right).
\end{equation}

Now, let $(E_1, E_2, E_3)$ be a local orthonormal frame for $M$ with $E_1 = \nu$, and denote by $G_{ij} := \hess{G}(E_i, E_j)$; note that this notation is consistent with \eqref{eq: notation for G11}.
By \eqref{eq: second fundamental form and mean curvature of level set of G}, $\|\two\|^2 = \|\nabla G\|^{-2} \left( G_{22}^2 + 2 G_{23}^2 + G_{33}^2 \right)$, and so, together with \eqref{eq: mean curvature of the level set of G}, we have
\begin{equation} \label{eq: exact codazzi equation for G part 1}
    H^2 - \|\two\|^2 = \frac{1}{\|\nabla G\|^2} \left( \frac{9k^2}{16} G^2 - \frac{3k}{2} G_{11} G \right) + \frac{1}{\|\nabla G\|^2} \left( G_{11}^2 - G_{22}^2 - 2 G_{23}^2 - G_{33}^2 \right).
\end{equation}
The right-hand side of \eqref{eq: exact codazzi equation for G part 1} may be rewritten in terms of $\|\!\hess{G}\|^2$ and $\|\nabla \|\nabla G\|\|^2$ as follows: since $d \|\nabla G\| = \hess{G} \left( \nu, \cdot \right)$, we have $\|\nabla \|\nabla G\|\|^2 = G_{11}^2 + G_{12}^2 + G_{13}^2$, and so
\allowdisplaybreaks
\begin{align*}
    \|\!\hess{G}\|^2 &= G_{11}^2 + 2G_{12}^2 + 2 G_{13}^2 + (G_{22}^2 + 2 G_{23}^2 + G_{33}^2) \\
    &= 2 \|\nabla \|\nabla G\|\|^2 - G_{11}^2 + (G_{22}^2 + 2 G_{23}^2 + G_{33}^2).
\end{align*}
Substituting this into \eqref{eq: exact codazzi equation for G part 1}, we obtain that
\begin{equation} \label{eq: exact codazzi equation for G part 2}
    H^2 - \|\two\|^2 = \frac{1}{\|\nabla G\|^2} \left(  2 \|\nabla \|\nabla G\|\|^2 - \left( \|\!\hess{G}\|^2 - \frac{9k^2}{16} G^2 + \frac{3k}{2} G_{11} G \right) \right).
\end{equation}
The desired conclusion now follows by combining \eqref{eq: exact codazzi equation for G part 2} with \eqref{eq: traced codazzi and schoen yau for level set of G}.
\end{proof}

We continue with a sharp ``Kato identity'' for $G$, which we will write on the regular level sets of $G$.
This is essentially the analogue of Lemma \ref{lemma: norm of hessian of s squared along the regular level sets of s}.

\begin{lemma} \label{lemma: sharp kato in dimension three}
If $\Sigma$ is a regular level set of $G$, then
\[
\|\!\hess{G}\|^2 = \frac{3}{2} G_{11}^2 - \frac{3k}{4} G_{11} G + \frac{9k^2}{32} G^2 + 2 \|\nabla^\Sigma \|\nabla G\|\|^2 + \|(\hess{G}|_\Sigma)^\circ\|^2,
\]
where $\nabla^\Sigma$ denotes the tangential component along $\Sigma$ of the gradient and $(\hess{G}|_\Sigma)^\circ$ is the traceless part of $\hess{G}|_\Sigma$.
\end{lemma}

\begin{proof}
As in the proof of Lemma \ref{lemma: exact codazzi equation for G}, let $(E_1, E_2, E_3)$ be a local orthonormal frame for $M$ with $E_1 := \nabla G/\|\nabla G\|$, and put $G_{ij} := \hess{G}(E_i, E_j)$.
Then, $\|\nabla^\Sigma \|\nabla G\|\|^2 = G_{12}^2 + G_{13}^2$, and $\|\!\hess{G}|_\Sigma\|^2 = G_{22}^2 + 2 G_{23}^2 + G_{33}^2$, so, on $\Sigma$, we have
\begin{equation} \label{eq: hessian in terms of tangential components}
    \|\!\hess{G}\|^2 = G_{11}^2 + 2 \|\nabla^\Sigma \|\nabla G\|\|^2 + \|\!\hess{G}|_\Sigma\|^2.
\end{equation}
Since $\hess{G}|_\Sigma = (\hess{G}|_\Sigma)^\circ + (\tr_\Sigma (\hess{G}|_\Sigma)/2) g|_\Sigma$, we may also compute
\allowdisplaybreaks
\begin{align} \label{align: tangential hessian in terms of traceless component}
    \|\!\hess{G}|_\Sigma\|^2 &= \|(\hess{G}|_\Sigma)^\circ\|^2 + \frac{1}{2} (G_{22} + G_{33})^2 \nonumber \\
    &= \|(\hess{G}|_\Sigma)^\circ\|^2 + \frac{1}{2} (\Delta G - G_{11})^2 \nonumber \\
    &= \|(\hess{G}|_\Sigma)^\circ\|^2 + \frac{9k^2}{32} G^2 - \frac{3k}{4} G_{11} G + \frac{1}{2} G_{11}^2,
\end{align}
where the last equality follows because $\Delta G = 3kG/4$.
The desired result is now obtained by combining \eqref{eq: hessian in terms of tangential components} with \eqref{align: tangential hessian in terms of traceless component}.
\end{proof}

We also remark that, from the point of view of a Kato inequality for $G$, Lemma \ref{lemma: sharp kato in dimension three} implies that
\[
\|\!\hess{G}\|^2 \geq \frac{3}{2} \|\nabla \|\nabla G\|\|^2 - \frac{3k}{4} \|\nabla \|\nabla G\|\| G + \frac{9k^2}{32} G^2 = \frac{3}{2} \left( \|\nabla\|\nabla G\|\| - \frac{k}{4} G \right)^2 + \frac{3k^2}{16} G^2,
\]
which reduces to the usual Kato inequality for harmonic functions when $k = 0$, and a similar identity also holds in higher dimensions.
However, for the purpose of our monotonic quantity from Theorem \ref{intro-th: alternative sharp monotonic quantity} (cf.~also \eqref{eq: alternative sharp monotonic quantity expression via G}), it is more convenient to keep precise track of the mixed $G_{11} G$ term in $\|\!\hess{G}\|^2$.

With the above ingredients, we are now ready to prove \eqref{eq: laplace of G along its level sets for alternative monotonic quantity}.
This is an analogue of Proposition \ref{prop: divergence on regular level sets for sharp monotonic quantity} written only in terms of $G$.

\begin{proposition} \label{prop: sharp bochner identity via G for alternative monotonic quantity}
If $\Sigma$ is a regular level set of $G$, then
\begin{equation} \label{eq: sharp bochner identity on the level sets of G for alternative monotonic quantity}
    \begin{split}
        \Delta \|\nabla G\| &= - \frac{\|\nabla G\|}{2} \scal_\Sigma + \frac{15k}{4} \|\nabla G\| + \frac{3}{2} \frac{k - 2 G^2}{G(G^2 - k/4)} \|\nabla G\| (\Delta G - G_{11}) - \frac{3}{4} \frac{(k - 2 G^2)^2}{G^2(G^2 - k/4)^2} \|\nabla G\|^3 \\
        &\quad \, + \frac{\| \nabla^\Sigma \|\nabla G\| \|^2}{\|\nabla G\|} + \frac{\|(\hess{G}|_\Sigma)^\circ\|^2}{2 \|\nabla G\|} + \frac{\|\nabla G\|}{2} (\scal - 6k) \\
        &\quad \, + \frac{3}{4\|\nabla G\|} \left( \Delta G - G_{11} - \frac{k - 2 G^2}{G(G^2 - k/4)} \|\nabla G\|^2 \right)^2.
    \end{split}
\end{equation}
\end{proposition}

\begin{proof}
As $\Delta G = 3kG/4$, from the Bochner formula, we know that, away from the singularity of $G$, we have
\[
\frac{1}{2} \Delta \|\nabla G\|^2 = \|\!\hess{G}\|^2 + \frac{3k}{4} \|\nabla G\|^2 + \ric(\nabla G, \nabla G).
\]
Using the product rule on the left-hand side, this identity may be rewritten as
\begin{equation} \label{eq: sharp bochner inequality along level sets part 1}
    \Delta \|\nabla G\| = \frac{1}{\|\nabla G\|} \|\!\hess{G}\|^2 + \frac{3k}{4} \|\nabla G\| + \frac{1}{\|\nabla G\|} \ric(\nabla G, \nabla G) - \frac{1}{\|\nabla G\|} \|\nabla \|\nabla G\|\|^2.
\end{equation}
We use Lemma \ref{lemma: exact codazzi equation for G} to rewrite the last two terms from the right-hand side above and obtain that, on $\Sigma$,
\begin{equation}
    \Delta \|\nabla G\| = \frac{1}{2\|\nabla G\|} \|\!\hess{G}\|^2 + \frac{3k}{4} \|\nabla G\| + \frac{\|\nabla G\|}{2} \left( \scal - \scal_\Sigma \right) + \frac{9k^2}{32} \frac{G^2}{\|\nabla G\|} - \frac{3k}{4} \frac{G_{11} G}{\|\nabla G\|}.
\end{equation}
Now, we apply Lemma \ref{lemma: sharp kato in dimension three} for the Hessian term and recast the above equation as
\allowdisplaybreaks
\begin{align*}
    \Delta \|\nabla G\| &= \frac{3}{4 \|\nabla G\|} G_{11}^2  - \frac{3k}{8} \frac{G_{11} G}{\|\nabla G\|}  + \frac{9 k^2}{64} \frac{G^2}{\|\nabla G\|} + \frac{\|\nabla^\Sigma \|\nabla G\|\|^2}{\|\nabla G\|} + \frac{\|(\hess{G}|_\Sigma)^\circ\|^2}{2\|\nabla G\|} \nonumber \\
    &\quad \, + \frac{3k}{4} \|\nabla G\| + \frac{\|\nabla G\|}{2} \left( \scal - \scal_\Sigma \right) + \frac{9k^2}{32} \frac{G^2}{\|\nabla G\|} - \frac{3k}{4} \frac{G_{11} G}{\|\nabla G\|} \nonumber \\
    &= \frac{\|\nabla G\|}{2} \left( \scal - \scal_{\Sigma} \right) + \frac{3k}{4} \|\nabla G\| + \frac{3}{4 \|\nabla G\|} \left( G_{11} - \frac{3k}{4} G \right)^2 + \frac{\|\nabla^\Sigma \|\nabla G\|\|^2}{\|\nabla G\|} + \frac{\|(\hess{G}|_\Sigma)^\circ\|^2}{2\|\nabla G\|}.
\end{align*}
Finally, we add and subtract the term $3k \|\nabla G\|$ on the right-hand side and use that $\Delta G = 3kG/4$ to obtain the identity
\begin{equation} \label{eq: sharp bochner identity on the level sets of G for alternative monotonic quantity part 1}
    \begin{split}
        \Delta \|\nabla G\| &= - \frac{\|\nabla G\|}{2} \scal_\Sigma + \frac{15k}{4} \|\nabla G\| + \frac{3}{4\|\nabla G\|} \left( \Delta G - G_{11} \right)^2 \\
        &\quad \, + \frac{\| \nabla^\Sigma \|\nabla G\| \|^2}{\|\nabla G\|} + \frac{\|(\hess{G}|_\Sigma)^\circ\|^2}{2 \|\nabla G\|} + \frac{\|\nabla G\|}{2} (\scal - 6k).
    \end{split}
\end{equation}

Now, the idea is to rewrite the term $(\Delta G - G_{11})^2$ from the right-hand side of \eqref{eq: sharp bochner identity on the level sets of G for alternative monotonic quantity part 1} using only $G$ and $\nabla G$.
We proceed by looking at this term on the model space of the three-sphere of curvature $k$.
By {\cite[Lemma 2.1]{ManeaSGE}}, on the three-sphere, we have $G = (2 \snk(\dist(p, \cdot)/2))^{-1}$, where $\snk$ is the curvature-adapted sine function.
A direct computation then shows that, in this case, $\Delta G - G_{11} = -2 G^{-1} \|\nabla G\|^2 + kG/2$.
Moreover, on the three-sphere, by {\cite[Theorem 1.1]{ManeaSGE}}, we also have that $\|\nabla G\|^2 + kG^2/4 = G^4$; hence, we may rewrite
\begin{equation} \label{eq: Delta G - G11 on the three sphere for alternative monotonic quantity}
    \Delta G - G_{11} = \frac{k - 2G^2}{G(G^2 - k/4)} \|\nabla G\|^2.
\end{equation}
In general, since $G \geq \sqrt{k}/2$, we complete the square around this ``model'' equation, and we obtain that
\begin{equation} \label{eq: Delta G - G11 alternative in general}
    \begin{split}
        (\Delta G - G_{11})^2 &= \left( \Delta G - G_{11} - \frac{k - 2G^2}{G(G^2 - k/4)} \|\nabla G\|^2 \right)^2 + 2 \frac{k - 2G^2}{G(G^2 - k/4)} \|\nabla G\|^2 (\Delta G - G_{11}) \\
        &\quad \, - \frac{(k - 2G^2)^2}{G^2(G^2 - k/4)^2} \|\nabla G\|^4.
    \end{split}
\end{equation}
The proof of this proposition is finished by combining \eqref{eq: sharp bochner identity on the level sets of G for alternative monotonic quantity part 1} with \eqref{eq: Delta G - G11 alternative in general}.
\end{proof}

With the above, we are now ready to prove Theorem \ref{intro-th: alternative sharp monotonic quantity}.

\begin{proof}[Proof of Theorem \ref{intro-th: alternative sharp monotonic quantity}]
As mentioned at the beginning of this section (cf.~\eqref{eq: alternative sharp monotonic quantity expression via G}), it is equivalent to show that
\begin{equation} \label{eq: start of proof of alternative monotonic quantity}
    \begin{split}
        \frac{\dd}{\dd r} \bigg( \frac{1}{r} &\int_{G = r} \|\nabla G\|^2 + \int_{G \geq r} \frac{3k(2G^2 + k - 6r^2)}{G^2(4G^2 - k)^2} \|\nabla G\|^3 \\
        &+ \int_{G \geq r} \left( \frac{3k}{4} - \frac{3k(G^2 - r^2)(4G^2 + 5k)}{8G^2(4G^2 - k)} \right) \|\nabla G\| - 4 \pi r \bigg) \leq 0,
    \end{split}
\end{equation}
with equality holding at some $r > \sqrt{k}/2$ only if $(M, g)$ is isometric to the three-sphere of curvature $k$.

As in the proof of Theorem \ref{intro-th: sharp monotonic quantity}, we start by differentiating the ``energy term'' $\int_{G = r} \|\nabla G\|^2$ and computing the integral of $\Delta \|\nabla G\|$ over $\{G \geq r\}$ using the coarea formula and \eqref{eq: sharp bochner identity on the level sets of G for alternative monotonic quantity}; the two other integral terms from \eqref{eq: start of proof of alternative monotonic quantity} encompass the terms in \eqref{eq: sharp bochner identity on the level sets of G for alternative monotonic quantity} that do not depend on the extrinsic geometry of the regular level sets of $G$.
As such, denote by
\begin{equation} \label{eq: definition of E and A for G and alternative quantity}
    E(r) := \int_{G = r} \|\nabla G\|^2, \qquad A(r) := \mathrm{area}(G = r),
\end{equation}
where, as before, $\mathrm{area}(G = r)$ denotes the two-dimensional Hausdorff measure of the level set $\{G = r\}$.
Lemma \ref{appendix-lemma: local AC of energy term} (or rather, an analogous proof) shows that $E$ is locally absolutely continuous, with
\begin{equation} \label{eq: derivative of energy via G for alternative quantity}
    E'(r) = \int_{G = r} \frac{\div(\|\nabla G\| \nabla G)}{\|\nabla G\|} = \int_{G = r} (\Delta G + G_{11}) = \frac{3k}{4} r A(r) + \int_{G = r} G_{11}
\end{equation}
whenever $r$ is a regular value of $G$.

We now claim that whenever $r \geq \sqrt{k}/2$ is a regular value of $G$, we have the differential identity
\begin{equation} \label{eq: sharp equation between E, A, and RHS of bochner}
    \frac{1}{r^2} E'(r) - \frac{3k}{4r} A(r) = - \frac{2}{r^3} E(r) + 6 \int_r^\infty \frac{1}{t^4} E(t) \dd t - \frac{3k}{2} \int_r^\infty \frac{1}{t^2} A(t) \dd t - \int_r^\infty \frac{1}{t^2} \left( \int_{G = t} \frac{\Delta \|\nabla G\|}{\|\nabla G\|} \right) \dd t.
\end{equation}
This is the analogue of \eqref{eq: proof of sharp monotonic quantity with s part 8} from the proof of Theorem \ref{intro-th: sharp monotonic quantity}.
Let us prove this claim.
As in {\cite[Proof of Theorem 3.1]{MunteanuWangComparisonPaper}}, by the product rule and using that $\Delta G = 3kG/4$, we have
\begin{equation} \label{eq: sharp monotonicity part 1}
    \Delta \left( \frac{\|\nabla G\|}{G^2} \right) = \frac{\Delta \|\nabla G\|}{G^2} + 6 \frac{\|\nabla G\|^3}{G^4} - \frac{3k}{2} \frac{\|\nabla G\|}{G^2} - \frac{4}{G^3} \langle \nabla \|\nabla G\|, \nabla G \rangle.
\end{equation}
We now integrate this equation over the set $\{r \leq G \leq s\}$ when $\{G = s\}$ is a regular level set of $G$, apply the divergence theorem on the left-hand side, and let $s \to \infty$.
By the asymptotics of $G$ near its singularity (cf.~Lemma \ref{appendix-lemma: asymptotics of greens function near singularity}), the boundary term near the singularity converges to zero, \ie 
\[
\lim_{s \to \infty} \int_{G = s} \left\langle \nabla \left( \frac{\|\nabla G\|}{G^2} \right), \frac{\nabla G}{\|\nabla G\|} \right\rangle = \lim_{s \to \infty} \int_{G = s} \left( \frac{G_{11}}{G^2} - 2 \frac{\|\nabla G\|^2}{G^3} \right) = 0.
\]
From this and the Divergence Theorem, it follows that
\[
\int_{G \geq r} \Delta \left( \frac{\|\nabla G\|}{G^2} \right) = \int_{G = r} \left\langle \nabla \left( \frac{\|\nabla G\|}{G^2} \right), - \frac{\nabla G}{\|\nabla G\|} \right\rangle = - \frac{1}{r^2} \int_{G = r} G_{11} + \frac{2}{r^3} \int_{G = r} \|\nabla G\|^2,
\]
and equating this with the right-hand side of \eqref{eq: sharp monotonicity part 1}, we obtain the equation
\begin{equation} \label{eq: sharp monotonicity part 2}
    - \frac{1}{r^2} \int_{G = r} G_{11} + \frac{2}{r^3} \int_{G = r} \|\nabla G\|^2 = \int_{G \geq r} \left( \frac{\Delta \|\nabla G\|}{G^2} + 6 \frac{\|\nabla G\|^3}{G^4} - \frac{3k}{2} \frac{\|\nabla G\|}{G^2} - \frac{4}{G^3} \langle \nabla \|\nabla G\|, \nabla G \rangle \right).
\end{equation}
We now rewrite all terms in the above identity using only $E$, $A$, and $\Delta \|\nabla G\|$.
On the one hand, by \eqref{eq: derivative of energy via G for alternative quantity} and the definition of $E$ from \eqref{eq: definition of E and A for G and alternative quantity}, we have
\begin{equation} \label{eq: sharp monotonicity part 3}
    - \frac{1}{r^2} \int_{G = r} G_{11} + \frac{2}{r^3} \int_{G = r} \|\nabla G\|^2 = - \frac{1}{r^2}E'(r) + \frac{3k}{4r} A(r) + \frac{2}{r^3} E(r).
\end{equation}
On the other hand, by the coarea formula, we have
\allowdisplaybreaks
\begin{align} \label{align: sharp monotonicity part 4}
    &\int_{G \geq r} \left( \frac{\Delta \|\nabla G\|}{G^2} + 6 \frac{\|\nabla G\|^3}{G^4} - \frac{3k}{2} \frac{\|\nabla G\|}{G^2} - \frac{4}{G^3} \langle \nabla \|\nabla G\|, \nabla G \rangle \right) \nonumber \\
    &= \int_r^\infty \int_{G = t} \left( \frac{1}{t^2} \frac{\Delta \|\nabla G\|}{\|\nabla G\|} + \frac{6}{t^4} \|\nabla G\|^2 - \frac{3k}{2t^2} - \frac{4}{t^3} G_{11} \right) \dd t \nonumber \\
    &= \int_r^\infty \frac{1}{t^2} \left( \int_{G = t} \frac{\Delta \|\nabla G\|}{\|\nabla G\|} \right) \dd t + 6 \int_r^\infty \frac{1}{t^4} E(t) \dd t - \frac{3k}{2} \int_r^\infty \frac{1}{t^2} A(t) \dd t - 4 \int_r^\infty \frac{1}{t^3} \left( \int_{G = t} G_{11} \right) \dd t,
\end{align}
where in the last equality we used the definitions of $E$ and $A$ from \eqref{eq: definition of E and A for G and alternative quantity}.
We now use \eqref{eq: derivative of energy via G for alternative quantity} and integration by parts (with the remark that $t^{-3}E(t) \to 0$ as $t \to \infty$, by the asymptotics of $G$ near its singularity -- Lemma \ref{appendix-lemma: asymptotics of greens function near singularity}) to further rewrite
\allowdisplaybreaks
\begin{align*}
    - 4 \int_r^\infty \frac{1}{t^3} \left( \int_{G = t} G_{11} \right) \dd t &= - 4 \int_r^\infty \frac{1}{t^3} \left( E'(t) - \frac{3k}{4} t A(t) \right) \dd t \\
    &= \frac{4}{r^3} E(r) - 12 \int_r^\infty \frac{1}{t^4} E(t) \dd t + 3k \int_r^\infty \frac{1}{t^2} A(t) \dd t.
\end{align*}
The proof of the claim is now finished by substituting the above equation into \eqref{align: sharp monotonicity part 4} and then combining the resulting equation with \eqref{eq: sharp monotonicity part 2} and \eqref{eq: sharp monotonicity part 3}.

Now, for ease of notation, put $f(t) := t^{-1} (t^2 - k/4)^{-1} (k - 2t^2)$.
We now substitute \eqref{eq: sharp bochner identity on the level sets of G for alternative monotonic quantity} from Proposition \ref{prop: sharp bochner identity via G for alternative monotonic quantity} into the identity from \eqref{eq: sharp equation between E, A, and RHS of bochner}.
After using the Gauss-Bonnet Theorem and recalling the definitions of $E$ and $A$ from \eqref{eq: definition of E and A for G and alternative quantity}, we obtain the following equation:
\begin{equation} \label{eq: alternative almost final monotonic quantity via G part 0}
    \begin{split}
        \frac{1}{r^2} E'(r) - \frac{3k}{4r} A(r) &= - \frac{2}{r^3} E(r) + 6 \int_r^\infty \frac{1}{t^4} E(t) \dd t - \frac{3k}{2} \int_r^\infty \frac{1}{t^2} A(t) \dd t \\
        &\quad \, - \int_r^\infty \frac{1}{t^2} \left( - 2 \pi \chi(G = t) + \frac{15k}{4} A(t) + \frac{3}{2} f(t) \int_{G = t} (\Delta G - G_{11}) - \frac{3}{4} f(t)^2 E(t) \right) \dd t \\
        &\quad \, - \int_r^\infty \frac{1}{t^2} X(t) \dd t
    \end{split}
\end{equation}
where
\begin{equation} \label{eq: alternative excess to the 3 sphere}
    \begin{split}
        X(t) := \int_{G = t} \bigg( &\frac{\| \nabla^{\{G = t\}} \|\nabla G\| \|^2}{\|\nabla G\|^2} + \frac{\|(\hess{G}|_{\{G = t\}})^\circ\|^2}{2 \|\nabla G\|^2} + \frac{\scal - 6k}{2} \\
        &+ \frac{3}{4 \|\nabla G\|^2} (\Delta G - G_{11} - f(G) \|\nabla G\|^2)^2 \bigg)
    \end{split}
\end{equation}
is the ``geometric excess'' with respect to the three-sphere of curvature $k$ from the point of view of the geometry of the foliation of $M \setminus \{p\}$ by the level sets of $G$.
Note that each term in the integrand from the definition of $X$ is non-negative, and so, in particular, $X(t) \geq 0$.

Recalling that $\Delta G = 3kG/4$ on $M \setminus \{p\}$ and using \eqref{eq: derivative of energy via G for alternative quantity} to replace the $G_{11}$ term in \eqref{eq: alternative almost final monotonic quantity via G part 0}, the above equation may be rewritten as
\begin{equation} \label{eq: alternative almost final monotonic quantity via G part 1}
    \begin{split}
        \frac{1}{r^2} E'(r) - \frac{3k}{4r} A(r) &= - \frac{2}{r^3} E(r) + 6 \int_r^\infty \frac{1}{t^4} E(t) \dd t - \frac{3k}{2} \int_r^\infty \frac{1}{t^2} A(t) \dd t + \int_r^\infty \frac{2 \pi \chi(G = t)}{t^2} \dd t \\
        &\quad \, - \int_r^\infty \frac{1}{t^2} \left( \frac{15k}{4} + \frac{9k}{4} t f(t) \right) A(t) \dd t + \frac{3}{2} \int_r^\infty \frac{f(t)}{t^2} E'(t) \dd t + \frac{3}{4} \int_r^\infty \frac{f(t)^2}{t^2} E(t) \dd t \\
        &\quad \, - \int_r^\infty \frac{1}{t^2} X(t) \dd t.
    \end{split}
\end{equation}
Now, since $f(t) = -2t^{-1} + O(t^{-3})$ as $t \to \infty$ and $E(t) \sim 4 \pi t^2$ as $t \to \infty$ (by the asymptotics of $G$ near its singularity -- Lemma \ref{appendix-lemma: asymptotics of greens function near singularity}), we may integrate by parts the term
\[
\frac{3}{2} \int_r^\infty \frac{f(t)}{t^2} E'(t) \dd t = - \frac{3f(r)}{2r^2} E(r) - \int_r^\infty E(t) \frac{d}{dt} \left( \frac{3f(t)}{2t^2} \right) \dd t,
\]
and so \eqref{eq: alternative almost final monotonic quantity via G part 1} becomes
\begin{equation} \label{eq: alternative almost final monotonic quantity via G part 2}
    \begin{split}
        \frac{1}{r^2} E'(r) - \frac{3k}{4r} A(r) &= - \left( \frac{2}{r^3} + \frac{3f(r)}{2r^2} \right) E(r) - \int_r^\infty \frac{1}{t^2} \left( \frac{21k}{4} + \frac{9k}{4} t f(t) \right) A(t) \dd t + 2\pi \int_r^\infty \frac{\chi(G = t)}{t^2} \dd t \\
        &\quad \, + \int_r^\infty \left( \frac{6}{t^4} - \frac{d}{dt} \left( \frac{3f(t)}{2t^2} \right) + \frac{3f(t)^2}{4t^2} \right) E(t) \dd t - \int_r^\infty \frac{1}{t^2} X(t) \dd t.
    \end{split}
\end{equation}
A direct computation then shows that
\[
- \frac{1}{t^2} \left( \frac{21k}{4} + \frac{9k}{4} t f(t) \right) = - \frac{3k(4t^2 + 5k)}{4t^2(4t^2 - k)}, \qquad \frac{6}{t^4} - \frac{d}{dt} \left( \frac{3f(t)}{2t^2} \right) + \frac{3f(t)^2}{4t^2} = \frac{36k}{t^2(4t^2 - k)^2},
\]
so \eqref{eq: alternative almost final monotonic quantity via G part 2} may be rewritten as
\begin{equation} \label{eq: alternative almost final monotonic quantity via G part 3}
    \begin{split}
        \frac{1}{r^2} E'(r) - \frac{3k}{4r} A(r) &= - \left( \frac{2}{r^3} + \frac{3f(r)}{2r^2} \right) E(r) - \int_r^\infty \frac{3k(4t^2 + 5k)}{4t^2(4t^2 - k)} A(t) \dd t + \int_r^\infty \frac{2\pi \chi(G = t)}{t^2} \dd t \\
        &\quad \, + \int_r^\infty \frac{36k}{t^2(4t^2 - k)^2} E(t) \dd t - \int_r^\infty \frac{1}{t^2} X(t) \dd t.
    \end{split}
\end{equation}

To finish the proof of the monotonicity statement from Theorem \ref{intro-th: alternative sharp monotonic quantity}, we proceed as follows.
Consider the functions
\[
u_1(t, r) := \frac{3k(2t^2 + k - 6r^2)}{t^2(4t^2 - k)^2}, \qquad u_2(t, r) = \frac{3k}{4} - \frac{3k(t^2 - r^2)(4t^2 + 5k)}{8t^2(4t^2 - k)}.
\]
After applying the coarea formula, the statement from \eqref{eq: start of proof of alternative monotonic quantity} is equivalent to the identity
\[
\frac{d}{dr} \left( \frac{1}{r} E(r) + \int_r^\infty u_1(t, r) E(t) \dd t + \int_r^\infty u_2(t, r) A(t) \dd t - 4 \pi r \right) \leq 0,
\]
together with a rigidity statement, where $E$ and $A$ are defined in \eqref{eq: definition of E and A for G and alternative quantity}.
To obtain this identity, note first that multiplying \eqref{eq: alternative almost final monotonic quantity via G part 3} by $r$ and subtracting the term $r^{-2}E(r)$ from both sides gives us the equation
\begin{equation} \label{eq: alternative almost final monotonic quantity via G part 4}
    \begin{split}
        \frac{d}{dr }\left( \frac{E(r)}{r} \right) &= - \frac{3k}{r^2(4r^2 - k)} E(r) + \int_r^\infty \frac{36k r }{t^2(4t^2 - k)^2} E(t) \dd t + \frac{3k}{4} A(r) - \int_r^\infty \frac{3kr(4t^2 + 5k)}{4t^2(4t^2 - k)} A(t) \dd t \\
        &\quad \, + r \int_r^\infty \frac{2 \pi \chi(G = t)}{t^2} \dd t - r \int_r^\infty \frac{1}{t^2} X(t) \dd t.
    \end{split}
\end{equation}
Since the functions $u_1$ and $u_2$ are chosen so that
\[
\frac{\partial u_1}{\partial r} = - \frac{36kr}{t^2(4t^2 - k)^2}, \qquad u_1(r, r) = - \frac{3k}{r^2(4r^2 - k)}
\]
and
\[
\frac{\partial u_2}{\partial r} = \frac{3kr(4t^2 + 5k)}{4t^2(4t^2 - k)}, \qquad u_2(r, r) = \frac{3k}{4},
\]
combining these identities with \eqref{eq: alternative almost final monotonic quantity via G part 4}, we immediately obtain that
\begin{equation} \label{eq: alternative monotonic quantity final equation via G}
    \frac{d}{dr} \left( \frac{E(r)}{r} + \int_r^\infty u_1(t, r) E(t) \dd t + \int_r^\infty u_2(t, r) A(t) \dd t \right) = r \int_r^\infty \frac{2 \pi \chi(G = t)}{t^2} \dd t - r \int_r^\infty \frac{1}{t^2} X(t) \dd t.
\end{equation}
Finally, under the assumption that all regular level sets of $G$ are path-connected, it follows that $\chi(G = t) \leq 2$ for almost every $t$ in the range of $G$, and thus
\begin{equation} \label{eq: alternative monotonic quantity final equation via G 2}
    \frac{d}{dr} \left( \frac{E(r)}{r} + \int_r^\infty u_1(t, r) E(t) \dd t + \int_r^\infty u_2(t, r) A(t) \dd t - 4 \pi r \right) =  r \int_r^\infty \frac{2 \pi (\chi(G = t) - 2)}{t^2} \dd t - r \int_r^\infty \frac{1}{t^2} X(t) \dd t,
\end{equation}
which completes the proof of the monotonicity part of Theorem \ref{intro-th: alternative sharp monotonic quantity} (cf.~\eqref{eq: start of proof of alternative monotonic quantity}), as the above right-hand side is non-positive.

For the corresponding rigidity statement, we proceed in a similar manner as we did for Theorem \ref{intro-th: sharp monotonic quantity}.
Suppose that there is some $r_0 > \sqrt{k}/2$ so that the above derivative is equal to zero at $r_0$.
Equation \eqref{eq: alternative monotonic quantity final equation via G 2} then tells us that $\chi(G = t) = 2$ and $X(t) = 0$ for almost every $t \geq r_0$.
However, by definition of $X(t)$ (cf.~\eqref{eq: alternative excess to the 3 sphere}), $X(t) = 0$ implies that, on $\{G = t\}$, we have
\begin{equation} \label{eq: alternative monotonicity rigidity via G}
    \frac{\| \nabla^{\{G = t\}} \|\nabla G\| \|^2}{\|\nabla G\|^2} = \frac{\|(\hess{G}|_{\{G = t\}})^\circ \|^2}{2 \|\nabla G\|^2} = \frac{1}{2} (\scal - 6k) = \frac{3}{4\|\nabla G\|^2} \left( (\Delta G - G_{11}) - f(G)\|\nabla G\|^2 \right)^2 = 0.
\end{equation}
The first equality from \eqref{eq: alternative monotonicity rigidity via G} (together with the connectedness of the regular level sets of $G$) tells us that $\|\nabla G\|$ is constant on $\{G = t\}$, whilst the second one tells us that $\hess{G}|_{\{G = t\}} = (\tr_\Sigma(\hess{G}|_{\{G = t\}})/2) g|_{\{G = t\}} = ((\Delta G - G_{11})/2) g|_{\{G = t\}}$.
Combining this with the last equality from \eqref{eq: alternative monotonicity rigidity via G} and recalling that $\Delta G = 3kG/4$ on $M \setminus \{p\}$, we obtain that, on $\{G = t\}$, we have
\begin{equation} \label{eq: alternative monotonicity rigidity via G part 2}
    \hess{G} = \frac{1}{\|\nabla G\|^2} \left( \frac{3k}{4} G - f(G) \|\nabla G\|^2 \right) dG^2 + \frac{1}{2} f(G) \|\nabla G\|^2 g|_{\{G = t\}}.
\end{equation}
A direct computation now shows that the function $G^{-8} (G^2 - k/4)^2 (\|\nabla G\|^2 + kG^2/4 - G^4)$ has zero derivative at every point of $\{G = t\}$.
Indeed, by \eqref{eq: alternative monotonicity rigidity via G part 2}, it follows that
\[
d \|\nabla G\|^2 = 2 \hess{G}(\nabla G, \cdot) = \left( \frac{3k}{2} G - 2 f(G) \|\nabla G\|^2 \right) dG,
\]
and the desired conclusion is now implied by the chain rule, as we now have
\[
d \left( \|\nabla G\|^2 + \frac{k}{4} G^2 - G^4 \right) = - 2 f(G) \left( \|\nabla G\|^2 + \frac{k}{4} G^2 - G^4 \right) dG
\]
and $d (G^{-8} (G^2 - k/4)^2) = 2 f(G) G^{-8} (G^2 - k/4)^2 dG$.

Since $G^{-8} (G^2 - k/4)^2 (\|\nabla G\|^2 + kG^2/4 - G^4)$ is a smooth function on $\{G > \sqrt{k}/2\}$, the density of the regular points of $G$ (cf.~Lemmas \ref{appendix-lemma: regular values of G} and \ref{appendix-lemma: nodal sets have finite area}) implies that this function must have zero derivative at all points in $\{G > r_0\}$.
Since $\{G > r_0\}$ is a connected subset of $M$ (cf.~the paragraph before \eqref{eq: sharp monotonic quantity rigidity equation for gradient} at the end of the proof of Theorem \ref{intro-th: sharp monotonic quantity}), it follows that $G^{-8} (G^2 - k/4)^2 (\|\nabla G\|^2 + kG^2/4 - G^4)$ is constant on $\{G > r_0\}$.

Combining the previous paragraph with the asymptotics of $G$ near $p$ (cf.~Lemma \ref{appendix-lemma: asymptotics of greens function near singularity}) and the inequality $G^{-8} (G^2 - k/4)^2 > 0$, we obtain that $\|\nabla G\|^2 + kG^2/4 = G^4$ on $\{G > r_0\}$.
In terms of $s \equiv G^{-1}$, this means that $\|\nabla s\|^2 + ks^2/4 = 1$ on $\{G > r_0\} = \{s < 1/r_0\}$, and inserting this into \eqref{eq: alternative monotonicity rigidity via G part 2} allows us to conclude that $\hess{s^2} = (2 - ks^2)g$ on $\{G > r_0\}$.
We may now apply verbatim the argument from \eqref{eq: sharp monotonic quantity rigidity equation for gradient} onwards in the proof of Theorem \ref{intro-th: sharp monotonic quantity}, and we may conclude that $(M, g)$ must be isometric to the three-sphere of curvature $k$.
\end{proof}
\section{Applications of Theorems \ref{intro-th: sharp monotonic quantity} and \ref{intro-th: alternative sharp monotonic quantity}} \label{sec: sharp monotonic quantity applications}

In this section, we discuss some applications of Theorems \ref{intro-th: sharp monotonic quantity} and \ref{intro-th: alternative sharp monotonic quantity}.
As before, let $(M, g)$ be a closed, connected, Riemannian three-manifold with $\scal \geq 6k$ for some $k > 0$; fix a point $p \in M$ and let $G$ be Green's function for the operator $-\Delta + 3k/4$, with singularity at $p$, and put $s := G^{-1}$.
We will describe how both theorems may be viewed as establishing the monotonicity of an ``enclosed mass'' for the superlevel sets of $s$ (in the ``conformal'' blow-up $(M \setminus \{p\}, G^4 g)$), and how the quantitative expressions for the derivatives of the monotonic quantities may be used to obtain mass-controlled, weighted Cohn-Vossen bounds for the scalar curvature.

We start with the consequences of Theorem \ref{intro-th: sharp monotonic quantity}.
Let us assume, for the rest of this section, that all regular level sets of $G$ are path-connected, and that all integral terms from the statement of Theorem \ref{intro-th: sharp monotonic quantity} are finite.
Let us compute the asymptotic expansion of the monotonic quantity from \eqref{intro-eq: sharp monotonic quantity} in Theorem \ref{intro-th: sharp monotonic quantity} as $r \downarrow 0$ (cf.~\eqref{intro-eq: asymptotics of sharp monotonic quantity near the pole gives the mass}), and explain how the ADM mass of the manifold $(M \setminus \{p\}, G^4 g)$ is recovered from this.
By Lemma \ref{appendix-lemma: asymptotics of greens function near singularity}, near $p$, we have
\begin{equation} \label{eq: ADM mass asymptotics of greens function}
    s = \dist(p, \cdot) - A \, \dist(p, \cdot)^2 + O(\dist(p, \cdot)^3), \qquad \|\nabla s\| =  1 - 2A \, \dist(p, \cdot) + O(\dist(p, \cdot)^2),
\end{equation}
where $A \geq 0$ is a constant equal to half the ADM mass of the asymptotically flat manifold $(M \setminus \{p\}, G^4 g)$, which we denote by $m(M, p)$.
Inverting the first asymptotic relation from \eqref{eq: ADM mass asymptotics of greens function}, we obtain that $\dist(p, \cdot) = s(1 + As + O(s^2))$ as $s \downarrow 0$, \ie that $\dist(p, \cdot) = r(1 + Ar + O(r^2))$ on $\{s = r\}$.
This further implies that, on $\{s = r\}$, we have
\begin{equation} \label{eq: asymptotics of nabla s squared on the level set s = r}
    \|\nabla s\|^2 = 1 - 4 A \, \dist(p, \cdot) + O(\dist(p, \cdot)^2) = 1 - 4A r + O(r^2).
\end{equation}
An analogous argument to that from {\cite[Proposition A.3]{ManeaSGE}} shows that the area of the level set $\{s = r\}$ is asymptotic to area of the metric sphere $\partial B(p, r(1 + Ar + O(r^2)))$, which gives us the expansion
\begin{equation} \label{eq: asymptotics of area of small metric sphere}
    \mathrm{area}(s = r) = 4 \pi r^2(1 + Ar + O(r^2))^2 + O(r^4) = 4 \pi r^2(1 + 2Ar + O(r^2))
\end{equation}
Combining \eqref{eq: asymptotics of area of small metric sphere} with \eqref{eq: asymptotics of nabla s squared on the level set s = r}, we obtain that, as $r \downarrow 0$,
\begin{equation} \label{eq: ADM mass asymptotics of energy term}
    \frac{1}{r^3} \int_{s = r} \|\nabla s\|^2 = \frac{1}{r^3} \cdot 4 \pi r^2(1 + 2 A r + O(r^2)) \cdot (1 - 4Ar + O(r^2)) = \frac{4 \pi}{r} - 8 \pi A + O(r).
\end{equation}

This computes the asymptotics of the first integral term from \eqref{intro-eq: sharp monotonic quantity}. 
For the other terms, we proceed as follows.
By the coarea formula, $\int_{s \leq r} \|\nabla s\| = \int_0^r \mathrm{area}(s = t) \dd t$, and so, by \eqref{eq: asymptotics of area of small metric sphere}, it follows that
\begin{equation} \label{eq: ADM mass asymptotics of 2nd integral term in sharp monotonic quantity}
    \frac{3k}{2r^2} \int_0^r \mathrm{area}(s = t) \dd t = \frac{3k}{2r^2} \int_0^r \left( 4 \pi t^2 + O(t^3) \right) \dd t = O(r)
\end{equation}
as $r \downarrow 0$.
Similarly, by \eqref{eq: asymptotics of nabla s squared on the level set s = r}, \eqref{eq: asymptotics of area of small metric sphere}, and the coarea formula,
\allowdisplaybreaks
\begin{align} \label{align: ADM mass asymptotics of 3rd integral term in sharp monotonic quantity}
    \frac{3k^2}{32r^2} \int_{s \leq r} \frac{r^2 - s^2}{\|\nabla s\|} &= \frac{3k^2}{32r^2} \int_0^r (r^2 - t^2) \left( \int_{s = t} \frac{1}{\|\nabla s\|^2} \right) \dd t \nonumber \\
    &= \frac{3k^2}{32r^2} \int_0^r (r^2 - t^2) \cdot 4 \pi t^2 (1 + 2At + O(t^2)) \cdot (1 + 4A t + O(t^2)) \dd t \nonumber \\
    &= O(r^3).
\end{align}
Combining \eqref{eq: ADM mass asymptotics of energy term} with \eqref{eq: ADM mass asymptotics of 2nd integral term in sharp monotonic quantity} and \eqref{align: ADM mass asymptotics of 3rd integral term in sharp monotonic quantity} and recalling that $2A = m(M, p)$, we obtain a proof for \eqref{intro-eq: asymptotics of sharp monotonic quantity near the pole gives the mass}.

\begin{corollary} \label{cllr: ADM mass via sharp monotonic quantity}
We have
\[
m(M, p) = - \frac{1}{4\pi} \lim_{r \downarrow 0} \left( \frac{1}{r^3} \int_{s = r} \|\nabla s\|^2 + \frac{3k}{2r^2} \int_{s \leq r} \|\nabla s\| + \frac{3k^2}{32 r^2} \int_{s \leq r} \frac{r^2 - s^2}{\|\nabla s\|} - \frac{4 \pi}{r} \right).
\]
\end{corollary}

Thus, the asymptotics of the quantity from Theorem \ref{intro-th: sharp monotonic quantity} recover the ADM mass of $(M \setminus \{ p \}, G^4 g)$.
If we define
\begin{equation} \label{eq: quasi-local mass using sharp monotonic quantity}
    m(r) := - \frac{1}{4 \pi r^3} \int_{s = r} \|\nabla s\|^2 - \frac{3\pi k}{8 \pi r^2} \int_{s \leq r} \|\nabla s\| - \frac{3 \pi k^2}{128 \pi r^2} \int_{s \leq r} \frac{r^2 - s^2}{\|\nabla s\|} + \frac{1}{r},
\end{equation}
then Theorem \ref{intro-th: sharp monotonic quantity} shows that $m' \leq 0$ almost everywhere.
Since $m$ is locally absolutely continuous (cf.~Appendix \ref{appendix-sec: integration over the (super)level sets of greens function}) with $m(r) \to m(M, p)$ as $r \downarrow 0$ (by Corollary \ref{cllr: ADM mass via sharp monotonic quantity}), the Monotone Convergence Theorem allows us to write $m(r) - m(M, p) = \int_0^r m'(t) \dd t$, and thus, by \eqref{intro-eq: quantitative derivative of sharp monotonic quantity}, we have
\allowdisplaybreaks
\begin{align} \label{align: mass and quasilocal mass difference via sharp monotonic quantity}
    m(r) - m(M, p) &= - \frac{1}{4 \pi} \int_0^r \frac{1}{t^3} \int_0^t \left[ 2\pi (2 - \chi(s = u)) + \int_{s = u} \left( \frac{1}{2} (\scal - 6k) + \frac{1}{2\|\nabla s^2\|^2} \|\!\tracelessHess{s^2} \|^2 \right) \dd u \right] \dd t \nonumber \\
    &= \frac{1}{8 \pi} \int_0^r \left( \frac{1}{u^2} - \frac{1}{r^2} \right) \left[ 2\pi (2 - \chi(s = u)) + \int_{s = u} \left( \frac{1}{2} (\scal - 6k) + \frac{1}{2\|\nabla s^2\|^2} \|\!\tracelessHess{s^2} \|^2 \right) \right] \dd u,
\end{align}
where the last equality follows from Tonelli's theorem.
Evaluating \eqref{align: mass and quasilocal mass difference via sharp monotonic quantity} at $r = \max_M s$ gives an expression for $m(M, p)$ in terms of integrals involving only Green's function $G$ and the scalar curvature of $M$.
This is in a similar spirit to {\cite[Theorem 1.1]{GurskyMalchiodi}}, where the authors obtained a description of the ADM mass of closed, Einstein four-manifolds with positive scalar curvature in terms of integrals involving Green's function of the conformal Laplacian.

We can also interpret $m(r)$ as the ``enclosed mass'' of the region $\{s \geq r\}$, or as a ``localized mass at scale $r$'' (away from the ``exterior region'' $\{s \leq r\}$), and the right-hand side of \eqref{align: mass and quasilocal mass difference via sharp monotonic quantity} as a geometric expression for the difference between the ADM mass of $(M \setminus \{p\}, G^4 g)$ and the corresponding localization.
This quantity also bears similarities to a \emph{quasi-local mass} for domains in asymptotically flat manifolds.
Indeed, in {\cite[Section 1]{AndersonQuasiLocalMass}}, Anderson described five properties that are desired for a notion of quasi-local mass.
In Anderson's notation, our quantity $m$ is only defined for the superlevel sets of $s$, but satisfies QL0, QL2, and QL3 for these subsets; Theorem \ref{intro-th: sharp monotonic quantity} may be seen as the proof of property QL2, the monotonicity of the enclosed mass.
Moreover, by the Positive Mass Theorem, $m(M, p) \geq 0$, so $m(r) \geq 0$ at least for $r$ close to zero, which is the content of QL1.
However, the quantity $m(r)$ does not satisfy property QL4, since it does not depend only on the geometry of the level set $\{s = r\}$, but rather on the entire ``exterior region'' $\{s \leq r\}$.

Another simple application we may obtain using Theorem \ref{intro-th: sharp monotonic quantity} is a ``mass-controlled'' weighted Cohn-Vossen bound for the excess of the scalar curvature of $M$ (with respect to the three-sphere), analogous to {\cite[Corollary 0.6]{CMGradientEsimatesScalarCurvature}}.
Specifically, \eqref{align: mass and quasilocal mass difference via sharp monotonic quantity} immediately implies the following bound.

\begin{corollary} \label{cllr: cohn-vossen bound for scalar curvature using sharp monotonic quantity}
For almost every $r$ in the range of $s$, we have
\[
\int_0^r \left( \frac{1}{u^2} - \frac{1}{r^2} \right) \left( \int_{s = u} (\scal - 6k) \right) \dd u \leq 16 \pi (m(M, p) - m(r)).
\]
\end{corollary}

This concludes the discussion for Theorem \ref{intro-th: sharp monotonic quantity}.
One can also prove similar applications as those above for Theorem \ref{intro-th: alternative sharp monotonic quantity}; we will only sketch these below.
As before, using the coarea formula, we may rewrite
\[
\int_{s \leq r} \frac{3k(2 + ks^2 - 6 r^{-2} s^2)}{s^2(4 - ks^2)^2} \|\nabla s\|^3 = \int_0^r \frac{3k(2 + kt^2 - 6r^{-2} t^2)}{t^2(4 - kt^2)^2} \left( \int_{s = t} \|\nabla s\|^2 \right) \dd t
\]
and
\[
\int_{s \leq r} \left( \frac{3k}{4} - \frac{3k(r^2 - s^2)(4 + 5k s^2)}{8r^2(4 - ks^2)} \right) \frac{\|\nabla s\|}{s^2} = \int_0^r \left( \frac{3k}{4} - \frac{3k(r^2 - t^2)(4 + 5k t^2)}{8r^2(4 - kt^2)} \right) \frac{\mathrm{area}(s = t)}{t^2} \dd t.
\]
Using the asymptotics of $s$ near $p$ (cf.~Lemma \ref{appendix-lemma: asymptotics of greens function near singularity}), one may check that both integral terms above converge to zero as $r \to 0$.
Thus, as in Corollary \ref{cllr: ADM mass via sharp monotonic quantity}, we obtain that
\allowdisplaybreaks
\begin{align*}
    m(M, p) = - \frac{1}{4 \pi} \lim_{r \downarrow 0 } \bigg( &\frac{1}{r^3} \int_{s = r} \|\nabla s\|^2 + \int_{s \leq r} \frac{3k(2 + ks^2 - 6 r^{-2} s^2)}{s^2(4 - ks^2)^2} \|\nabla s\|^3 \\
    &+ \int_{s \leq r} \left( \frac{3k}{4} - \frac{3k(r^2 - s^2)(4 + 5k s^2)}{8r^2(4 - ks^2)} \right) \frac{\|\nabla s\|}{s^2} - \frac{4 \pi}{r} \bigg).
\end{align*}
Hence, the right-hand side above may have the same enclosed mass interpretation as for \eqref{eq: quasi-local mass using sharp monotonic quantity}, and similar corollaries hold in this case as well.

\appendix

\section{Integration over the (super)level sets and the asymptotics of Green's function} \label{appendix-sec: integration over the (super)level sets of greens function}

In this section, we present various results involving the asymptotics of Green's function near its singularity, and prove that the quantities from Theorems \ref{intro-th: sharp monotonic quantity} and \ref{intro-th: alternative sharp monotonic quantity} are locally absolutely continuous; this allows us to rigorously deduce that they are indeed monotonic.

Throughout the rest of this section, let $(M, g)$ be a closed, connected, Riemannian three-manifold and let $k > 0$; fix a point $p \in M$ and let $G$ be Green's function for the operator $L := -\Delta + 3k/4$, with singularity at $p$.
By the strong maximum principle, $G$ is positive on $M \setminus \{p\}$; denote by $s := G^{-1}$ and by $m > 0$ the minimal value of $G$.

We start by recalling the asymptotics of $G$ near its singularity $p$; a proof for the simpler asymptotics (without the zeroth-order term, but including the Hessian estimate) is provided in {\cite[Theorem A.1]{ManeaSGE}}, but the arguments from {\cite[Lemma 6.4 and Theorem 6.5]{LeeParkerYamabeProblem}} are readily adapted to apply to $G$, as $M$ has dimension three.
These sharper asymptotics will be required for the proof of rigidity of both Theorems \ref{intro-th: sharp monotonic quantity} and \ref{intro-th: alternative sharp monotonic quantity}, and for their applications (cf.~Section \ref{sec: sharp monotonic quantity applications}).

\begin{lemma} \label{appendix-lemma: asymptotics of greens function near singularity}
Near $p$, we have
\[
G = \frac{1}{\dist(p, \cdot)} + A + O(\dist(p, \cdot)), \qquad \nabla G = - \frac{1}{\dist(p, \cdot)^2} \nabla \dist(p, \cdot) + O(1)
\]
for some $A \in \R$, and there is a constant $C > 0$ so that, in a punctured neighbourhood of $p$, we have
\[
\|\!\hess{G}\| \leq \frac{C}{\dist(p, \cdot)^3}.
\]

In terms of $s \equiv G^{-1}$, near $p$, we have 
\[
s = \dist(p, \cdot) - A \, \dist(p, \cdot)^2 + O(\dist(p, \cdot)^3), \qquad \nabla s = (1 - 2A \, \dist(p, \cdot))\nabla \dist(p, \cdot) + O(\dist(p, \cdot)^2),
\]
and there is a constant $\widetilde{C} > 0$ so that, in a punctured neighbourhood of $p$, we have
\[
\|\!\hess{s}\| \leq \frac{\widetilde{C}}{\dist(p, \cdot)}.
\]
\end{lemma}

We may thus extend $s$ to a continuous function on all of $M$ by setting $s(p) := 0$; we also remark that, by construction, $s$ is smooth on $M \setminus \{p\}$.
The next lemma discusses the density of the regular values of $s$.

\begin{lemma} \label{appendix-lemma: regular values of G}
The regular values of $s$ form an open and full measure (in particular, dense) subset of $(0, 1/m)$.
\end{lemma}

\begin{proof}
Since all (super/sub)level sets of $s$ are compact, it is clear that the set of regular values of $s$ is open in $(0, 1/m)$.
Sard's theorem implies that it is also a full measure subset of $(0, 1/m)$.
\end{proof}

We will also require the following property of the set of critical points of $s$, which is shown in {\cite[Lemma B.2]{ManeaSGE}}.

\begin{lemma} \label{appendix-lemma: nodal sets have finite area}
The set of critical points of $s$ is contained in a finite union of embedded, relatively compact surfaces in $M$.
\end{lemma}

We now continue with the well-definedness and local absolute continuity of the monotonic quantities from Theorems \ref{intro-th: sharp monotonic quantity} and \ref{intro-th: alternative sharp monotonic quantity}.
The following result describes the local absolute continuity of the integral of $\|\nabla s\|^2$ along the level sets of $s$, which is present in both theorems.

\begin{lemma} \label{appendix-lemma: local AC of energy term}
The function 
\[
r \mapsto \int_{s = r} \|\nabla s\|^2
\]
is absolutely continuous on $(0, r_0]$ for all $r_0 < 1/m$, and if $r_0 \in (0, 1/m)$ is a regular value of $s$, then
\begin{equation} \label{appendix-eq: derivative of energy term}
    \left. \frac{\dd}{\dd r} \right|_{r = r_0} \int_{s = r} \|\nabla s\|^2 = \int_{s = r_0} \frac{\div(\|\nabla s\| \nabla s)}{\|\nabla s\|} = \int_{s = r_0} \left( \Delta s + \left\langle \nabla \|\nabla s\|, \frac{\nabla s}{\|\nabla s\|} \right\rangle \right).
\end{equation}
\end{lemma}

\begin{proof}
Denote by $E(r) := \int_{s = r} \|\nabla s\|^2$.
To prove the desired identity, it suffices to show that for all $r \in (0, 1/m)$, we have
\begin{equation} \label{appendix-eq: rewriting of the energy term as a divergence}
    E(r) = - \int_{s > r} \div(\|\nabla s\|\nabla s).
\end{equation}
Indeed, if \eqref{appendix-eq: rewriting of the energy term as a divergence} holds, then the desired absolute continuity of the left-hand side follows immediately, as the right-hand side satisfies this property by Lemma \ref{appendix-lemma: nodal sets have finite area}.
Moreover, to compute the derivative of $E$ at a regular value $r_0 \in (0, 1/m)$, we proceed as follows.
By Lemma \ref{appendix-lemma: regular values of G}, if $|h|$ is small enough, $r_0 + h$ is also a regular value of $s$, so \eqref{appendix-eq: rewriting of the energy term as a divergence} and the coarea formula tell us that
\[
E(r_0 + h) - E(r_0) = \int_{r_0}^{r_0 + h} \left( \int_{s = r} \frac{\div(\|\nabla s\|\nabla s)}{\|\nabla s\|} \right) \dd r.
\]
The fundamental theorem of calculus (since the integrand on the right-hand side above is continuous in $r$) implies that \eqref{appendix-eq: derivative of energy term} holds at $r_0$.

To prove \eqref{appendix-eq: rewriting of the energy term as a divergence}, one may apply the argument from {\cite[Proof of Proposition B.3]{ManeaSGE}} verbatim (with $f = \|\nabla s\|$ in the notation from {\cite{ManeaSGE}}), as in this case, $\|\nabla s\| \nabla s$ is a $C^1$ vector field on $M \setminus \{p\}$ (and the Divergence Theorem may be applied).
\end{proof}

Lemma \ref{appendix-lemma: local AC of energy term} covers the local absolute continuity of the first term in both Theorems \ref{intro-th: sharp monotonic quantity} and \ref{intro-th: alternative sharp monotonic quantity}.
This is also the only term which is integrated along the level sets of $s$; the others involve integrals over its sublevel sets.
Their local absolute continuity is more straightforward, but for completeness, we include short proofs of this as well.
The arguments will depend on the following elementary lemma.

\begin{lemma} \label{appendix-lemma: general local AC result}
Let $r_0 > 0$ and $f : [r_0, \infty) \to \R$ be a measurable function for which
\[
\int_{r_0}^\infty (t^2 - r_0^2) |f(t)| \dd t < \infty.
\]
Then, for every $r > r_0$, $f$ is integrable on $[r, \infty)$, and the function 
\[
F : (r_0, \infty) \to \R, \quad F(r) := (r^2 - r_0^2) \int_r^\infty f(t) \dd t
\]
is locally absolutely continuous, with
\[
F'(r) = 2r \int_r^\infty f(t) \dd t - (r^2 - r_0^2) f(r)
\]
for almost every $r > r_0$.
\end{lemma}

\begin{proof}
To see that $f$ is integrable on $[r, \infty)$ for $r > r_0$, we just use the estimate
\[
\int_{r}^\infty |f(t)| \dd t \leq \frac{1}{r^2 - r_0^2} \int_r^\infty (t^2 - r_0^2) |f(t)| \dd t < \infty
\]
The above, together with the Dominated Convergence Theorem, also implies that $F(r) \to 0$ as $r \downarrow r_0$.
The local absolute continuity of $F$ is also immediately implied by this, and the computation for the derivative of $F$ follows from the product rule.
\end{proof}

The above lemma allows us to conclude that the quantity from Theorem \ref{intro-th: sharp monotonic quantity} is locally absolutely continuous.
Indeed, Lemma \ref{appendix-lemma: local AC of energy term} shows that the first integral term from Theorem \ref{intro-th: sharp monotonic quantity} is locally absolutely continuous.
The second one has the same property, as, by Lemma \ref{appendix-lemma: asymptotics of greens function near singularity}, $\|\nabla s\|$ is integrable near $p$ (and thus on $M$), so the coarea formula tells us that for every $r < 1/m$,
\[
\int_{s \leq r} \|\nabla s\| = \int_0^r \mathrm{area}(s = t) \dd t,
\]
where $\mathrm{area}(s = t)$ denotes the two-dimensional Hausdorff measure of the level set $\{s = t\}$.
Thus, the left-hand side of the preceding equation is locally absolutely continuous in $r \in (0, 1/m)$, with
\begin{equation} \label{appendix-eq: derivative of second term in sharp monotonic quantity}
    \frac{\dd}{\dd r} \int_{s \leq r} \|\nabla s\| = \mathrm{area}(s = r)
\end{equation}
for almost every $r \leq 1/m$.
The local absolute continuity of the third term is recorded in the following lemma. 

\begin{lemma} \label{appendix-lemma: local AC in sharp monotonic quantity}
If there is some $r_0 \leq 1/m$ so that
\[
\int_{s \leq r_0} \frac{r_0^2 - s^2}{\|\nabla s\|} < \infty,
\]
then, for every $r \in (0, r_0)$, $\|\nabla s\|^{-1} (r^2 - s^2)$ is integrable on $\{s \leq r\}$ and locally absolutely continuous as a function of $r \in (0, r_0)$, with
\begin{equation} \label{appendix-eq: derivative of third term in sharp monotonic quantity}
    \frac{\dd}{\dd r} \int_{s \leq r} \frac{r^2 - s^2}{\|\nabla s\|} = 2r \int_0^r \left( \int_{s = t} \frac{1}{\|\nabla s\|^2} \right) \dd t
\end{equation}
for almost every $r \in (0, r_0)$.
\end{lemma}

\begin{proof}
Put $u_0 := 1/r_0$.
Since $s \equiv G^{-1}$, by the chain rule, we equivalently need to show that, for $u > u_0$, the function $\|\nabla G\|^{-1}(G^2 - u^2)$ is integrable on $\{G \geq u\}$, and locally absolutely continuous as a function of $u \in (u_0, \infty)$, with
\[
\frac{\dd}{\dd u} \int_{G \geq u} \frac{G^2 - u^2}{\|\nabla G\|} = - 2u \int_u^\infty \left( \int_{G = t} \frac{1}{\|\nabla G\|^2} \right) \dd t.
\]

By our assumption and the coarea formula, for $u \geq u_0$, we may write
\allowdisplaybreaks
\begin{align*}
    \int_{G \geq u} \frac{G^2 - u^2}{\|\nabla G\|} &= \int_u^\infty (t^2 - u^2) \left( \int_{G = t} \frac{1}{\|\nabla G\|^2} \right) \dd t \\
    &=  \int_u^\infty (t^2 - u_0^2) \left( \int_{G = t} \frac{1}{\|\nabla G\|^2} \right) \dd t - (u^2 - u_0^2) \int_u^\infty \left( \int_{G = t} \frac{1}{\|\nabla G\|^2} \right) \dd t.
\end{align*}
Both terms on the right-hand side above are finite and locally absolutely continuous as functions of $u$ by Lemma \ref{appendix-lemma: general local AC result}, and the same lemma also computes the derivative of the left-hand side.
\end{proof}

The last technical result needed for the proof of Theorem \ref{intro-th: sharp monotonic quantity} is a lemma that ensures we can integrate by parts terms involving the derivative of the area of the level sets of $s$ even though this function may not be, in general, locally absolutely continuous.
This takes the following form.

\begin{lemma} \label{appendix-lemma: local AC of area of the level sets of s}
Suppose that there are some $r_1 < r_2$ so that $\|\nabla s\|^{-1}$ is integrable on the subset $\{r_1 \leq s \leq r_2\}$.

Then, there is a locally absolutely continuous function $\widetilde{A} : (r_1, r_2) \to \R$ so that $\widetilde{A}(r) = \mathrm{area}(s = r) =: A(r)$ for every regular value $r \in (r_1, r_2)$, and for which
\begin{equation} \label{appendix-eq: derivative of area of level sets of s}
    \widetilde{A}'(r) = \int_{s = r} \frac{1}{\|\nabla s\|} \div \left( \frac{\nabla s}{\|\nabla s\|} \right) = \int_{s = r} \left( \frac{\Delta s}{\|\nabla s\|^2} - \left\langle \nabla \|\nabla s\|, \frac{\nabla s}{\|\nabla s\|^3} \right\rangle \right).
\end{equation}
for almost every $r \in (r_1, r_2)$; here, $\mathrm{area}(s = r)$ denotes the two-dimensional Hausdorff measure of the level set $\{s = r\}$.
\end{lemma}

Before we prove Lemma \ref{appendix-lemma: local AC of area of the level sets of s}, let us make some remarks.
Since $s$ is smooth on $M \setminus \{p\}$ with $\|\nabla s\| \to 1$ at $p$ by Lemma \ref{appendix-lemma: asymptotics of greens function near singularity}, the function $\|\nabla s\|$ is integrable on $M$, and so the coarea formula tells us that $A(r) \equiv \mathrm{area}(s = r)$ is a well-defined $L^1$ function for $r$ in the range of $s$.
Note that Lemma \ref{appendix-lemma: local AC of area of the level sets of s} does not claim that the function $r \mapsto \mathrm{area}(s = r)$ is (locally absolutely) continuous in $r$.
We only need a weaker statement because, in the proof of Theorem \ref{intro-th: sharp monotonic quantity} (cf.~\eqref{eq: proof of sharp monotonic quantity with s part 6}), we will integrate by parts the term $\int_u^r t A'(t) \dd t$ when $u$ and $r$ are regular values of $s$ and the $A'(t)$ term is understood as the right-hand side of \eqref{appendix-eq: derivative of area of level sets of s}; to do this rigorously, we would proceed as follows: since $A' = \widetilde{A}'$ almost everywhere, it follows that 
\begin{equation} \label{appendix-eq: how to integrate by parts the derivative of the area of the level sets of s}
    \int_u^r t A'(t) \dd t = \int_u^r t \widetilde{A}'(t) \dd t = r \widetilde{A}(r) - u \widetilde{A}(u) - \int_u^r \widetilde{A}(t) \dd t = r A(r) - u A(u) - \int_u^r A(t) \dd t,
\end{equation}
where the last equality follows because $A = \widetilde{A}$ at all regular values of $s$, which form a dense and full measure subset by Lemma \ref{appendix-lemma: regular values of G}.

We also note that the hypothesis of Lemma \ref{appendix-lemma: local AC of area of the level sets of s} is implied by that of Lemma \ref{appendix-lemma: local AC in sharp monotonic quantity} by the same argument as the one in the proof of Lemma \ref{appendix-lemma: general local AC result}.
Indeed, if there is some $r_0 > 0$ so that
\[
\int_{s \leq r_0} \frac{r_0^2 - s^2}{\|\nabla s\|} < \infty,
\]
then for every $r_1, r_2 \in (0, r_0)$ with $r_1 < r_2$, we have
\begin{equation} \label{appendix-eq: condition for local AC of area of level sets of s}
    \int_{r_1 \leq s \leq r_2} \frac{r_0^2 - s^2}{\|\nabla s\|} \geq (r_0^2 - r_2^2) \int_{r_1 \leq s \leq r_2} \frac{1}{\|\nabla s\|},
\end{equation}
and thus $\|\nabla s\|^{-1}$ is integrable on $\{r_1 \leq s \leq r_2\}$, so Lemma \ref{appendix-lemma: local AC of area of the level sets of s} applies in this case.

\begin{proof}[Proof of Lemma \ref{appendix-lemma: local AC of area of the level sets of s}]
Formula \eqref{appendix-eq: derivative of area of level sets of s} follows from the first variation of area, since $A(r)$ is smooth whenever $r$ is a regular value of $s$ and $\widetilde{A} = A$ at all such values.

Put $\nu := \nabla s/\|\nabla s\|$ and define $H := \div(\nu)$.
We claim that $H$ is integrable on $\{r_1 \leq s \leq r_2\}$ and that
\begin{equation} \label{appendix-eq: local AC for area of level sets of s equation to prove}
    A(r) - A(t) = \int_{t \leq s \leq r} H
\end{equation}
for all regular values $r, t \in (r_1, r_2)$ of $s$ with $t \leq r$.
This then immediately implies the desired result, as one could take the function
\[
\widetilde{A}(r) := A(r_0) + \int_{r_0}^r \left( \int_{s = t} \frac{H}{\|\nabla s\|} \right) \dd t
\]
where $r_0 \in (r_1, r_2)$ is a fixed regular value of $s$; the right-hand side above is locally absolutely continuous as the integral of an $L^1$ function.

As such, let us prove the above claim.
First, we remark that $H$ is defined almost everywhere on $M \setminus \{p\}$ by Lemma \ref{appendix-lemma: nodal sets have finite area}.
To see that $H$ is integrable on $(r_1, r_2)$, we may just compute
\[
H \equiv \div (\nu) = \frac{\Delta s}{\|\nabla s\|} - \frac{1}{\|\nabla s\|} \left\langle \nabla \|\nabla s\|, \nu \right\rangle = \frac{1}{\|\nabla s\|} \left( \Delta s - \frac{1}{\|\nabla s\|^2} \hess{s}(\nabla s, \nabla s) \right),
\]
where the last equality above used that $d \|\nabla s\| = \hess{s}(\nu, \cdot)$.
Now, $s$ is smooth on $M \setminus \{p\}$, so $|\Delta s|$ and $\|\!\hess{s}\|$ are bounded on $\{r_1 \leq s \leq r_2\}$.
This implies that there is a constant $C > 0$ so that $|H| \leq C/\|\nabla s\|$, and thus $H$ is integrable on $\{r_1 \leq s \leq r_2\}$.

To show \eqref{appendix-eq: local AC for area of level sets of s equation to prove}, we proceed using a regularization of $\nu$ inspired by the technique from {\cite[Section 2]{SternLevelSetScalarCurvature}}, as $H$ is only almost everywhere defined on $\{t \leq s \leq r\}$.
For $\delta > 0$, define $\nu_\delta := (\|\nabla s\|^2 + \delta)^{-1/2} \nabla s$ and $H_\delta := \div(\nu_\delta)$.
Then, both $\nu_\delta$ and $H_\delta$ are well-defined and smooth on $M \setminus \{p\}$, with $\nu_\delta \to \nu$ pointwise almost everywhere as $\delta \downarrow 0$, and we may compute (as for $\div(\nu)$)
\[
\div(\nu_\delta) = \frac{\Delta s}{\sqrt{\|\nabla s\|^2 + \delta}} - \frac{\hess{s}(\nabla s, \nabla s)}{(\|\nabla s\|^2 + \delta)^{3/2}}.
\]

Now, fix regular values $r, t \in (r_1, r_2)$ of $s$, with $t \leq r$.
The preceding computation implies that $\div(\nu_\delta) \to \div(\nu) \equiv H$ pointwise almost everywhere as $\delta \downarrow 0$ and that $|\!\div(\nu_\delta)| \leq \widetilde{C}/\|\nabla s\|$ on $\{r_1 \leq s \leq r_2\}$ for some constant $\widetilde{C} > 0$, so the Dominated Convergence Theorem tells us that
\begin{equation} \label{appendix-eq: local AC for area of level sets of s part 1}
    \lim_{\delta \downarrow 0} \int_{t \leq s \leq r} \div(\nu_\delta) = \int_{t \leq s \leq r} \div(\nu).
\end{equation}
At the same time, as both $t$ and $r$ are regular values of $s$, the Divergence Theorem allows us to write
\[
\int_{t \leq s \leq r} \div(\nu_\delta) = \int_{s = r} \langle \nu_\delta, \nu \rangle - \int_{s = t} \langle \nu_\delta, \nu \rangle,
\]
and the Dominated Convergence Theorem again implies that
\begin{equation} \label{appendix-eq: local AC for area of level sets of s part 2}
    \lim_{\delta \downarrow 0} \left( \int_{s = r} \langle \nu_\delta, \nu \rangle - \int_{s = t} \langle \nu_\delta, \nu \rangle \right) = A(r) - A(t).
\end{equation}
The proof of the desired claim is now finished by combining \eqref{appendix-eq: local AC for area of level sets of s part 1} with \eqref{appendix-eq: local AC for area of level sets of s part 2}.
\end{proof}

\begin{remark} \label{appendix-remark: integrability of the singular vector field for sharp monotonic quantity}
The same proof idea as in Lemma \ref{appendix-lemma: local AC of area of the level sets of s} (using the ``regularization'' of $\|\nabla s\|^{-1}$) also shows that, under the assumption that $\|\nabla s\|^{-1}$ is locally integrable on $M$, for all regular values $u < r$ of $s$, we have that $\div( \|\nabla s\|^{-1} \nabla (\|\nabla s\|^2 + ks^2/4))$ is integrable on the subset $\{u \leq s \leq r\}$, and the Divergence Theorem holds as
\pagebreak
\allowdisplaybreaks
\begin{align*}
    \int_{u \leq s \leq r} \div \left( \frac{1}{\|\nabla s\|} \nabla \left( \|\nabla s\|^2 + \frac{k}{4} s^2 \right) \right) &= \int_{s = r} \left\langle \nabla \left( \|\nabla s\|^2 + \frac{k}{4} s^2 \right), \frac{\nabla s}{\|\nabla s\|^2} \right\rangle \\
    &\quad \, - \int_{s = u} \left\langle \nabla \left( \|\nabla s\|^2 + \frac{k}{4} s^2 \right), \frac{\nabla s}{\|\nabla s\|^2} \right\rangle.
\end{align*}
As described in Section \ref{sec: intro}, this identity is needed in the proof of Theorem \ref{intro-th: sharp monotonic quantity}.
\end{remark}

The above concludes the technical discussion for Theorem \ref{intro-th: sharp monotonic quantity}.
In the rest of this section, we will show the well-definedness and local absolute continuity of the quantity from Theorem \ref{intro-th: alternative sharp monotonic quantity}.
Let us first rewrite the statement of this theorem in terms of $G$ and not $s$, as this is the way we will prove it in Section \ref{sec: alternative monotonic quantity proof of monotonicity}.
As $s \equiv G^{-1}$, by the chain rule, Theorem \ref{intro-th: alternative sharp monotonic quantity} is equivalently stated as the identity
\begin{equation*}
    \begin{split}
        \frac{\dd}{\dd r} \bigg( \frac{1}{r} &\int_{G = r} \|\nabla G\|^2 + \int_{G \geq r} \frac{3k(2G^2 + k - 6r^2)}{G^2(4G^2 - k)^2} \|\nabla G\|^3 \\
        &+ \int_{G \geq r} \left( \frac{3k}{4} - \frac{3k(G^2 - r^2)(4G^2 + 5k)}{8G^2(4G^2 - k)} \right) \|\nabla G\| - 4 \pi r \bigg) \leq 0,
    \end{split}
\end{equation*}
with equality holding at some $r > \sqrt{k}/2$ if and only if $(M, g)$ is isometric to the three-sphere of curvature $k$.
The first integral term from above, the one involving the level sets of $G$, is locally finite and absolutely continuous by Lemma \ref{appendix-lemma: local AC of energy term}.
Let us now prove that the two integrals over the superlevel sets of $G$ satisfy the same property.

\begin{lemma} \label{appendix-lemma: welldefinedness of alternative monotonic quantity}
Suppose, in addition, that $G \geq \sqrt{k}/2$.
Then, for all $r \geq \sqrt{k}/2$, the functions
\[
\frac{3k(2G^2 + k - 6r^2)}{G^2(4G^2 - k)^2} \|\nabla G\|^3 \text{ and } \left( \frac{3k}{4} - \frac{3k(G^2 - r^2)(4G^2 + 5k)}{8G^2(4G^2 - k)} \right) \|\nabla G\|
\]
are integrable on $\{G \geq r\}$.
\end{lemma}

\begin{proof}
Fix $r \geq \sqrt{k}/2$.
If $r > \sqrt{k}/2$, then $4G^2 - k$ is bounded away from zero on $\{G \geq r\}$.
Moreover, both functions are continuous on $\{G \geq r\}$, hence the non-integrability may only arise near the singularity $p$.
However, near $p$, by Lemma \ref{appendix-lemma: asymptotics of greens function near singularity}, we know that $G \to \infty$, and so
\[
\frac{3k(2G^2 + k - 6r^2)}{G^2(4G^2 - k)^2} \|\nabla G\|^3 \sim \frac{6k G^2}{16G^6} \|\nabla G\|^3 = \frac{3k}{8} \frac{\|\nabla G\|^3}{G^4}
\]
and
\[
\left( \frac{3k}{4} - \frac{3k(G^2 - r^2)(4G^2 + 5k)}{8G^2(4G^2 - k)} \right) \|\nabla G\| \sim \left( \frac{3k}{4} - \frac{12kG^4}{32G^4} \right) \|\nabla G\| = \frac{3k}{8} \|\nabla G\|.
\]
The same lemma also tells us that $G \sim \dist(p, \cdot)^{-1}$ and $\|\nabla G\| \sim \dist(p, \cdot)^{-2}$ near $p$, and so both $G^{-4}\|\nabla G\|^3$ and $\|\nabla G\|$ are $O(\dist(p, \cdot)^{-2})$, hence they are integrable as the dimension of $M$ is three.

Now, suppose that $r = \sqrt{k}/2$; in this case, we need to show that the functions
\[
\frac{3k(2G^2 + k - 3k/2)}{G^2(4G^2 - k)^2} \|\nabla G\|^3 \text{ and } \left( \frac{3k}{4} - \frac{3k(G^2 - k/4)(4G^2 + 5k)}{8G^2(4G^2 - k)} \right) \|\nabla G\|
\]
are integrable on $M$.
The same argument as in the paragraph above shows that these functions are integrable away from the level set $\{G = \sqrt{k}/2\}$, so this is the only place where we still have to analyse the integrability.
For the first function, note that we can rewrite it as
\[
\frac{3k(2G^2 + k - 3k/2)}{G^2(4G^2 - k)^2} \|\nabla G\|^3 = \frac{3k}{8} \frac{\|\nabla G\|^3}{G^2(G^2 - k^2/4)} = \frac{3k}{8} \frac{\|\nabla G\|^3}{G^2(G + \sqrt{k}/2) (G - \sqrt{k}/2)}.
\]
However, near the level set $\{G = \sqrt{k}/2\}$, which is necessarily a minimum set of $G$ (since, by assumption, $G \geq \sqrt{k}/2$), $G$ is a smooth function, hence $\|\!\hess{G}\|$ is bounded, and thus Taylor's theorem (along geodesics in the direction of $-\nabla G$) tells us that there is a constant $C > 0$ so that $\|\nabla G\|^2 \leq C (G - \sqrt{k}/2)$ near $\{G = \sqrt{k}/2\}$.
It then follows that, near this minimum level set of $G$, we have
\[
\frac{3k}{8} \frac{\|\nabla G\|^3}{G^2(G + \sqrt{k}/2) (G - \sqrt{k}/2)} \leq C \cdot \frac{3k}{8} \frac{\|\nabla G\|}{G^2(G + \sqrt{k}/2)},
\]
and the right-hand side is integrable as it is continuous.
For the second function, we may also rewrite
\[
\left( \frac{3k}{4} - \frac{3k(G^2 - k/4)(4G^2 + 5k)}{8G^2(4G^2 - k)} \right) \|\nabla G\| = \frac{3k (4G^2 - 5k)}{32 G^2} \|\nabla G\|,
\]
and the function also integrable, as it is continuous near the level set $\{G = \sqrt{k}/2\}$.
\end{proof}

We finish this section with the local absolute continuity of the two integral terms from Theorem \ref{intro-th: alternative sharp monotonic quantity} which involve the superlevel sets of $G$.
This will also be a consequence of Lemma \ref{appendix-lemma: general local AC result}.

\begin{lemma} \label{appendix-lemma: local AC of alternative monotonic quantity}
Suppose, in addition, that $G \geq \sqrt{k}/2$.
Then, the functions
\[
r \mapsto \int_{G \geq r} \frac{3k(2G^2 + k - 6r^2)}{G^2(4G^2 - k)^2} \|\nabla G\|^3 \text{ and } r \mapsto \int_{G \geq r} \left( \frac{3k}{4} - \frac{3k(G^2 - r^2)(4G^2 + 5k)}{8G^2(4G^2 - k)} \right) \|\nabla G\|
\]
are locally absolutely continuous for $r \in (m, \infty)$, with
\begin{equation} \label{appendix-eq: derivative of second term in alternative monotonic quantity}
    \frac{d}{dr} \int_{G \geq r} \frac{3k(2G^2 + k - 6r^2)}{G^2(4G^2 - k)^2} \|\nabla G\|^3 = \frac{3k}{r^2(4r^2 - k)} \int_{G = r} \|\nabla G\|^2 - \int_r^\infty \frac{36kr}{t^2(4t^2 - k)^2} \left( \int_{G = t} \|\nabla G\|^2 \right) \dd t
\end{equation}
and
\begin{equation} \label{appendix-eq: derivative of third term in alternative monotonic quantity}
    \frac{d}{dr} \int_{G \geq r} \left( \frac{3k}{4} - \frac{3k(G^2 - r^2)(4G^2 + 5k)}{8G^2(4G^2 - k)} \right) \|\nabla G\| = - \frac{3k}{4} \mathrm{area}(G = r) + \int_r^\infty \frac{3kr(4t^2 + 5k)}{4t^2(4t^2 - k)} \mathrm{area}(G = t) \dd t
\end{equation}
for almost every $r \geq m$, where $\mathrm{area}(G = t)$ denotes the two-dimensional Hausdorff measure of the level set $\{G = t\}$.
\end{lemma}

\begin{proof}
For ease of notation, put
\[
u_1(t, r) := \frac{3k(2t^2 + k - 6r^2)}{t^2(4t^2 - k)^2}, \qquad u_2(t, r) := \frac{3k}{4} - \frac{3k(t^2 - r^2)(4t^2 + 5k)}{8t^2(4t^2 - k)},
\]
and also denote by $E(r) := \int_{G = r} \|\nabla G||^2$ and $A(r) := \mathrm{area}(G = r)$.
By the coarea formula, we see that
\[
\int_{G \geq r} \frac{3k(2G^2 + k - 6r^2)}{G^2(4G^2 - k)^2} \|\nabla G\|^3 = \int_r^\infty u_1(t, r) E(t) \dd t
\]
and that
\[
\int_{G \geq r} \left( \frac{3k}{4} - \frac{3k(G^2 - r^2)(4G^2 + 5k)}{8G^2(4G^2 - k)} \right) \|\nabla G\| = \int_r^\infty u_2(t, r) A(t) \dd t.
\]
For the first function, a direct computation shows that we may write
\[
\int_r^\infty u_1(t, r) E(t) \dd t = \int_r^\infty u_1(t, \sqrt{k}/2) E(t) \dd t - (r^2 - k/4) \int_r^\infty \frac{18k}{t^2(4t^2 - k)^2} E(t) \dd t.
\]
The first integral on the right-hand side above is locally absolutely continuous in $r$ (as the integrand is an integrable function), and the second one is locally absolutely continuous by Lemma \ref{appendix-lemma: general local AC result}, which also directly allows us to compute its derivative.
For the second function, we may again write
\[
\int_r^\infty u_2(t, r) A(t) \dd t = \int_r^\infty u_2(t, \sqrt{k}/2) A(t) \dd t + (r^2 - k/4) \int_r^\infty \frac{3k(4t^2 + 5k)}{8t^2(4t^2 - k)} A(t) \dd t,
\]
and the same argument as before finishes the proof of this lemma.
\end{proof}
\section{Path-connectedness of the level sets of Green's function} \label{appendix-sec: connectedness of the level sets of greens functions}

This section is dedicated to proving that no topological assumption may ensure that the (regular) level sets of Green's functions investigated in this manuscript are path-connected.
The arguments that follow are not particular to dimension three, and may be readily generalized to higher dimensions.
Throughout the rest of this section, for $K > 0$, we will denote by $\S^n_K$ the round $n$-sphere of constant curvature $K$, and we also set $\S^n := \S^n_1$, the unit, round $n$-sphere.

The concrete goal of this section is to prove the following result.

\begin{theorem} \label{appendix-th: arbitrarily many connected components on every psc topology}
Let $M$ be a closed, connected three-manifold which admits metrics of positive scalar curvature, let $k > 0$, and fix a point $p \in M$.

Then, for every integer $n \geq 1$, there exists a Riemannian metric $g$ on $M$ with $\scal_g \geq 6k$ so that, if $G$ is Green's function for the operator $-\Delta_g + 3k/4$ with singularity at $p$, then $G$ has at least one level set which has at least $n$ path-connected components.
\end{theorem}

Theorem \ref{appendix-th: arbitrarily many connected components on every psc topology} will be a consequence of the Gromov-Lawson-Schoen-Yau Surgery Theorem for positive scalar curvature (cf.~{\cite{SchoenYauPSC, GromovLawsonPSC}}).
Let us shortly describe the proof.
By this Surgery Theorem, if $g$ is a Riemannian metric on $M$ with $\scal \gg 6k$ and $p \in M$, then for any point $q \neq p$ in $M$, there is a chart $U \subset M$ of $q$ with $p \notin U$ on which we may deform $g$ to a new metric $g'$ so that $g = g'$ on $M \setminus \overline{U}$, $\scal_{g'} \geq 6k$ on $M$, and $g'$ is isometric to the direct product $(0, \ell) \times \S^2_K$ near $q$ for some $\ell, K > 0$ which may be chosen arbitrarily large.
This way, we may deform $M$ into $M \# (n \cdot \S^3)$ (which does not change the topology, as the resulting manifold is diffeomorphic to $M$) with a metric $g = g_n$ for which $\scal_g \geq 6k$ and for which there are regions in $M$ which are isometric to cylinders $(0, \ell_i) \times \S^2_{K_i}$ with $\ell_i$ arbitrarily large.
Then, we will show that Green's function $G$ for the operator $-\Delta_g + 3k/4$ decays exponentially fast (with respect to $\ell_i$) on each such region, whilst being bounded below by a constant independent of each $\ell_i$ outside the gluing regions.
It then follows that the level sets $G^{-1}(t)$ for $t$ close to $0$ intersect each cylinder, and hence must have $n$ components.

The aforementioned decay estimate is described in the following proposition.

\begin{proposition} \label{appendix-prop: decay of greens function on long cylindrical tubes}
Suppose that $(X, g)$ and $(Y, h)$ are compact, connected, Riemannian three-manifolds with boundary and scalar curvatures bounded below by $6k$ for some $k > 0$.
Assume that both $\partial X$ and $\partial Y$ have collar neighbourhoods $C_X$ and $C_Y$ isometric to the direct product manifold $[0, \varepsilon) \times \S^2_K$ for some $\varepsilon > 0$ and $K \geq 3k$.

For $\ell \geq 1$, consider the Riemannian manifold $M$ obtained from the (smooth) isometric gluing 
\[
M := X \cup_{\partial X \sim \{0\} \times \S^2_K} ([0, \ell] \times \S^2_K) \cup_{\{\ell\} \times \S^2_K \sim \partial Y} Y,
\]
where $[0, \ell] \times \S^2_K$ is endowed with the direct product metric.
Fix a point $p \in \mathrm{int}(X) \setminus C_X$ and let $G = G_\ell$ be Green's function for the operator $- \Delta + 3k/4$, with singularity at $p$.

Then, there is a constant $C > 0$, dependant on $\varepsilon$, $X$, $Y$, $k$, and $K$, but independent of $\ell$, so that
\[
\sup_{\widetilde{Y}} G \leq C e^{-\ell \sqrt{3k}/2},
\]
where $\widetilde{Y}$ denotes the region inside $Y$ obtained after removing the subset $[0, \varepsilon/2) \times \S^2_K \subset C_Y$.
\end{proposition}

Using a finer analysis and separation of variables, one may obtain further control of the asymptotics of $G$ as $\ell \to \infty$ on the region $\widetilde{Y}$.
We will, however, not present these results, as they are not needed for the proof of Theorem \ref{appendix-th: arbitrarily many connected components on every psc topology}.
The motivation for Proposition \ref{appendix-prop: decay of greens function on long cylindrical tubes} and these results comes from a corresponding simple analysis on cylindrical manifolds.
We will describe it first, as it will be the prototype for the proof of this proposition.

\begin{motivatingExample}[Cylindrical manifolds] \label{motivatingExample: asymptotically cylindrical manifold}
Let $(M, g)$ be a complete Riemannian three-manifold for which there exists an open, precompact subset $U$ so that $(M \setminus U, g)$ is isometric to the direct product $[0, \infty) \times \S^2_K$, where $K \geq 3k$ and $k > 0$ is a positive constant.
This choice of scale ensures that the scalar curvature of this metric is bounded below by $6k$ on $M \setminus U$.

Fix a point $p \in U$ and denote by $L := - \Delta + 3k/4$.
We claim that $L$ admits a minimal, positive Green's function $G$ with singularity at $p$ for which there is some $C > 0$ so that $G(r, \theta) \leq C e^{-r \sqrt{3k}/2}$ for all $(r, \theta) \in [0, \infty) \times \S^2_K$.
Similar results have been proved for Green's function for the conformal Laplacian on cylinders -- manifolds of the form $\R \times N$ for some $N$ (cf.~{\cite[Proposition 5.12 and Lemmas 6.8 and 6.9]{YamabeMetricsCylindricalManifolds}}).
They almost apply to our situation as well, but for completeness, we present a simple argument proving our claim.

The proof for the existence of $G$ is fairly standard, inspired by {\cite{LiTamGreensFunction}}.
For $R > 0$, denote by $U_R := M \setminus ((R, \infty) \times \S^2_K)$, and let $G_R$ be the Dirichlet Green's function for $L$ on $U_R$ with singularity at $p$.
The strong maximum principle implies that $G_R > 0$ on $\mathrm{int}(U_R)$ and that the sequence $(G_R)_{R > 0}$ is pointwise non-decreasing.
The existence of a positive fundamental solution for $L$ (which is implied by the heat kernel representation of a fundamental solution, cf.~{\cite[Chapter 7]{GrigoryanHeatKernels}}) and the maximum principle also show that the sequence $(G_R)_{R > 0}$ is locally uniformly bounded in $R$.
Thus, the sequence $(G_R)_{R > 0}$ converges to a function $G$ on $M$ as $R \to \infty$, which is then necessarily the minimal, positive Green's function for $L$ with singularity at $p$.

Let us now prove the decay estimate for $G$ on $[0, \infty) \times \S^2_K$.
The idea is the following: if $G$ would not depend on the $\S^2_K$ variable on $[0, \infty) \times \S^2_K$, the equation $\Delta G = 3kG/4$ would reduce to a linear, second-order, ``radial'' differential equation which can be explicitly solved.
If we would also already have the ``boundary condition'' that $G \to 0$ at infinity, then we would immediately obtain that $G(r, \theta) = C e^{- \sqrt{3k}r/2}$ for all $(r, \theta) \in [0, \infty) \times \S^2$, for a fixed constant $C > 0$.
In general, we proceed by solving the corresponding ``radial'' equation on the bounded domains $[0, R] \times \S^2_K$, comparing it to $G_R$, and then letting $R \to \infty$.

Let $C > 0$ be a constant for which $G_R \leq C$ on $\{0\} \times \S^2_K$ for all $R > 0$.
Consider the family of functions
\[
f_R : [0, R] \times \S^2_K \to [0, \infty), \quad f_R(r, \theta) := \frac{C}{\sinh \left( R \sqrt{3k}/2 \right)} \sinh \left( \frac{\sqrt{3k}}{2} (R - r) \right).
\]
A direct computation shows that $L f_R = 0$, with $f_R(0, \cdot) = C$ and $f_R(R, \cdot) = 0$.
Thus, the maximum principle implies that $G_R \leq f_R$ on $[0, R] \times \S^2_K$; letting $R \to \infty$, we obtain that for all $(r, \theta) \in [0, \infty) \times \S^2_K$, we have 
\[
G(r, \theta) \leq \lim_{R \to \infty} f_R(r, \theta) = C e^{-\sqrt{3k}r/2},
\]
which finishes the proof of the desired claim.

A finer analysis (using separation of variables) may also show that $G(r, \theta)$ is, up to a constant, asymptotic (in the $C^\infty$ topology) to $e^{-r\sqrt{3k}/2}$ as $r \to \infty$ uniformly in $\theta \in \S^2_K$, but we will not prove this.
\end{motivatingExample}

The proof of Proposition \ref{appendix-prop: decay of greens function on long cylindrical tubes} is similar to the argument from the Motivating Example \ref{motivatingExample: asymptotically cylindrical manifold}.

\begin{proof}[Proof of Proposition \ref{appendix-prop: decay of greens function on long cylindrical tubes}]
By construction, $M$ has a subset isometric to the cylinder $(-\varepsilon, \ell + \varepsilon) \times \S^2_K$, where $(-\varepsilon, 0] \times \S^2_K$ and $[\ell, \ell + \varepsilon) \times \S^2_K$ are isometric to the collar neighbourhoods $C_X$ and $C_Y$ of $\partial X$ and $\partial Y$, respectively.

First, note that it suffices to show that there is a constant $C > 0$, independent of $\ell$, so that $G \leq C e^{-\ell \sqrt{3k}/2}$ on $\{\ell + \varepsilon/2 \} \times \S^2_K$.
Indeed, since $G$ is positive on $M$ (by the strong maximum principle), satisfies $\Delta G = 3kG/4$ away from its singularity $p$, and since $\{\ell + \varepsilon/2\} \times \S^2_K \simeq \partial \widetilde{Y}$, the maximum principle implies that the maximum of $G|_{\widetilde{Y}}$ is attained on $\partial \widetilde{Y}$.

Thus, we only need to estimate $G$ on $\{\ell + \varepsilon/2 \} \times \S^2_K$.
We will do this by controlling the integral of $G$ on $\{\ell + \varepsilon/2 \} \times \S^2_K$ and then using the local Harnack inequality.
Define the smooth function
\[
f : [-\varepsilon/2, \ell + \varepsilon/2] \to \R, \quad f(r) := \int_{\{r\} \times \S^2_K} G,
\]
where the above integral is taken with respect to the induced area measure.
Since $G > 0$ and $\Delta G = 3kG/4$ on $M \setminus \{p\}$, it follows that $f$ is positive and (by the Divergence Theorem) satisfies the differential equation $f'' = 3kf/4$.

Fundamentally, $f$ is just the zeroth-order mode in the expansion of $G$ given by separating variables on the direct product $(-\varepsilon, \ell + \varepsilon) \times \S^2_K$, and it will control the asymptotics of $G$ for large $\ell$ as the round sphere has positive spectral gap.
The idea behind this proof will be to determine $f$ by solving the differential equation it satisfies (as we did in the Motivating Example \ref{motivatingExample: asymptotically cylindrical manifold}).
However, unlike in that situation, where we could have assumed the ``boundary condition'' of Green's function converging to zero at infinity, we will need to use a weaker property to uniquely determine $G$: that of the monotonicity of the integral of $G$ over the subsets $\{r\} \times \S^2_K$.

Specifically, the important observation is that $f$ is strictly decreasing on $[-\varepsilon/2, \ell + \varepsilon/2]$.
Indeed, for $r \in [-\varepsilon/2, \ell + \varepsilon/2]$, if we denote by $U_r := ([r, \ell] \times \S^2_K) \cup_{\{\ell\} \times \S^2_K \sim \partial Y} Y \subset M$, then, integrating the equation $\Delta G = 3kG/4$ on $U_r$ and using the Divergence Theorem, we obtain that
\[
\int_{\partial U_r} \langle \nabla G, \nu \rangle = \frac{3k}{4} \int_{U_r} G,
\]
where $\nu$ is the outward pointing unit normal to $\partial U_r$.
By construction, $\nu = - \partial_r$, and so this identity tells us that $-f'(r) = (3k/4) \int_{U_r} G$, and this right-hand side is positive as $G > 0$.

Now, we solve the differential equation $f'' = 3kf/4$ with initial conditions $f(\ell + \varepsilon/2)$ and $f'(\ell + \varepsilon/2)$ to obtain that there are constants $A, B \in \R$ so that
\[
f(r) = A \cosh \left( (\ell + \varepsilon/2 - r) \frac{\sqrt{3k}}{2} \right) + B \sinh \left( (\ell + \varepsilon/2 - r) \frac{\sqrt{3k}}{2} \right)
\]
for all $r \in [-\varepsilon/2, \ell + \varepsilon/2]$.
The coefficients $A$ and $B$ are explicitly determined as 
\[
A = f(\ell + \varepsilon/2), \qquad B = - \frac{2f'(\ell + \varepsilon/2)}{\sqrt{3k}},
\]
and thus both are positive constants.
Then, as $f(-\varepsilon/2) = A \cosh ( (\ell + \varepsilon) \sqrt{3k}/2 ) + B \sinh ( (\ell + \varepsilon) \sqrt{3k}/2 )$, we may simply estimate
\begin{equation} \label{eq: right endpoint ell independent estimate for integral of greens function}
    f(\ell + \varepsilon/2) \equiv A \leq \frac{f(-\varepsilon/2)}{\cosh ( (\ell + \varepsilon) \sqrt{3k}/2 )} \leq 2 e^{- \varepsilon \sqrt{3k}/2} f(-\varepsilon/2) e^{-\ell \sqrt{3k}/2}.
\end{equation}

We now show that there is a constant $C > 0$, independent of $\ell$, so that 
\begin{equation} \label{appendix-eq: ell independent bound for the initial value of the integral of greens function on cylinders}
    f(-\varepsilon/2) \leq C.
\end{equation}
Since $(-\Delta + 3k/4)G = 4 \pi \delta_p$, testing the equation against the constant function one yields the identity $\int_M G = (3k)^{-1} 16\pi$.
Let $U$ be a tubular neighbourhood of $\{-\varepsilon/2\} \times \S^2_K$ in $\mathrm{int}(X)$; this is then a tubular neighbourhood of $\{-\varepsilon/2\} \times \S^2_K$ in $M$ which is independent of $\ell$.
The Harnack inequality on $U$ tells us that, after possibly shrinking $U$ (but still containing $\{-\varepsilon/2\} \times \S^2_K$), there is a constant $C > 0$ (dependent only on $U$, as the metric on $M$ outside the cylinder $[0, \ell] \times \S^2_K$ is independent of $\ell$) so that $\sup_U G \leq C \inf_U G$.
It then follows that
\[
\sup_{\{-\varepsilon/2\} \times \S^2_K} G \leq \sup_U G \leq C \inf_U G \leq \frac{C}{\vol(U)} \int_U G \leq \frac{C}{\vol(U)} \cdot \frac{16 \pi}{3k}.
\]
This then immediately implies that $f(-\varepsilon/2)$ is bounded above by a constant dependent on $U$, $K$, and $k$, but not on $\ell$.

With the above estimates, we are ready to finish the proof.
Let $V \subset M$ be a tubular neighbourhood of $\{\ell + \varepsilon/2\} \times \S^2_K$ in $M$, which is dependent on $Y$ but independent of $\ell$ (chosen in the same manner as $U$).
The Harnack inequality then implies that, after possibly shrinking $V$, there is another constant $\widetilde{C} > 0$, dependent only on $V$, so that $\sup_V G \leq \widetilde{C} \inf_V G$.
We may then simply estimate
\[
\sup_{\{\ell + \varepsilon/2\} \times \S^2_K} G \leq \sup_V G \leq \widetilde{C} \inf_V G \leq \widetilde{C} \inf_{\{\ell + \varepsilon/2\} \times \S^2_K} G \leq \frac{\widetilde{C}}{\mathrm{area}(\S^2_K)} \int_{\{\ell + \varepsilon/2\} \times \S^2_K} G \equiv \frac{\widetilde{C}}{\mathrm{area}(\S^2_K)} f(\ell + \varepsilon/2),
\]
and the proof of this proposition is completed by combining this preceding inequality with \eqref{eq: right endpoint ell independent estimate for integral of greens function} and \eqref{appendix-eq: ell independent bound for the initial value of the integral of greens function on cylinders}.
\end{proof}

With the above, we may now prove Theorem \ref{appendix-th: arbitrarily many connected components on every psc topology}.

\begin{proof}[Proof of Theorem \ref{appendix-th: arbitrarily many connected components on every psc topology}]
Let $h$ be a Riemannian metric on $M$ with scalar curvature much larger than $6k$ and let $n \geq 1$ be an arbitrary integer.
By {\cite[Theorem A]{GromovLawsonPSC}} (cf.~also {\cite{SchoenYauPSC}}), we can select $n$ pairwise distinct small, open, geodesic balls $B_1, \cdots, B_n$ which do not contain $p$, smaller open balls $V_i \subsetneqq U_i \subsetneqq B_i$, and deform the metric $h$ into a new metric on $M \setminus (V_1 \cup \cdots \cup V_n)$ that has scalar curvature bounded below by $6k$, agrees with $h$ on $M \setminus (B_1 \cup \cdots \cup B_n)$, and for which $\overline{U_i} \setminus V_i$ is isometric to the direct product manifold $[0, \varepsilon] \times \S^2_{K_i}$ for some $\varepsilon > 0$, where $K_i > 0$ may be chosen arbitrarily large.
We may then apply the same reasoning to $n$ disjoint copies of the three-sphere and isometrically (and smoothly) glue each resulting deformed sphere to $\partial (\overline{U_i} \setminus V_i)$ through a cylinder of the form $[0, \ell_i] \times \S^2_{K_i}$, where $\ell_i > 0$ may also be chosen arbitrarily large. 

The result of this gluing is the manifold $M \# (n \cdot \S^3)$ (which is diffeomorphic to $M$), equipped with a smooth Riemannian metric $g$ for which $\scal \geq 6k$ and for which there are subsets $\mathcal{C}_1, \cdots, \mathcal{C}_n \subset M$ and $\mathcal{S}_1, \cdots, \mathcal{S}_n \subset M$ so that each $\mathcal{C}_i$ is isometric to $[-\varepsilon/2, \ell_i + \varepsilon/2] \times \S^2_{K_i}$ and $\mathcal{C}_i \cup \mathcal{S}_i$ is diffeomorphic to a closed ball in $\R^3$ (as the region obtained from performing the Gromov-Lawson deformation on each three-sphere).
Now, let $G$ be Green's function for $-\Delta_g + 3k/4$ with singularity at $p$.
By Proposition \ref{appendix-prop: decay of greens function on long cylindrical tubes}, we have that 
\begin{equation} \label{appendix-eq: decay of greens function on long cylinders for path connectedness}
    \sup_{S_i} G \leq C_i e^{-\ell_i \sqrt{3k}/2},
\end{equation}
where $C_i > 0$ is a constant depending only on $\varepsilon$, $k$, each $K_i$, and the metric $h$ outside the gluing regions, but is indenpedent of $\ell_1, \cdots, \ell_n$.
Denote by $m_i := \inf_{\mathcal{C}_i} G$.

On the one hand, if $\min_i \ell_i \to \infty$, then \eqref{appendix-eq: decay of greens function on long cylinders for path connectedness} tells us that $\max_i m_i \to 0$.
On the other hand, if $U$ denotes the region in $M$ obtained after removing the (glued) subsets $((0, \ell_i + \varepsilon/2] \times \S^2_{K_i}) \cup \mathcal{S}_i$ and if $\widetilde{G}$ denotes the Dirichlet Green's function of $U$ for the operator $-\Delta + 3k/4$ (in the metric $g$, which is fixed on this region, independent of $\ell_1, \cdots, \ell_n$), the strong maximum principle tells us that $\widetilde{G} > 0$ and $G \geq \widetilde{G}$ on $\mathrm{int}(U)$.
Thus, after slightly shrinking $U$ so that $\widetilde{G} > 0$ on $\overline{U}$ and $M \setminus \overline{U} \simeq \sqcup_{i = 1}^n ([-\delta_i, \ell_i + \varepsilon/2] \times \S^2_{K_i}) \cup \mathcal{S}_i)$, for some $\delta_1, \cdots, \delta_n > 0$ small enough, there is a constant $C > 0$, independent of $\ell_1, \cdots \ell_n$, so that $G \geq \widetilde{G} \geq C$ on $U$.

The above two properties tell us that, as $\min_i \ell_i \to \infty$, the level sets $G^{-1}(t)$ for $t > 0$ close to zero are non-empty and intersect each region $\mathcal{C}_i \cup \mathcal{S}_i$ in a subset that is disjoint from the other $(\mathcal{C}_j \cup \mathcal{S}_j)$'s; in particular, the number of path-connected components of $G^{-1}(t)$ is at least $n$ when $t$ is a regular value of $G$.
\end{proof}

Thus, unlike in non-negative scalar curvature on non-compact three-manifolds (such as in {\cite{MunteanuWangComparisonPaper, BrayKazarasHarmonicFunctionsAndMass, OronzioADMMassAreaCapacity}}, no topological assumption can ensure the path-connectedness of the level sets of Green's function on its own.
At the same time, the following example shows that a geometric condition also cannot ensure this connectedness.

\begin{example}[Non-connectedness of the level sets on constant curvature spaces] \label{appendix-example: non-connectedness on level sets on constant curvature spaces}
Consider the manifold $M := \S^3/Q_8$ equipped with the quotient of the unit, round metric on $\S^3$, where we view $\S^3$ as a subset of the quaternions, and $Q_8 = \{\pm 1, \pm \mathbf{i}, \pm \mathbf{j}, \pm \mathbf{k} \} \subset \S^3$ is the quaternion group.
Fix the point $1 \in \S^3$ and let $G$ be Green's function on $M$ for the operator $-\Delta + 3/4$, with singularity at $\pi(1)$, where $\pi : \S^3 \to M$ is the quotient projection; this is just Green's function for the conformal Laplacian of $M$.
By uniqueness of Green's functions, we have
\[
G \circ \pi = \sum_{\gamma \in Q_8} \widetilde{G}_{\gamma},
\]
where $\widetilde{G}_{\gamma}$ is the Green's function for the operator $-\Delta + 3/4$ on $\S^3$ with singularity at $\gamma$.

By {\cite[Lemma 2.1]{ManeaSGE}}, for every $\gamma \in Q_8$, $\widetilde{G}_{\gamma} = (2 \sin (\dist(\gamma, \cdot)/2))^{-1}$.
In terms of the Euclidean norm $|\cdot|$ of $\R^4$, this can be written as $\widetilde{G}_{\gamma} = |\gamma - \cdot|^{-1}$, and so $G(\pi(x)) = \sum_{\gamma \in Q_8} |\gamma - x|^{-1}$.
Now, denote by $e_1 := 1, e_2 := \mathbf{i}, e_3 := \mathbf{j}$, and $e_4 := \mathbf{k}$, and write $(x_1, x_2, x_3, x_4)$ for the coordinates of a point $x \in \R^4$ in the basis $\{1, \mathbf{i}, \mathbf{j}, \mathbf{k}\}$.
A direct computation shows that, for every $x \in \S^3$, $|x - e_l|^2 = 2(1 - x_l)$ and $|x + e_l|^2 = 2(1 + x_l)$, hence $G(\pi(x)) = f(x_1) + f(x_2) + f(x_3) + f(x_4)$, where $f(t) := (2(1 - t))^{-1/2} + (2(1 + t))^{-1/2}$.
Note that $f$ is a strictly convex function, and one can then check that for every $x \in \S^3$, $G(\pi(x)) \geq 4 f(1/2)$, with equality holding if and only if $x_l = \pm 1/2$ for all $l \in \{1, 2, 3, 4\}$.
Consequently, the set of global minima of $G \circ \pi$ consists of $16$ points in $\S^3$.
Since $\#Q_8 = 8$, it follows that $G$ has exactly two global minima in $M$, and they are necessarily non-degenerate by symmetry and because $\Delta G = 3G/4 > 0$.
The Morse Lemma then implies that the level sets of $G$ close to its minimal level set $\{G = 4 f(1/2)\}$ are diffeomorphic to $\S^2 \sqcup \S^2$, and hence have exactly two path-connected components.

By a similar computation as above, the same result also holds for the level sets of Green's function of the operator $-\Delta + 3k/4$ on $M$ with singularity at $p$, for all $k \geq 1$.
\end{example}

Hence, by Theorem \ref{appendix-th: arbitrarily many connected components on every psc topology} and Example \ref{appendix-example: non-connectedness on level sets on constant curvature spaces}, one would expect that, to ensure the connectedness of the level sets of Green's function, we would need both a topological and a geometric assumption on a closed three-manifold with positive scalar curvature.
However, it is unclear what such assumptions one could take; for this reason, we decided to add the path-connectedness of the regular level sets as a hypothesis in both Theorems \ref{intro-th: sharp monotonic quantity} and \ref{intro-th: alternative sharp monotonic quantity}.
\section{Finiteness of the third integral term from Theorem \ref{intro-th: sharp monotonic quantity}} \label{appendix-sec: morse theory for greens function and integrability of third integral term}

As in the previous sections, let $(M, g)$ be a closed, connected, Riemannian three-manifold and let $k > 0$; fix a point $p \in M$ and let $G$ be Green's function for the operator $L := -\Delta + 3k/4$, with singularity at $p$.
Denote by $s := G^{-1}$ and let $m$ be the minimal value of $G$, which is positive by the strong maximum principle.

The goal of this section is to study the integrability of the function $\|\nabla s\|^{-1}(r^2 - s^2)$ over the sublevel set $\{s \leq r\}$.
In the course of the proof of Theorem \ref{intro-th: sharp monotonic quantity} in Section \ref{sec: sharp monotonic quantity proof of monotonicity}, we implicitly assumed this integrability, and in Appendix \ref{appendix-sec: integration over the (super)level sets of greens function}, we showed that, in this case, this is also locally absolutely continuous as a function of $r \in (0, 1/m)$.
However, a priori, it is not clear whether the integral of $\|\nabla s\|^{-1}(r^2 - s^2)$ over $\{s \leq r\}$ is finite for all $r \leq 1/m$.
By the asymptotics of $s$ near $p$ (cf.~Lemma \ref{appendix-lemma: asymptotics of greens function near singularity} and \eqref{align: ADM mass asymptotics of 3rd integral term in sharp monotonic quantity}), this integrability holds for $r > 0$ small enough, since, as $r \downarrow 0$,
\[
\int_{s \leq r} \frac{r^2 - s^2}{\|\nabla s\|} = O(r^5).
\]

The finiteness of this integral is related (albeit not completely) to the integrability of $\|\nabla s\|^{-1}$ (as we saw, for instance, in \eqref{appendix-eq: condition for local AC of area of level sets of s}; the finiteness of this integral implies the local integrability of $\|\nabla s\|^{-1}$ away from the minimum level set of $G$).
Imposing that $\int_M \|\nabla s\|^{-1} < \infty$ is a viable condition for Theorem \ref{intro-th: sharp monotonic quantity} to be well-posed, but it seems not to be the correct one.
Indeed, one can concretely check that, for $k = 1$, $\R \P^3$ equipped with the quotient of the unit, round metric of the three-sphere is a concrete example where $\|\nabla s\|^{-1}$ is not globally integrable, but $\int_{s \leq r} \|\nabla s\|^{-1} (r^2 - s^2) < \infty$ for all $r \in (0, 1/m]$.
The reason for this is that the set of critical points of $s$ on $\R \P^3$ is precisely equal to $\R \P^2 \subset \R \P^3$, which is a smooth, embedded surface and $\|\nabla s\|^{-1}$ is not integrable near this subset.

Thus, the finiteness of this integral is more delicate.
In the rest of this section, we will discuss certain divergence and convergence criteria for the integral of $\|\nabla s\|^{-1}(r^2 - s^2)$ over $\{s \leq r\}$.
We will see that, for instance, if Green's function is a Morse function (which should be true for generic metrics $g$ with positive scalar curvature), then this integral is always finite.
We will also provide a convergence criterion for when $(M, g)$ is a real-analytic Riemannian manifold.
The proof of all of these results will be a local analysis of $\|\nabla s\|^{-1}$ near its critical points, and essentially depend on the existence of a ``Generalized Morse Lemma''.
At the same time, it seems difficult to construct a (global) example of a closed manifold $(M, g)$ with positive scalar curvature and Green's function $G$ for which $\int_{s \leq r} \|\nabla s\|^{-1} (r^2 - s^2)$ is infinite for some $r > 0$.
Local models for which this integrability fails may be constructed, but it is not clear if they can be made global.

To keep the coming arguments in the spirit of a ``Morse-theoretic analysis'' for Green's function, we will rewrite these integrals in terms of $G$ instead of $s$.
As such, since $s \equiv G^{-1}$, the chain rule tells us that for all $u \in (0, 1/m)$, we have
\[
\int_{s \leq u} \frac{u^2 - s^2}{\|\nabla s\|} = \frac{1}{r^2} \int_{G \geq r} \frac{G^2 - r^2}{\|\nabla G\|},
\]
where $r := u^{-1}$.
Thus, we will instead study the integrability of $\|\nabla G\|^{-1} (G^2 - r^2)$ over the subset $\{G \geq r\}$ for $r \in (m, \infty)$.

All results in this section rely on the idea of stratifying the critical points of $G$ by the rank of the Hessian of $G$ and finding normal forms for $G$ in terms of this rank.
Since $\Delta G = 3kG/4$ and $G$ is positive on $M \setminus \{p\}$ by the strong maximum principle, it follows that $\hess{G}$ is non-zero at every point, \ie at every point of $M \setminus \{p\}$, $\hess{G}$ has rank at least one.
The following simple lemma shows that $\|\nabla G\|^{-1}$ is integrable near every critical point where the rank of $\hess{G}$ is at least two.
In particular, this shows that, for Morse Green's functions, the monotonic quantity from Theorem \ref{intro-th: sharp monotonic quantity} is always finite.

\begin{lemma} \label{appendix-lemma: integrability of the inverse of the gradient of G near points of hessian rank at least two}
If $q \in M \setminus \{p\}$ is a point at which $\nabla G|_q = 0$ and $\mathrm{rank} (\hess{G}|_q) \geq 2$, then $\|\nabla G\|^{-1}$ is integrable near $q$.
\end{lemma}

\begin{proof}
As $\mathrm{rank} (\hess{G}|_q) \geq 2$, we can choose an orthonormal set $\{e_1, e_2\}$ of eigenvectors of $\hess{G}|_q$ corresponding to non-zero eigenvalues of $\hess{G}$.
We may then consider two locally defined orthonormal vector fields $E_1$ and $E_2$, with $E_1|_q = e_1$ and $E_2|_q = e_2$, and define the functions $f_i := \langle \nabla G, E_i \rangle$.
Then, both $f_1$ and $f_2$ are smooth, with $(df_i)_q = \hess{G}|_q(e_i, \cdot)$.
This implies that $(df_1)_q$ and $(df_2)_q$ are linearly independent elements of $T_q^* M$.

Now, let $\lambda \in T_q^*M$ be another element so that $((df_1)_q, (df_2)_q, \lambda)$ is a basis for $T_q^*M$, and let $f_3$ be a smooth function near $q$ for which $f_3(q) = 0$ and $(df_3)_q = \lambda$.
It follows that the map $F := (f_1, f_2, f_3)$ has an invertible derivative at $q$, so the Inverse Function Theorem tells us that $F$ is a coordinate chart for $M$ near $q$.
We may then estimate, using the definition of $f_1$ and $f_2$,
\[
\|\nabla G\|^2 \geq \langle \nabla G, E_1 \rangle^2 + \langle \nabla G, E_2 \rangle^2 = f_1^2 + f_2^2,
\]
and so $\|\nabla G\|^{-1} \leq (f_1^2 + f_2^2)^{-1/2}$.
The desired conclusion now follows because the function $(x, y, z) \mapsto (x^2 + y^2)^{-1/2}$ is integrable near $0 \in \R^3$.
\end{proof}

Thus, by Lemma \ref{appendix-lemma: integrability of the inverse of the gradient of G near points of hessian rank at least two}, the non-integrability of $\|\nabla G\|^{-1}$ may appear only in the presence of critical points where $\hess{G}$ has rank equal to one.
The simplest situation where such points exist (inspired by concrete examples, such as $\R \P^3$) is when the gradient of $G$ vanishes on an embedded surface in $M \setminus \{p\}$.
For this, we have the following divergence criterion.

\begin{lemma} \label{appendix-lemma: divergence criterion for third term in sharp monotonic quantity}
For $r \geq m$, suppose that there is a smooth, connected, embedded surface $\Sigma$ in $M \setminus \{p\}$ along which $\nabla G = 0$ and $G \equiv G_0$.

If $G_0 > r$, then
\[
\int_{G \geq r} \frac{G^2 - r^2}{\|\nabla G\|} = \infty,
\]
and if $G_0 = r$, then the function $\|\nabla G\|^{-1} (G^2 - r^2)$ is locally integrable near $\Sigma$.
\end{lemma}

\begin{proof}
Note that $\Sigma$ is connected and $\nabla G = 0$ on $\Sigma$, hence $G$ is constant on $\Sigma$, equal to some $G_0 \geq m$.
This further implies that if $X$ is a tangent vector field to $\Sigma$, then $\hess{G}(X, \cdot) = 0$, and so $\hess{G}|_\Sigma = 0$.
If we denote by $\nu$ a local unit, normal vector field to $\Sigma$, since $\hess{G}$ is non-zero at every point, we must thus have $\hess{G}(\nu, \nu) = \Delta G = 3k G_0/4$.
Hence, if we choose Fermi coordinates around $\Sigma$ induced by $\nu$ and denote by $\rho$ the corresponding signed distance function to $\Sigma$, Taylor's theorem (along normal geodesics to $\Sigma$) implies that $G$ has the asymptotic expansion
\begin{equation} \label{appendix-eq: local expansion of G in terms of distance function to critical surface}
    G = G_0 + \frac{3k}{8} G_0 \rho^2 + O(\rho^3)
\end{equation}
as $\rho \to 0$.
This means that $\|\nabla G\| = 3kG_0 |\rho|/4 + O(\rho^2)$ as $\rho \to 0$, and so
\[
\frac{G^2 - r^2}{\|\nabla G\|} = \frac{4}{3k} \frac{G_0^2 - r^2}{G_0} \frac{1}{|\rho|} + O(1).
\]
The above right-hand side is integrable near $\rho = 0$ if $G_0 = r$ and not integrable if $G_0 > r$.
\end{proof}

In particular, Lemma \ref{appendix-lemma: divergence criterion for third term in sharp monotonic quantity} shows that $\|\nabla G\|^{-1}(G^2 - m^2)$ is integrable near the minimum level set $\{G = m\}$ if this level set is an embedded surface in $M$.
This is the reason the monotonic quantity from Theorem \ref{intro-th: sharp monotonic quantity} is, for instance, always finite when $k = 1$ and $(M, g) \simeq \R \P^3$; in this case, one may check explicitly that the set of critical points of $G$ is equal to $\R \P^2 \subset \R \P^3$.

The converse of Lemma \ref{appendix-lemma: divergence criterion for third term in sharp monotonic quantity} may not necessarily be true.
However, in the real-analytic category, the following result shows that the non-existence of embedded surfaces of critical points implies that $\|\nabla G\|^{-1}$ is integrable.
This result is proved by using a ``Generalized Morse Lemma'' for real-analytic functions near Hessian-rank one critical points.

\begin{lemma} \label{lemma: integrability of third term in sharp monotonic quantity when metric is analytic}
Suppose, in addition, that $(M, g)$ is a real-analytic Riemannian manifold and that $M \setminus (\{p\} \cup \{G = m\})$ does not contain smooth, embedded surfaces consisting entirely of critical points of $G$.

Then, for a critical point $q \in M \setminus (\{p\} \cup \{G = m\})$ of $G$ for which $\mathrm{rank}(\hess{G}|_q) = 1$, $\|\nabla G\|^{-1}$ is integrable near $q$.
\end{lemma}

\begin{proof}
Let $q \in M \setminus (\{p\} \cup \{G = m\})$ be a critical point of $G$ for which $\mathrm{rank}(\hess{G}|_q) = 1$.
Since $M$ and $g$ are real-analytic, it follows that $G$ is also real-analytic.
The Splitting Lemma for real-analytic functions (cf.~{\cite[Theorem 2.47]{SingularitiesAnalyticFunctions}} for the holomorphic case, whose proof readily generalizes to the real-analytic case, and cf.~{\cite{SplittingLemmaAnyCharacteristic}} for a recent proof of this lemma over a field of arbitrary characteristic) tells us that there are real-analytic coordinates $(x^1, x^2, x^3)$ centred at $q$, a constant $c \in \R \setminus \{0\}$, and a real-analytic function $h : \R^2 \to \R$ with $h(0) = G(q)$, $dh(0, 0) = 0$, and $D^2h(0, 0) = 0$ (where $D^2 h$ denotes the Euclidean Hessian of $h$ in the coordinates $(x^2, x^3)$), so that, for $x$ near $q$, we have $G(x) = c (x^1)^2 + h(x^2, x^3).$
The constant $c$ is necessarily positive because, by assumption, $\hess{G}|_q$ has two zero eigenvalues and the third one is positive since $\Delta G > 0$.
After rescaling the coordinate $x^1$, we may assume, without loss of generality, that $c = 1$, and thus $G$ has the normal form 
\[
G(x) = (x^1)^2 + h(x^2, x^3)
\]
for $x$ near $q$.

Now suppose, for the sake of contradiction, that $h$ is constant.
Then, $G$ would have the normal form $G(x) = G(q) + (x^1)^2$ for $x$ near $q$, and the embedded surface $\{x^1 = 0\} \subset M \setminus (\{p\} \cup \{G = m\})$ would consist entirely of critical points of $G$; such a surface cannot exist by the hypotheses of this lemma, and thus $h$ is not constant.
Finally, as the metric $g$ is locally bi-Lipschitz equivalent to the Euclidean metric, this normal form for $G$ further implies that there is a constant $C > 0$ so that, for $x$ near $q$, we have
\[
C^{-1} ((x^1)^2 + |dh(x^2, x^3)|^2) \leq \|\nabla G\|^2 \leq C (((x^1)^2 + |dh(x^2, x^3)|^2)),
\]
where $|\cdot|$ is the Euclidean norm on $\R^2$.
This inequality, combined with the result that $\log |u|$ is locally integrable near $0$ whenever $u$ is a non-constant, real-analytic function with $u(0) = 0$, tells us that $\|\nabla G\|^{-1}$ is integrable near $q$.
\end{proof}

Hence, under the assumption of Lemma \ref{lemma: integrability of third term in sharp monotonic quantity when metric is analytic} (together with Lemma \ref{appendix-lemma: integrability of the inverse of the gradient of G near points of hessian rank at least two}), we may conclude that, if the minimal level set $\{G = m\}$ is a smooth surface in $M$, then $\|\nabla G\|^{-1}(G^2 - r^2)$ is integrable on $\{G \geq r\}$ for all $r$ in the range of $G$, and hence the quantity from \eqref{intro-eq: sharp monotonic quantity} in Theorem \ref{intro-th: sharp monotonic quantity} is finite in this case.
However, the hypothesis that $M \setminus (\{p\} \cup \{G = m\})$ does not contain embedded surfaces of critical points of $G$ seems strong and slightly artificial; it may be possible to use another local form for $G$ to show that this condition is satisfied for real-analytic manifolds.

\bibliographystyle{alpha}
\bibliography{bibliography}

\end{document}